\documentclass[10pt,oneside,english]{amsart}
\usepackage{mathpazo}
\usepackage[T1]{fontenc}
\usepackage[latin9]{inputenc}
\usepackage{geometry}
\usepackage{babel}
\usepackage{enumitem}
\usepackage{amsbsy}
\usepackage{amstext}
\usepackage{amsthm}
\usepackage{amssymb}
\usepackage{stmaryrd}
\usepackage{graphicx}
\usepackage{setspace}
\usepackage[unicode=true,pdfusetitle,
 bookmarks=true,bookmarksnumbered=false,bookmarksopen=false,
 breaklinks=false,pdfborder={0 0 1},backref=false,colorlinks=false]
 {hyperref}

\makeatletter
\numberwithin{equation}{section}
\numberwithin{figure}{section}
\theoremstyle{plain}
\newtheorem{thm}{\protect\theoremname}[section]
\theoremstyle{plain}
\newtheorem{rem}[thm]{\protect\remarkname}
\theoremstyle{plain}
\newtheorem{lem}[thm]{\protect\lemmaname}
\theoremstyle{definition}
\newtheorem{defn}[thm]{\protect\definitionname}
\theoremstyle{plain}
\newtheorem{fact}[thm]{\protect\factname}
\theoremstyle{plain}
\newtheorem{prop}[thm]{\protect\propositionname}
\theoremstyle{remark}
\newtheorem*{claim*}{\protect\claimname}
\theoremstyle{plain}
\newtheorem{conjecture}[thm]{\protect\conjecturename}
\theoremstyle{plain}
\newtheorem{cor}[thm]{\protect\corollaryname}
\theoremstyle{definition}
\newtheorem{example}[thm]{\protect\examplename}
\theoremstyle{definition}
\newtheorem{problem}[thm]{\protect\problemname}

\makeatother

\providecommand{\claimname}{Claim}
\providecommand{\conjecturename}{Conjecture}
\providecommand{\corollaryname}{Corollary}
\providecommand{\definitionname}{Definition}
\providecommand{\examplename}{Example}
\providecommand{\factname}{Fact}
\providecommand{\lemmaname}{Lemma}
\providecommand{\problemname}{Problem}
\providecommand{\propositionname}{Proposition}
\providecommand{\remarkname}{Remark}
\providecommand{\theoremname}{Theorem}

\begin{document}
\title{Cut distance identifying graphon parameters over weak{*} limits}
\author{Martin Doležal, Jan Grebík, Jan Hladký, Israel Rocha, Václav Rozho\v{n}}
\address{\emph{Doležal:} Institute of Mathematics, Czech Academy of Sciences.
Žitná~25, 115~67, Praha 1, Czechia. With institutional support RVO:67985840\emph{}\\
\emph{Grebík: }Mathematics Institute, University of Warwick, Coventry,
CV4~7AL, UK.\emph{ Most of this work was done while the author worked
}at the Institute of Mathematics, Czech Academy of Sciences\emph{.}\\
\emph{Hladký: }Institute of Computer Science, Czech Academy of Sciences.
Pod Vodárenskou v\v{e}ží 2, 182~07, Prague, Czechia. With institutional
support RVO:67985807.\emph{ Most of this work was done while the author
worked }at TU Dresden and the Institute of Mathematics, Czech Academy
of Sciences\emph{.}\\
\emph{Rocha}: \emph{This work was done while the author worked} at
Institute of Computer Science, Czech Academy of Sciences. Pod Vodárenskou
v\v{e}ží 2, 182~07, Prague, Czechia. With institutional support RVO:67985807.\emph{}\\
\emph{Rozho\v{n}: }ETH Zürich, Switzerland. \emph{Most of this work
was done while the author worked} at the Institute of Computer Science,
Czech Academy of Sciences. With institutional support RVO:67985807.}
\thanks{\emph{Hladký} was supported by the Alexander von Humboldt Foundation\emph{.
Rocha }and \emph{Rozho\v{n}} were supported by the Czech Science Foundation,
grant number GJ16-07822Y. \emph{Doležal} and \emph{Hladký} were supported
by the Czech Science Foundation, grant number GJ18-01472Y. \emph{Grebík}
was supported by the Czech Science Foundation, grant number 17-33849L}
\email{dolezal@math.cas.cz, greboshrabos@seznam.cz, hladky@cs.cas.cz, israelrocha@gmail.com,
vaclavrozhon@gmail.com }
\begin{abstract}
The theory of graphons comes with the so-called cut norm and the derived
cut distance. The cut norm is finer than the weak{*} topology (when
considering the predual of $L^{1}$-functions). Doležal and Hladký
{[}J. Combin. Theory Ser. B 137 (2019), 232-263{]} showed, that given
a sequence of graphons, a cut distance accumulation graphon can be
pinpointed in the set of weak{*} accumulation points as a minimizer
of the entropy. Motivated by this, we study graphon parameters with
the property that their minimizers or maximizers identify cut distance
accumulation points over the set of weak{*} accumulation points. We
call such parameters \emph{cut distance identifying}.

Of particular importance are cut distance identifying parameters coming
from homomorphism densities, $t(H,\cdot)$. This concept is closely
related to the emerging field of graph norms, and the notions of the
step Sidorenko property and the step forcing property introduced by
Krá\v{l}, Martins, Pach and Wrochna {[}J. Combin. Theory Ser. A 162
(2019), 34-54{]}. We prove that a connected graph is weakly norming
if and only if it is step Sidorenko, and that if a graph is norming
then it is step forcing.

Further, we study convexity properties of cut distance identifying
graphon parameters, and find a way to identify cut distance limits
using spectra of graphons. We also show that continuous cut distance
identifying graphon parameters have the <<pumping property>>, and
thus can be used in the proof of the Frieze\textendash Kannan regularity
lemma.
\end{abstract}


\maketitle
\global\long\def\JUSTIFY#1{\mbox{\fbox{\tiny#1}}\quad}%

\global\long\def\SUPPORT{\mathrm{supp\:}}%

\global\long\def\SUPPORTPOSITIVE{\mathrm{supp}}%

\global\long\def\ESSINF{\mathrm{essinf\:}}%
\global\long\def\ESSSUP{\mathrm{esssup\:}}%

\global\long\def\DIST{d}%

\global\long\def\cutDIST{\delta_{\square}}%

\global\long\def\EXP{\mathbf{E}}%

\global\long\def\PROB{\mathbf{P}}%

\global\long\def\INT{\mathrm{INT}}%

\global\long\def\ACC{\mathbf{ACC}_{\mathrm{w}*}}%
\global\long\def\LIM{\mathbf{\mathbf{LIM}_{\mathrm{w}*}}}%

\global\long\def\WEAKCONV{\overset{\mathrm{w}^{*}}{\;\longrightarrow\;}}%

\global\long\def\NOTWEAKCONV{\overset{\mathrm{w^{*}}}{\;\not\longrightarrow\;}}%

\global\long\def\CUTNORMCONV{\overset{\|\cdot\|_{\square}}{\;\longrightarrow\;}}%

\global\long\def\CUTDISTCONV{\overset{\cutDIST}{\;\longrightarrow\;}}%

\global\long\def\LONECONV{\overset{\|\cdot\|_{1}}{\;\longrightarrow\;}}%

\global\long\def\NOTCUTNORMCONV{\overset{\|\cdot\|_{\square}}{\;\not\longrightarrow\;}}%

\global\long\def\NOTDISTCONV{\overset{\cutDIST}{\;\not\longrightarrow\;}}%

\global\long\def\NOTLONECONV{\overset{\|\cdot\|_{1}}{\;\not\longrightarrow\;}}%

\global\long\def\SO{\stackrel{S}{\preceq}}%

\global\long\def\SSO{\stackrel{S}{\prec}}%

\global\long\def\GRAPHONSPACE{\mathcal{W}_{0}}%

\global\long\def\KERNELSPACE{\mathcal{W}}%

\global\long\def\POSITIVEKERNELSPACE{\mathcal{W}^{+}}%

\tableofcontents{}

\section{Introduction\label{sec:Intro}}

The theory of graphons, initiated in~\cite{Borgs2008c,Lovasz2006}
and covered in depth in~\cite{Lovasz2012}, provides a powerful formalism
for handling large graphs that are dense, i.e., they contain a positive
proportion of all potential edges. In this paper, we study the relation
between the cut norm and the weak{*} topology on the space of graphons
through various graphon parameters. Let us give basic definitions
needed to explain our motivation and results.

We write $\GRAPHONSPACE$ for the space of all \emph{graphons}, i.e.,
all symmetric measurable functions from $\Omega^{2}$ to $[0,1]$.
Here as well as in the rest of the paper, $\Omega$ is an arbitrary
standard Borel space with an atomless probability measure $\nu$.
Given a graphon $W$ and a measure preserving bijection (\emph{m.p.b.},
for short) $\varphi:\Omega\rightarrow\Omega$, we define the \emph{version
of $W$} by 
\[
W^{\varphi}(x,y)=W(\varphi(x),\varphi(y))\;.
\]
Let us recall that the \emph{cut norm} is defined by\footnote{All the sets and functions below are tacitly assumed to be measurable.}\emph{
\begin{equation}
\left\Vert Y\right\Vert _{\square}=\sup_{S,T\subset\Omega}\left|\int_{S\times T}Y\right|\quad\text{for each \ensuremath{Y\in L^{1}(\Omega^{2})}\:.}\label{eq:defcutnorm}
\end{equation}
}Given two graphons $U$ and $W$ we define in~(\ref{eq:defcutnormdist})
their \emph{cut norm distance} and in~(\ref{eq:defcutdist}) their
\emph{cut distance}, 
\begin{align}
\DIST_{\square}\left(U,W\right): & =\left\Vert U-W\right\Vert _{\square}\;,\text{ and}\label{eq:defcutnormdist}\\
\cutDIST(U,W) & :=\inf_{\varphi:\Omega\rightarrow\Omega\,\mathrm{m.p.b.}}\DIST_{\square}\left(U,W^{\varphi}\right)\;.\label{eq:defcutdist}
\end{align}

Recall that the key property of the space $\GRAPHONSPACE$, which
makes the theory so powerful in applications in extremal graph theory,
random graphs, property testing, and other areas, is its compactness
with respect to the cut distance. This result was first proven by
Lovász and Szegedy~\cite{Lovasz2006} using the regularity lemma,\footnote{See also~\cite{Lovasz2007} and~\cite{MR3425986} for variants of
this approach.} and then by Elek and Szegedy~\cite{ElekSzegedy} using ultrafilter
techniques, by Austin~\cite{MR2426176} and Diaconis and Janson~\cite{MR2463439}
using the theory of exchangeable random graphs, and finally by Doležal
and Hladký~\cite{DH:WeakStar} and by Doležal, Grebík, Hladký, Rocha,
and Rozho\v{n}~\cite{DGHRR:ACCLIM} in a way explained below. For
our later purposes, it is more convenient to state the compactness
theorem in terms of the cut norm distance.
\begin{thm}
\label{thm:compactness}For every sequence $\Gamma_{1},\Gamma_{2},\Gamma_{3},\ldots$
of graphons there is a subsequence $\Gamma_{n_{1}},\Gamma_{n_{2}},\Gamma_{n_{3}},\ldots$,
measure preserving bijections $\pi_{n_{1}},\pi_{n_{2}},\pi_{n_{3}},\ldots:\Omega\rightarrow\Omega$
and a graphon $\Gamma$ such that $\DIST_{\square}\left(\Gamma_{n_{i}}^{\pi_{n_{i}}},\Gamma\right)\rightarrow0$. 
\end{thm}

Let us now explain the approach from~\cite{DH:WeakStar} and from~\cite{DGHRR:ACCLIM},
which is based on the weak{*} topology. Throughout the paper, we regard
graphons as functions in the Banach space $L^{\infty}(\Omega^{2})$,
to which we associate the concept of weak{*} convergence given by
its predual Banach space $L^{1}(\Omega^{2})$. Therefore a sequence
of graphons $\Gamma_{1},\Gamma_{2},\Gamma_{3},\ldots$ \emph{converges
weak{*} }to a graphon $W$ if for every  $Q\subset\Omega^{2}$ we
have $\lim_{n\rightarrow\infty}(\int_{Q}\Gamma_{n}-\int_{Q}W)=0$.
Since the sigma-algebra of measurable sets on $\Omega^{2}$ is generated
by sets of the form $S\times T$ where $S,T\subset\Omega$ we have
that this is equivalent to requiring that $\lim_{n\rightarrow\infty}(\int_{S\times T}\Gamma_{n}-\int_{S\times T}W)=0$
for each $S,T\subset\Omega$. This latter perspective on the definition
of weak{*} convergence is better suited for our purposes as we are
ranging over the same space as in~(\ref{eq:defcutnorm}). In particular,
we get that the weak{*} topology is weaker than the topology generated
by $\DIST_{\square}$, which can be viewed as a certain uniformization
of the weak{*} topology.

So, the idea in~\cite{DH:WeakStar} and~\cite{DGHRR:ACCLIM}, on
a high level, is to look at the set $\ACC\left(\Gamma_{1},\Gamma_{2},\Gamma_{3},\ldots\right)$
of all weak{*} accumulation points of sequences, 
\[
\ACC\left(\Gamma_{1},\Gamma_{2},\Gamma_{3},\ldots\right)=\bigcup_{\pi_{1}\pi_{2},\pi_{3},\ldots:\Omega\rightarrow\Omega\,\mathrm{m.p.b.}}\text{weak* accumulation points of}\,\Gamma{}_{1}^{\pi_{1}},\Gamma_{2}^{\pi_{2}},\Gamma{}_{3}^{\pi_{3}},\ldots\;,
\]
and locate in the set $\ACC\left(\Gamma_{1},\Gamma_{2},\Gamma_{3},\ldots\right)$
one graphon $\Gamma$ that is an accumulation point not only with
respect to the weak{*} topology but also with respect to the cut norm
distance. In~\cite{DH:WeakStar}, this was done by choosing $\Gamma$
as a maximizer\footnote{In fact, the supremum of $\left\{ \INT_{f}(W):W\in\ACC\left(\Gamma_{1},\Gamma_{2},\Gamma_{3},\ldots\right)\right\} $
need not be attained (see~\cite[Section 7.4]{DH:WeakStar}), so the
rigorous treatment needs to be a bit more technical. Similarly, we
simplify the presentation of the approach from~\cite{DGHRR:ACCLIM}
below. The correct way is shown in Theorems~\ref{thm:subsequencegeneral}
and~\ref{thm:subsequencemaxbest}.} of an operator $\INT_{f}(\cdot)$ on $\GRAPHONSPACE$, defined for
a continuous strictly convex function $f:[0,1]\rightarrow\mathbb{R}$
by
\begin{equation}
\INT_{f}(W):=\int_{x}\int_{y}f\left(W(x,y)\right)\;.\label{eq:FDolHl}
\end{equation}

In~\cite{DGHRR:ACCLIM}, we then approached Theorem~\ref{thm:compactness}
by more abstract means. Namely, we showed that $\Gamma$ can be chosen
as the element with the maximum <<envelope>> in $\ACC\left(\Gamma_{1},\Gamma_{2},\Gamma_{3},\ldots\right)$.
We recall the notion of envelopes in Section~\ref{subsec:EnvelopesAndStructuredness}.
For now, it suffices to say that each envelope is a subset of $L^{\infty}(\Omega^{2})$
and the notion of maximality is with respect to the set inclusion.
In particular, envelopes are not numerical quantities.

The main focus of this paper is to return to the numerical program
initiated in~\cite{DH:WeakStar}. We provide a comprehensive study
of graphon parameters where the maximization problem over $\ACC\left(\Gamma_{1},\Gamma_{2},\Gamma_{3},\ldots\right)$
pinpoints cut distance accumulation points. We call such parameters
<<cut distance identifying>>; further we call parameters satisfying
somewhat a weaker property <<cut distance compatible>> (definitions
are given in Section~\ref{subsec:BasicsOfCDIGP}). We introduce similar
but more abstract (i.e., non-numerical) notions of <<cut distance
identifying graphon orders>> and <<cut distance compatible graphon
orders>>. In Section~\ref{subsec:BasicsOfCDIGP} we sketch that
each cut distance identifying parameter/order can indeed be used to
prove Theorem~\ref{thm:compactness}. While we initially regarded
cut distance identifying/compatible graphon parameters/orders merely
as a tool to understanding the relation between the weak{*} and the
cut norm topologies, as we shall see below, it naturally led to results
regarding quasirandomness, the Regularity Lemma, and graph norms.
Let us now highlight some of these results, following the order in
the paper. In this presentation we are somewhat imprecise (in particular,
we omit various continuity assumptions) and use notions that can be
found in the main body of the text.

\subsubsection*{Relation to quasirandomness}

Recall that the Chung\textendash Graham\textendash Wilson theorem
provides several characterizations of quasirandom graph sequences.
Two of these characterizations are minimization characterizations;
dealing with the 4-cycle density and the spectrum of the adjacency
matrix, respectively. In Section~\ref{subsec:RelationToQuasirandomness}
we explain that each cut distance identifying parameter/order gives
rise to such a minimization counterpart to the Chung\textendash Graham\textendash Wilson
theorem. As we show, the 4-cycle density and the spectrum of the adjacency
matrix (or, rather of the graphon), indeed possess these stronger
properties (see Theorems~\ref{thm:norming->step} and~\ref{thm:spectrum})
and can be used as cut distance identifying parameters/orders.

\subsubsection*{Index-pumping }

Starting with Szemerédi's Regularity Lemma~\cite{Szemeredi1978},
the field of regularizations of graphs is now one of the most powerful
areas of graph theory. In the heart of proofs of these regularity
lemmas is a certain <<index-pumping>> parameter. Recall that the
most common index-pumping parameter is the <<mean-square density>>.
Sometimes, another index-pumping parameter is more convenient. For
example, Scott~\cite{Scott} used a slight modification of the mean-square
density to get a better handle on sparse graphs. Also, Gowers~\cite{MR2373376}
famously used <<octahedral densities>> for index-pumping in his
hypergraph regularity lemma; projected to the 2-uniform case (i.e.,
graphs), this would correspond to using the $C_{4}$-density for index-pumping.
In Section~\ref{subsec:IndexPumping} we show that each cut distance
identifying graphon parameter can be used for <<index-pumping>>
in the Frieze-Kannan Regularity Lemma (see Proposition~\ref{prop:pumpingCDIGP}).
Our results in particular imply that any norming graph can play the
same role (by Theorem~\ref{thm:norming->step}).

\subsubsection*{The parameter $\protect\INT_{f}(\cdot)$}

In Section~\ref{subsec:RevisingINT} we reprove the result of Doležal
and Hladký and show that the assumption of $f$ being continuous in~(\ref{eq:FDolHl})
is not really needed (see Theorem~\ref{thm:INTdiscontinuous}). This
result is a short application of our concept of so-called <<range
frequencies>> which we previously introduced in~\cite{DGHRR:ACCLIM}.
In particular, our current approach gives us a shorter proof of the
results from~\cite{DH:WeakStar}, even when the necessary theory
from~\cite{DGHRR:ACCLIM} is counted.

\subsubsection*{Convex graphon parameters}

In Section~\ref{subsec:convexity_and_structurdness} we make a connection
between graphon parameters that are convex on the space of graphons
and cut distance compatible graphon parameters (see Theorem~\ref{thm:convex_functions_are_compatible}).
A similar observation was made independently by Lee and Schülke~\cite{MR4282091}
who derived from it that certain graphs are not norming/weakly norming
(see Section~\ref{subsec:LeeSchulke}).

\subsubsection*{Spectrum}

In Section~\ref{subsec:Spectrum}, we prove that a so-called <<spectral
quasiorder>>, which we define in Section~\ref{subsec:SpectrumAndSpectralQuasiorder}
using the spectral properties of graphons, is a cut distance identifying
graphon order (see Theorem~\ref{thm:spectrum}). As was previously
mentioned, this in particular strengthens the spectral part of the
Chung\textendash Graham\textendash Wilson theorem.

\subsubsection*{Graph norms}

Last, but most importantly, in Section~\ref{subsec:SubgraphDensities}
we study cut distance identifying and cut distance compatible graphon
parameters of the form $t(H,\cdot)$, that is, densities of a fixed
graph $H$. Such parameters are central in extremal graph theory.
The famous <<Sidorenko conjecture>> (by Erd\H{o}s\textendash Simonovits~\cite{Erdos1983}
and independently by Sidorenko~\cite{Sidorenko1993}) asserts that
if $H$ is a bipartite graph and $W$ is a graphon of edge density
$p$, then $t(H,W)\ge p^{e(H)}$, and the Forcing conjecture asserts
that this inequality is strict unless $W\equiv p$. Krá\v{l}, Martins,
Pach and Wrochna~\cite{KMPW:StepSidorenko} introduced a stronger
concept. They say that $H$ has the <<step Sidorenko property>>
if for each graphon $\ensuremath{W}$ and each finite partition $\mathcal{P}$
of $\Omega$ we have $t\left(H,W\right)\ge t\left(H,W^{\Join\mathcal{P}}\right)$,
where $W^{\Join\mathcal{P}}$ is the stepping of $W$ according to
$\mathcal{P}$, that is, a graphon obtained by averaging $W$ on the
steps of $\mathcal{P}\times\mathcal{P}$. The <<step forcing property>>
can be formulated similarly. These concepts are very much related
to the main focus of our paper. Indeed, as we show in Proposition~\ref{prop:CDIPstep},
$H$ has the step Sidorenko property if and only if $t(H,\cdot)$
is cut distance compatible. An analogous equivalence between the step
forcing property and cut distance identifying parameters is the subject
of Conjecture~\ref{conj:identifying_characterization} where we expect
that $H$ has the step forcing property if and only if $t(H,\cdot)$
is cut distance identifying; let us note that the implication from
right to left is trivial.

In Theorem~\ref{thm:compatibility->norming} we prove that if for
a connected graph $H$ we have that $t(H,\cdot)$ is cut distance
compatible, then $H$ is weakly norming. This answers a question of
Krá\v{l}, Martins, Pach and Wrochna~\cite[Section 5]{KMPW:StepSidorenko}.
The opposite implication is trivial. Combined with the equivalence
between the step Sidorenko property and being cut distance compatible,
we get a characterization of connected weakly norming graphs as those
that have the step Sidorenko property.

Our another main result about graph norms, Theorem~\ref{thm:norming->step},
states that for each norming graph $H$, the graphon parameter $t(H,\cdot)$
is cut distance identifying. Thus, by the trivial direction of Conjecture~\ref{conj:identifying_characterization}
mentioned above, we in particular obtain that each norming graph $H$
has the step forcing property.

These implications are shown in Figure~\ref{fig:weakly_norming_diagram}.

\section{Preliminaries}

In this section we introduce necessary notation and work up facts
from real and functional analysis, probability theory and facts about
graphons. Among these auxiliary results, two are quite difficult,
and need a good amount of preparation. These are Proposition~\ref{prop:feelnorming}
and Lemma~\ref{lem:L1approxByVersions}. We also recall results from~\cite{DGHRR:ACCLIM}
which we build on in this paper.

\subsection{General notation and basic analysis}

We write $\stackrel{\varepsilon}{\approx}$ for equality up to $\varepsilon$.
For example, $1\stackrel{0.2}{\approx}1.1\stackrel{0.2}{\approx}1.3$.
We write $\diamond$ for the symmetric difference of two sets. We
write $P_{k}$ for a path on $k$ vertices and $C_{k}$ for a cycle
on $k$ vertices. 

If $A$ and $B$ are measure spaces then we say that a map $f:A\rightarrow B$
is an \emph{almost-bijection} if there exist measure zero sets $A_{0}\subset A$
and $B_{0}\subset B$ such that $f_{\restriction A\setminus A_{0}}$
is a bijection between $A\setminus A_{0}$ and $B\setminus B_{0}$.
Note that in~(\ref{eq:defcutdist}), we could have worked with measure
preserving almost-bijections $\varphi$ instead.

\subsubsection{Moduli of convexity\label{subsec:ModuliOfConvexity}}

We recall the notion of the modulus of convexity. Suppose that $Y$
is a linear space with a seminorm $\left\Vert \cdot\right\Vert _{Y}$.
Then the \emph{modulus of convexity of} $\left\Vert \cdot\right\Vert _{Y}$
is the function $\mathfrak{d}_{Y}:(0,+\infty)\rightarrow[0,+\infty)$
defined by 
\begin{align}
\mathfrak{d}_{Y}(\varepsilon) & :=\inf\left\{ 1-\left\Vert \frac{x+y}{2}\right\Vert _{Y}\::\:x,y\in Y,\left\Vert x-y\right\Vert _{Y}\ge\varepsilon,\left\Vert x\right\Vert _{Y}=\left\Vert y\right\Vert _{Y}=1\right\} .\label{eq:defmodulusconvexity}
\end{align}
The seminorm $\left\Vert \cdot\right\Vert _{Y}$ is said to be \emph{uniformly
convex} if $\mathfrak{d}_{Y}(\varepsilon)>0$ for each $\varepsilon>0$. 

For each $p\in\left(1,+\infty\right)$, the $L^{p}$-norm is known
to be uniformly convex; the most streamlined argument to show this
is due to Hanner~\cite{MR0077087}.
\begin{rem}
\label{rem:whataremoduliofconvexity}The modulus of convexity is a
basic parameter in the theory of Banach spaces, but let us give some
explanation for nonexperts. Any seminorm $\left\Vert \cdot\right\Vert _{Y}$
must satisfy the triangle inequality $\left\Vert x+y\right\Vert _{Y}\le\left\Vert x\right\Vert _{Y}+\left\Vert y\right\Vert _{Y}$,
and this inequality is tight; we certainly have an equality if $y$
is a nonnegative multiple of $x$. The modulus of convexity gives
us a lower bound on the gap in the triangle inequality if we are guaranteed
that $x$ and $y$ are far from being colinear.
\end{rem}

\subsection{Probability}

We write $\EXP$ and $\PROB$ for expectation and probability, respectively.
We use two concentration inequalities, which we now recall. The first
one is the Chernoff bound in the form that can be found in~\cite[Theorem A.1.16]{Alon2008}.
\begin{lem}
\label{lem:Chernoff}Suppose that $N\in\mathbb{N}$ and $X_{1},X_{2},\ldots,X_{N}$
are mutually independent random variables with $\EXP[X_{i}]=0$ and
$|X_{i}|\le1$ for each $i\in[N]$. Then for each $a>0$ we have
\[
\PROB\left[\left|\sum X_{i}\right|>a\right]<2\exp\left(\frac{-a^{2}}{2N}\right)\;.
\]
\end{lem}

Next, we give a tailored version of the Method of Bounded Differences~\cite{McDiarmid1989}.
For convenience, we give it in two versions, the former an easy consequence
of the latter. 
\begin{lem}
\label{lem:MethodBoundedDifferences}~
\begin{enumerate}[label=(\alph*)]
\item \label{enu:MCDiarmidBasic}Suppose that $r\in\mathbb{N}$, $a>0$
and $Z$ is a random variable on the probability space $[0,1]^{r}$.
Suppose that for each two vectors $\mathbf{c},\mathbf{c}'\in[0,1]^{r}$
that differ on at most one coordinate, we have $\left|Z(\mathbf{c})-Z(\mathbf{c}')\right|\le a$.
Then we have for each $d>0$ that
\[
\PROB\left[\left|Z-\EXP Z\right|>d\right]\le\exp\left(-\frac{2d^{2}}{ra^{2}}\right)\;.
\]
\item \label{enu:MCDiarmidBetter}Suppose that $r\in\mathbb{N}$, $\mathbf{a}\in\mathbb{R}^{r}$
is a non-zero vector and $Z$ is a random variable on the probability
space $[0,1]^{r}$. Suppose that for each $i\in[r]$ and for each
two vectors $\mathbf{c},\mathbf{c}'\in[0,1]^{r}$ that differ only
on the $i$-th coordinate, we have $\left|Z(\mathbf{c})-Z(\mathbf{c}')\right|\le\mathbf{a}_{i}$.
Then we have for each $d>0$ that
\[
\PROB\left[\left|Z-\EXP Z\right|>d\right]\le\exp\left(-\frac{2d^{2}}{\sum_{i}\mathbf{a}_{i}^{2}}\right)\;.
\]
\end{enumerate}
\end{lem}

\subsection{Graphons}

Our notation is mostly standard, following~\cite{Lovasz2012}. Let
us fix a standard Borel space $\Omega$ with an atomless probability
measure $\nu$. Let $\KERNELSPACE$ denote the space of \textit{kernels},
i.e. of all bounded symmetric measurable real functions defined on
$\Omega^{2}$. We always work modulo differences on null-sets. For
example, if $U,W\in\KERNELSPACE$ are such that $U\neq W$, then $\left\Vert U-W\right\Vert _{1}>0$.
We write $\GRAPHONSPACE\subset\KERNELSPACE$ for the space of all
\emph{graphons}, that is, symmetric measurable functions from $\Omega^{2}$
to $[0,1]$, and $\POSITIVEKERNELSPACE\subset\KERNELSPACE$ for the
space of all bounded symmetric measurable functions from $\Omega^{2}$
to $[0,+\infty)$. The definitions of the cut norm and cut distance
given in~(\ref{eq:defcutnormdist}) and~(\ref{eq:defcutdist}) extend
to kernels verbatim. We write $\nu^{\otimes k}$ for the product measure
on $\Omega^{k}$.

For $p\in[0,1]$, we write $\mathcal{G}_{p}=\left\{ W\in\GRAPHONSPACE:\int_{x}\int_{y}W(x,y)=p\right\} $
for all graphons with edge density $p$.
\begin{rem}
\label{rem:OtherGroundSpaces}It is a classical fact that there is
a measure preserving bijection between each two standard atomless
probability spaces. So, while most of the time we shall work with
graphons on $\Omega^{2}$, a graphon defined on a square of any other
probability space as above can be represented (even though not in
a unique way) on $\Omega^{2}$.
\end{rem}

If $W\colon\Omega^{2}\rightarrow[0,1]$ is a graphon and $\varphi,\psi$
are two measure preserving bijections of $\Omega$ then we use the
short notation $W^{\psi\varphi}$ for the graphon $W^{\psi\circ\varphi}$,
i.e. $W^{\psi\varphi}(x,y)=W(\psi(\varphi(x)),\psi(\varphi(y)))=W^{\psi}(\varphi(x),\varphi(y))=\left(W^{\psi}\right)^{\varphi}\left(x,y\right)$
for $(x,y)\in\Omega^{2}$.

The next well-known lemma says that one can define cut norm using
just disjoint sets or squares, by losing just a constant factor.
\begin{lem}[{\cite[(7.1), (7.2)]{Borgs2008c}}]
\label{lem:DisjointWitness}Let $U,V\in\GRAPHONSPACE$ be graphons
with $\left\Vert U-V\right\Vert _{\square}=\delta$. Then there exist
$A,B,X\subset\Omega$ with $A\cap B=\emptyset$ such that 
\begin{align*}
\left|\int_{A\times B}(U-V)\right| & \ge\frac{\delta}{4}\;,\text{and}\\
\left|\int_{X\times X}(U-V)\right| & \ge\frac{\delta}{2}\;.
\end{align*}
\end{lem}

\subsubsection{Subgraphons\label{subsec:Subgraphons}}

Suppose that $W\in\GRAPHONSPACE$ and $A,B\subset\Omega$ are disjoint
sets of positive measures. Then we write $W[A,B]$ for the following
bipartite graphon. The graphon is defined on $(A\cup B)^{2}$ where
$A\cup B$ is equipped by the measure $\nu$ normalized by $\frac{1}{\nu(A\cup B)}$
(so that it is a probability measure), and it is given by $W[A,B](x,y)=W\left(x,y\right)$
if $(x,y)\in A\times B\cup B\times A$, and $W[A,B](x,y)=0$ if $(x,y)\in A\times A\cup B\times B$.

Similarly, we write $W[A,A]$ for the following graphon. The graphon
is defined on $A^{2}$ where $A$ is equipped by the measure $\nu$
normalized by $\frac{1}{\nu(A)}$ (so that it is a probability measure),
and it is given by $W[A,A](x,y)=W\left(x,y\right)$ for every $(x,y)\in A^{2}$.

\subsubsection{Homomorphism densities\label{subsec:SubgraphDensitiesDef}}

As usual, given a finite graph $H$ on the vertex set $\left\{ v_{1},v_{2},\ldots,v_{n}\right\} $
and a graphon $W$, we write 
\begin{equation}
t(H,W):=\int_{x_{1}\in\Omega}\int_{x_{2}\in\Omega}\cdots\int_{x_{n}\in\Omega}\prod_{v_{i}v_{j}\in E(H)}W(x_{i},x_{j})\label{eq:defdensity}
\end{equation}
for the \emph{homomorphism density} of $H$ in $W$. Note that~(\ref{eq:defdensity})
extends to all $W\in\KERNELSPACE$. The Counting Lemma allows to bound
the difference between $t(H,W_{1})$ and $t(H,W_{2})$ in terms of
the cut distance between $W_{1}$ and $W_{2}$.
\begin{lem}[Exercise 10.27 in~\cite{Lovasz2012}]
\label{lem:countinglemma}Suppose that $H$ is a graph with $m$
edges, and $W_{1},W_{2}\in\KERNELSPACE$ satisfy $\left\Vert W_{1}\right\Vert _{\infty},\left\Vert W_{2}\right\Vert _{\infty}\le c$.
Then $\left|t(H,W_{1})-t(H,W_{2})\right|\le4m\cdot c^{m-1}\cdot\cutDIST(W_{1},W_{2})$.\footnote{Exercise 10.27 in~\cite{Lovasz2012} is stated with $\left\Vert W_{1}\right\Vert _{\infty},\left\Vert W_{2}\right\Vert _{\infty}\le1$.
To reduce our more general setting to that in \cite{Lovasz2012},
we divide the values of $W_{1}$ and of $W_{2}$ by $c$, hence decreasing
$t(H,W_{1})$ and $t(H,W_{2})$ by a factor of $c^{m}$ and $\cutDIST(W_{1},W_{2})$
by a factor of $c$.}
\end{lem}

We call the quantity $t(P_{2},W)=\int_{x}\int_{y}W(x,y)$ the \emph{edge
density of $W$}. Recall also that for $x\in\Omega$, we have the
\emph{degree of $x$} in $W$ defined as $\deg_{W}(x)=\int_{y}W\left(x,y\right)$.
Recall that measurability of $W$ gives that $\deg_{W}(x)$ exists
for almost each $x\in\Omega$. We say that $W$ is \emph{$p$-regular
}if for almost every $x\in\Omega$, $\deg_{W}(x)=p$. Note that the
notions of edge density, degree and regularity\footnote{Here, by <<regularity>> we mean all the degrees having the same
value, and not Szemerédi's concept of regularity.} extend to kernels. In particular, there exist non-trivial 0-regular
kernels (for example the difference of the constant $\frac{1}{2}$-graphon
and a complete balanced bipartite graphon).

We will need to generalize homomorphism densities to decorated graphs,
as is done in~\cite[p. 120]{Lovasz2012}. A \emph{$\mathcal{W}$-decorated
graph} is a finite simple graph $H$ on the vertex set $\left\{ v_{1},v_{2},\ldots,v_{n}\right\} $
in which each edge $v_{i}v_{j}\in E(H)$ is labelled by an element
$W_{v_{i}v_{j}}\in\mathcal{W}$. We denote such a $\mathcal{W}$-decorated
graph by $(H,w)$, where $w=\left(W_{v_{i}v_{j}}\right)_{v_{i}v_{j}\in E(H)}$.
For such a $\mathcal{W}$-decorated graph $(H,w)$ we define 
\begin{align*}
t(H,w) & =\int_{x_{1}\in\Omega}\int_{x_{2}\in\Omega}\cdots\int_{x_{n}\in\Omega}\prod_{v_{i}v_{j}\in E(H)}W_{v_{i}v_{j}}(x_{i},x_{j})\;.
\end{align*}
Analogous definitions can be formulated to introduce $\mathcal{W}_{0}$-decorated
graphs and $\mathcal{W}^{+}$-decorated graphs.

\subsubsection{Tensor product}

Finally, we will need the definition of the tensor product of two
graphons. Suppose that $U,V:\Omega^{2}\rightarrow[0,1]$ are two graphons.
We define their \emph{tensor product} as a $[0,1]$-valued function
$U\otimes V:\left(\Omega^{2}\right)^{2}\rightarrow[0,1]$ by $(U\otimes V)\left(\left(x_{1},x_{2}\right),\left(y_{1},y_{2}\right)\right)=U(x_{1},y_{1})V(x_{2},y_{2})$. 

Using Remark~\ref{rem:OtherGroundSpaces}, we can think of $U\otimes V$
as a graphon in $\GRAPHONSPACE$. Note that for every graph $H$ we
have 
\begin{equation}
t(H,U\otimes V)=\int_{\Omega{}^{2v(H)}}\prod_{v_{i}v_{j}\in E(H)}U\left(x_{i},x_{j}\right)\prod_{v_{i}v_{j}\in E(H)}V\left(x_{v(H)+i},x_{v(H)+j}\right)=t(H,U)\cdot t(H,V)\;.\label{eq:tensorproducteq}
\end{equation}
 One can deal with the generalised homomorphism density for decorations
on a fixed finite graph $H$ (where the tensor product $w_{1}\otimes w_{2}$
is defined coordinatewise) in the same way and get that 
\begin{equation}
t(H,w_{1}\otimes w_{2})=t(H,w_{1})\cdot t(H,w_{2})\;.\label{eq:tensorgeneraliseddensities}
\end{equation}

\subsubsection{Spectrum and the spectral quasiorder\label{subsec:SpectrumAndSpectralQuasiorder}}

We recall the basic spectral theory for graphons, details and proofs
can be found in~\cite[\S7.5]{Lovasz2012}. We shall work with the
real Hilbert space $L^{2}\left(\Omega\right)$, inner product on which
is denoted by $\left\langle \cdot,\cdot\right\rangle $. Given a graphon
$W:\Omega^{2}\rightarrow[0,1]$, we can associate to it an operator
$T_{W}:L^{2}\left(\Omega\right)\rightarrow L^{2}\left(\Omega\right)$
given by
\[
\left(T_{W}f\right)(x):=\int_{y}W(x,y)f(y)\;.
\]
$T_{W}$ is a Hilbert\textendash Schmidt operator, and hence has a
discrete spectrum of finitely or countably many non-zero eigenvalues
(with possible multiplicities). All these eigenvalues are real, bounded
in modulus by~1, and their only possible accumulation point is~0.
Whenever there is no danger of confusion, we do not distinguish between
the graphon $W$ and the associated operator $T_{W}$ in the following
text. For a given graphon $W$ we denote its eigenvalues, taking into
account their multiplicities, by
\begin{align*}
\lambda_{1}^{+}(W)\ge\lambda_{2}^{+}(W)\ge\lambda_{3}^{+}(W)\ge\ldots & \ge0\;,\\
\lambda_{1}^{-}(W)\le\lambda_{2}^{-}(W)\le\lambda_{3}^{-}(W)\le\ldots & \le0\;.
\end{align*}
(We pad zeros if the spectrum has only finitely many positive or negative
eigenvalues.) 

We now introduce the notion of spectral quasiorder (it seems that
this definition has not appeared in other literature). We write $W\SO U$
if $\lambda_{i}^{+}(W)\leq\lambda_{i}^{+}(U)$ and $\lambda_{i}^{-}(W)\geq\lambda_{i}^{-}(U)$
for all $i=1,2,3,\ldots$. Further we write $W\SSO U$ if $W\SO U$
and at least one of the above inequalities is strict. Then $\SO$
is a quasiorder on $\GRAPHONSPACE$, which we call the \emph{spectral
quasiorder}.

Recall that the eigenspaces are pairwise orthogonal. Recall also that
(see e.g. \cite[eq. (7.23)]{Lovasz2012})
\begin{equation}
\left\Vert W\right\Vert _{2}^{2}=\sum_{i}\lambda_{i}^{+}(W)^{2}+\sum_{i}\lambda_{i}^{-}(W)^{2}\;.\label{eq:Parseval}
\end{equation}

In Section~\ref{subsec:SubgraphDensities} we shall use the following
formula connecting eigenvalues and cycle densities. For any graphon
$W$ and for any $k\ge3$, we have by~\cite[eq. (7.22)]{Lovasz2012},
\begin{equation}
t\left(C_{k},W\right)=\sum_{i}\lambda_{i}^{+}(W)^{k}+\sum_{i}\lambda_{i}^{-}(W)^{k}\;.\label{eq:EigenCycle}
\end{equation}

\subsubsection{The stepping operator}

Suppose that \emph{$W:\Omega^{2}\rightarrow[0,1]$} is a graphon.
We say that $W$ is a\emph{ step graphon} if $W$ is constant on each
$\Omega_{i}\times\Omega_{j}$, for a suitable finite partition $\mathcal{P}$
of $\Omega$, $\mathcal{P}=\left\{ \Omega_{1},\Omega_{2},\ldots,\Omega_{k}\right\} $.

We recall the definition of the stepping operator.
\begin{defn}
Suppose that $\Gamma:\Omega^{2}\rightarrow[0,1]$ is a graphon. For
a finite partition $\mathcal{P}$ of $\Omega$, $\mathcal{P}=\left\{ \Omega_{1},\Omega_{2},\ldots,\Omega_{k}\right\} $,
we define the graphon $\Gamma^{\Join\mathcal{P}}$ by setting it on
the rectangle $\Omega_{i}\times\Omega_{j}$ to be the constant $\frac{1}{\nu^{\otimes2}(\Omega_{i}\times\Omega_{j})}\int_{\Omega_{i}}\int_{\Omega_{j}}\Gamma(x,y)$.
We allow graphons to have not well-defined values on null sets which
handles the cases $\nu(\Omega_{i})=0$ or $\nu(\Omega_{j})=0$.
\end{defn}

In~\cite{Lovasz2012}, the stepping is denoted by $\Gamma_{\mathcal{P}}$
rather than $\Gamma^{\Join\mathcal{P}}$.

We will need the following technical result which is Lemma 2.5 in~\cite{DGHRR:ACCLIM}.
\begin{lem}
\label{lem:approx}Suppose that $\Gamma:\Omega^{2}\rightarrow[0,1]$
is a graphon and $\varepsilon$ is a positive number. Then there exists
a finite partition $\mathcal{P}$ of $\Omega$ such that $\left\Vert \Gamma-\Gamma^{\Join\mathcal{P}}\right\Vert _{1}<\varepsilon$.
\end{lem}

We call \emph{$\Gamma^{\Join\mathcal{P}}$ }with properties as in
Lemma~\ref{lem:approx} an \emph{averaged $L^{1}$-approximation
of $\Gamma$ by a step-graphon for precision $\varepsilon$}.

The next easy lemma says that weak{*} convergence of graphons is preserved
under stepping.
\begin{lem}
\label{lem:weakstarstepping}Suppose that $W:\Omega^{2}\rightarrow[0,1]$
and $\left(U_{n}:\Omega^{2}\rightarrow[0,1]\right)_{n}$ are graphons
such that $U_{n}\WEAKCONV W$. Suppose that $\mathcal{R}$ is a finite
partition of $\Omega$. Then $\left(U_{n}\right)^{\Join\mathcal{R}}\WEAKCONV W^{\Join\mathcal{R}}$.
\end{lem}

\begin{proof}
We need to prove that for any $S,T\subset\Omega$ we have $\int_{S\times T}W^{\Join\mathcal{R}}=\lim_{n\rightarrow\infty}\int_{S\times T}\left(U_{n}\right)^{\Join\mathcal{R}}$.
Let $\mathcal{R}=\left\{ \Omega_{1},\Omega_{2},\ldots,\Omega_{k}\right\} $.
Then we have 
\begin{equation}
\int_{S\times T}W^{\Join\mathcal{R}}=\sum_{i=1}^{k}\sum_{j=1}^{k}\frac{\nu(S\cap\Omega_{i})\nu(T\cap\Omega_{j})}{\nu(\Omega_{i})\nu(\Omega_{j})}\int_{\Omega_{i}\times\Omega_{j}}W\;,\label{eq:UI1}
\end{equation}
and a similar decomposition formula for each $U_{n}$,
\begin{equation}
\int_{S\times T}\left(U_{n}\right)^{\Join\mathcal{R}}=\sum_{i=1}^{k}\sum_{j=1}^{k}\frac{\nu(S\cap\Omega_{i})\nu(T\cap\Omega_{j})}{\nu(\Omega_{i})\nu(\Omega_{j})}\int_{\Omega_{i}\times\Omega_{j}}U_{n}\;.\label{eq:UI2}
\end{equation}

Since $U_{n}\WEAKCONV W$, we have that for each $i$ and $j$, $\int_{\Omega_{i}\times\Omega_{j}}W=\lim_{n\rightarrow\infty}\int_{\Omega_{i}\times\Omega_{j}}U_{n}$.
Substituting this into~(\ref{eq:UI1}) and~(\ref{eq:UI2}), we get
the claim.
\end{proof}
Finally, we say that a graphon $U$ \emph{refines }a graphon $W$,
if $W$ is a step graphon and for a suitable partition $\mathcal{P}$
of $\Omega$ we have $U^{\Join\mathcal{P}}=W$.

\medskip{}

Recall that a partition $\mathcal{Q}=\left\{ Q_{1},Q_{2},\ldots,Q_{k}\right\} $
of a finite measure space $\left(\Lambda,\lambda\right)$ is an \emph{equipartition}
if $\lambda(Q_{i})=\frac{1}{k}\cdot\lambda(\Lambda)$ for each $i=1,\ldots,k$.
The last lemma of this section asserts that given a step graphon~$\Gamma^{\Join\mathcal{U}}$
such that the steps of $\mathcal{U}$ are refined by a certain finer
equipartition $\mathcal{O}$, we can group cells of $\mathcal{O}$
into a coarser partition $\mathcal{R}$ but finer than $\mathcal{U}$
in such a way that $\Gamma^{\Join\mathcal{R}}$ is close to~$\Gamma^{\Join\mathcal{U}}$
in $L^{1}$.
\begin{lem}
\label{lem:concentrationrandomgrouping}Suppose that $\Gamma:\Omega^{2}\rightarrow[0,1]$
is a graphon. Let $\mathcal{O}$ and $\mathcal{U}$ be finite partitions
of $\Omega$ such that $\mathcal{O}$ is an equipartition which refines
$\mathcal{U}$, and the number $b_{U}$ of cells of $\mathcal{O}$
in each part $U$ of $\mathcal{U}$ is a multiple of a number $t\in\mathbb{N}$.
Then for each $U\in\mathcal{U}$ there is a partition of the cells
of $\mathcal{O}$ that lie in $U$ into $d_{U}:=\frac{b_{U}}{t}$
many groups $R_{U,1},R_{U,2},\ldots,R_{U,d_{U}}$ with $t$ elements
each such that the partition $\mathcal{R}$ of $\Omega$ defined as
$\bigcup R_{U,\ell}$, where $U\in\mathcal{\mathcal{U}}$, $\ell\in[d_{U}]$,
satisfies $\left\Vert \Gamma^{\Join\mathcal{U}}-\Gamma^{\Join\mathcal{R}}\right\Vert _{1}\le2t^{-1/4}$.
\end{lem}

\begin{proof}
For $A,B\subseteq\Omega$ define $m\left(A,B\right)=\frac{1}{\lambda\left(A\right)\lambda\left(B\right)}\int_{A\times B}\Gamma$.
In the special case when $U,V\in\mathcal{U}$ we set $c_{U,V}=m\left(U,V\right)$.
We equip the set of all $t$-partitions $\mathcal{R}=\left\{ \bigcup R_{U,i}\right\} _{U\in\mathcal{U},i\in\left[d_{U}\right]}$
that respect $\mathcal{O}$ with the uniform measure. For every $U,V\in\mathcal{U}$
and $i\in\left[d_{U}\right],j\in\left[d_{V}\right]$ we define a random
variable $B_{U,V,i,j}$ as an indicator function of the event $\left|m\left(\bigcup R_{U,i},\bigcup R_{V,j}\right)-c_{U,V}\right|>t^{-\frac{1}{4}}$.
We will show that 
\begin{equation}
\EXP\left[\sum B_{U,V,i,j}\right]\le\left(\frac{\left|\mathcal{O}\right|}{t}\right)^{2}\exp\left(-\sqrt{t}\right).\label{eq:hlgr}
\end{equation}
Once we have this then we can pick $\mathcal{R}$ that satisfies $\sum B_{U,V,i,j}\left(\mathcal{R}\right)\le\left(\frac{\left|\mathcal{O}\right|}{t}\right)^{2}\exp\left(-\sqrt{t}\right)$
and we have 
\begin{align*}
\left\Vert \Gamma^{\Join\mathcal{U}}-\Gamma^{\Join\mathcal{R}}\right\Vert _{1} & =\sum\left\Vert \left(\Gamma^{\Join\mathcal{U}}-\Gamma^{\Join\mathcal{R}}\right)_{\restriction\bigcup R_{U,i}\times\bigcup R_{V,j}}\right\Vert _{1}\\
 & \le t^{-\frac{1}{4}}+\left(\frac{t}{\left|\mathcal{O}\right|}\right)^{2}\left(\frac{\left|\mathcal{O}\right|}{t}\right)^{2}\exp\left(-\sqrt{t}\right)\le2t^{-\frac{1}{4}}\;.
\end{align*}
This will finish the proof.

To show~(\ref{eq:hlgr}), it is clearly enough to show that $\EXP\left[B_{U,V,i,j}\right]\le\exp\left(-\sqrt{t}\right)$
for each $U,V\in\mathcal{U}$ and $i\in\left[d_{U}\right],j\in\left[d_{V}\right]$.
To prove that we use the method of bounded differences in three distinguished
cases. Set $I_{U}=\left\{ A\in\mathcal{O}:A\subseteq U\right\} $
for every $U\in\mathcal{U}$.

\emph{Case I: We have $U\not=V$.}\\
Consider the space $\mathcal{X}_{U,V}$ of all pairs of unions of
$t$-subsets of $I_{U}$ and $I_{V}$ endowed with the uniform measure,
i.e., elements of $\mathcal{X}_{U,V}$ are of the form $\left(\bigcup I,\bigcup J\right)$
for some $I\subseteq I_{U}$ and $J\subseteq I_{V}$ such that $\left|I\right|=\left|J\right|=t$.
It is easy to see that 
\[
\EXP\left[B_{U,V,i,j}\right]=\PROB\left[\left|m-c_{U,V}\right|>t^{-\frac{1}{4}}\right]\;,
\]
where we abuse the notation and view $m$ as a random variable on
$\mathcal{X}_{U,V}$. We can represent $\mathcal{X}_{U,V}$ as vectors
in $\left[0,1\right]^{b_{U}}\times\left[0,1\right]^{b_{V}}$, endowed
with the product Lebesgue measure, using the following procedure.
Let $v\in\left[0,1\right]^{b_{U}}$ be a random vector taken with
respect to the Lebesgue measure on $\left[0,1\right]^{b_{U}}$. Note
that almost surely, the coordinates of $v$ are pairwise distinct.
Let $I\subset I_{U}$ be the set of the $t$ indices of $v$ with
the smallest values. Note that this procedure is <<Lipschitz>> in
the sense that if $v$ and $v'$ are two vectors that differ in only
one coordinate then $I$ and the corresponding $I'$ differ in at
most one index. Similarly $J$ can be selected from an independent
random vector $w\in[0,1]^{b_{V}}$. Thus, we can think of $m$ as
a random variable on $\left[0,1\right]^{b_{U}}\times\left[0,1\right]^{b_{V}}$
with expectation $c_{U,V}$. Recall that the range of $\Gamma$ is
in $[0,1]$. Observe that if $\left(v,w\right)$ changes in at most
one index, then $m$ changes by at most $\frac{1}{b_{U}}$ or $\frac{1}{b_{V}}$.
Thus, Lemma~\ref{lem:MethodBoundedDifferences}\ref{enu:MCDiarmidBetter}
implies that
\[
\PROB\left[\left|m-c_{U,V}\right|>d\right]\le\exp\left(-\frac{2d^{2}}{b_{U}\cdot\left(\frac{1}{b_{U}}\right)^{2}+b_{V}\cdot\left(\frac{1}{b_{V}}\right)^{2}}\right)\le\exp\left(-d^{2}t\right)
\]
for every $d>0$. The statement follows by taking $d:=t^{-1/4}$.

\emph{Case II: We have $U=V$ and $i=j$.}\\
 First, note that if $b_{U}=t$, then there is nothing to prove, so
we assume that $b_{U}\ge2t$. Consider the space $\mathcal{X}_{U}$
of all unions of $t$-subsets of $I_{U}$ endowed with the uniform
measure, i.e., elements of $\mathcal{X}_{U}$ are of the form $\bigcup I$
for some $I\subseteq I_{U}$ such that $\left|I\right|=t$. Again,
it is easy to see that 
\[
\EXP\left[B_{U,U,i,i}\right]=\PROB\left[\left|m-c_{U,U}\right|>t^{-\frac{1}{4}}\right]\;,
\]
where we abuse the notation and view $m$ as a random variable on
$\mathcal{X}_{U}$. We can represent $\mathcal{X}_{U}$ as vectors
in $\left[0,1\right]^{b_{U}}$ as in Case~I and use Lemma~\ref{lem:MethodBoundedDifferences}\ref{enu:MCDiarmidBetter}
to get 
\[
\PROB\left[\left|m-c_{U,U}\right|>d\right]\le\exp\left(-\frac{2d^{2}}{b_{U}\cdot\left(\frac{2}{b_{U}}\right)^{2}}\right)\le\exp\left(-d^{2}t\right)
\]
for every $d>0$. The statement follows by taking $d:=t^{-1/4}$.

\emph{Case III: We have $U=V$ and $i\not=j$.}\\
Again, note that if $b_{U}=t$, then there is nothing to prove, so
we assume that $b_{U}\ge2t$. Consider the space $\mathcal{Y}_{U}$
of all pairs of unions of $t$-subsets of $I_{U}$ that are disjoint
endowed with the uniform measure, i.e., elements of $\mathcal{Y}_{U}$
are of the form $\left(\bigcup I,\bigcup I'\right)$ for some $I,I'\subseteq I_{U}$
such that $\left|I\right|=\left|I'\right|=t$ and $I\cap I'=\emptyset$.
Again, it is easy to see that 
\[
\EXP\left[B_{U,U,i,j}\right]=\PROB\left[\left|m-c_{U,U}\right|>t^{-\frac{1}{4}}\right]\;,
\]
where we view $m$ as a random variable on $\mathcal{Y}_{U}$. We
can represent $\mathcal{Y}_{U}$ as vectors in $\left[0,1\right]^{b_{U}}$
endowed with the Lebesgue measure. Namely, if $v\in\left[0,1\right]^{b_{U}}$,
then we set $I\subseteq I_{U}$ to be the set of the $t$ indices
of $v$ with the smallest values and $I'\subseteq I_{U}$ to be the
set of the $t$ indices of $v$ with the biggest values. Observe that
if $v$ changes in at most one index, then $m$ changes by at most
$\frac{2}{b_{U}}$. Thus, Lemma~\ref{lem:MethodBoundedDifferences}\ref{enu:MCDiarmidBetter}
implies that
\[
\PROB\left[\left|m-c_{U,V}\right|>d\right]\le\exp\left(-\frac{2d^{2}}{b_{U}\cdot\left(\frac{2}{b_{U}}\right)^{2}}\right)\le\exp\left(-d^{2}t\right)
\]
for every $d>0$. The statement follows by taking $d:=t^{-1/4}$.
\end{proof}

\subsection{Norms defined by graphs\label{subsec:GraphNorms}}

In this section we briefly recall how homomorphism densities $t(H,\cdot)$
induce norms on the space of graphons. More details can be found in~\cite[\S 14.1]{Lovasz2012}.

We now introduce the (semi)norming and weakly norming graphs and graphs
with the (weak) Hölder property, concepts first introduced in~\cite{Hat:Siderenko}.
We say that a graph $H$ is \textit{(semi)norming}\emph{,} if the
function 
\begin{equation}
\left\Vert W\right\Vert _{H}:=\left|t(H,W)\right|{}^{1/e(H)}\label{eq:defineseminorm}
\end{equation}
 is a (semi)norm on $\mathcal{W}$. This means that we require that
$\left\Vert \cdot\right\Vert _{H}$ is subadditive and homogeneous
(i.e., $\left\Vert c\cdot W\right\Vert _{H}=\left|c\right|\cdot\left\Vert W\right\Vert _{H}$
for each $c\in\mathbb{R}$), and in the case of norming graphs we
moreover assume that there does not exist a kernel $W$ that is not
identically zero, but $t(H,W)=0$\emph{. }

We list several properties of (semi)norming graphs. 
\begin{fact}
\label{fact:basicseminorming}~
\begin{enumerate}[label=(\alph*)]
\item {[}Exercise~14.7 in~\cite{Lovasz2012}{]} \label{enu:SeminormingBipartite}Each
seminorming graph is bipartite.
\item \label{enu:NoTreeNorming}No tree is norming.
\end{enumerate}
\end{fact}

\begin{proof}[Proof of Fact~\ref{fact:basicseminorming}\ref{enu:NoTreeNorming}]
Theorem~2.10(ii) from~\cite{Hat:Siderenko} implies that a tree
$F$ is not norming, unless $F$ is a star, say $K_{1,m}$. So, it
remains to argue that a star $K_{1,m}$ cannot be norming. To this
end, observe that for any 0-regular kernel $U$, we have that $\left\Vert U\right\Vert _{K_{1,m}}=0$,
while we can have $U\neq0$.\footnote{Recall our remark about 0-regular kernels from Section~\ref{subsec:SubgraphDensitiesDef}.}
\end{proof}
We say that a graph $H$ is \emph{weakly norming,} if the function
$\left\Vert W\right\Vert _{H}:=t(H,|W|)^{1/e(H)}$ is a seminorm on
$\mathcal{W}$. Note that by \cite[Exercise 14.7 (a)]{Lovasz2012},
every weakly norming graph is bipartite. It follows that in this case
the seminorm above is also a norm. Indeed, if $U\in\mathcal{W}$ is
not zero almost everywhere, then $U$ contains a Lebesgue point $(a,b)$
such that $|U(a,b)|>0$ and, denoting the bipartition of $H$ as $\left\{ u_{1},\ldots,u_{k}\right\} \sqcup\left\{ v_{1},\ldots,v_{\ell}\right\} $,
we can write 
\[
\left\Vert U\right\Vert _{H}^{e(H)}=t(H,|U|)=\int_{x_{1}}\cdots\int_{x_{k}}\int_{y_{1}}\cdots\int_{y_{\ell}}\prod_{u_{i}v_{j}\in E(H)}|U(x_{i},y_{j})|\;.
\]
Now, if all $x_{i}$'s are restricted to a sufficiently small neighborhood
of $a$ and all $y_{j}$'s are in a small neighborhood of $b$, then
the fact that $(a,b)$ is a Lebesgue point at which |$U(a,b)|>0$
tells us that $\prod_{u_{i}v_{j}\in E(H)}|U(x_{i},y_{j})|$ is positive,
say at least 50\% of the time. We conclude that $\left\Vert U\right\Vert _{H}^{e(H)}>0$.

Since the homomorphism density $t(H,\cdot)$ satisfies $\left|t(H,cW)\right|=\left|c\right|^{e(H)}\left|t(H,W)\right|$,
the only nontrivial requirement in the definition of seminorming or
weakly norming graphs, respectively, is the subadditivity of the homomorphism
density defined on the space $\mathcal{W}$ of kernels, or on the
space $\mathcal{W}_{0}$ of graphons, respectively. In other words
we ask that for each $W_{1},W_{2}\in\mathcal{W}$, or for each $W_{1},W_{2}\in\GRAPHONSPACE$,
respectively, we have 
\[
\left\Vert W_{1}+W_{2}\right\Vert _{H}\le\left\Vert W_{1}\right\Vert _{H}+\left\Vert W_{2}\right\Vert _{H}\;.
\]
Complete bipartite graphs (in particular, stars), complete balanced
bipartite graphs without a perfect matching, even cycles, and hypercubes
are some of the most prominent examples of weakly norming graphs.
All these classes fall within a much wider family of so-called <<reflection
graphs>> which were shown to be weakly norming by Conlon and Lee~\cite{MR3667583}.

A graph $H$ has the \emph{Hölder property}, if for every $\mathcal{W}$-decoration
$w=\left(W_{e}\right)_{e\text{\ensuremath{\in}}E(H)}$ of $H$ we
have 
\begin{align}
t(H,w)^{e(H)} & \le\prod_{e\in E(H)}t(H,W_{e})\;.\label{eq:def_holder}
\end{align}
The graph $H$ has the \emph{weak Hölder property}, if~(\ref{eq:def_holder})
holds for every $\mathcal{W}_{0}$-decoration $w$ of $H$.

Our next lemma says that for the weak Hölder property it is enough
to test~(\ref{eq:def_holder}) over a slightly different set of decorations
of $H$. 
\begin{lem}
\label{lem:HomogeneousWH}Suppose that $H$ is graph which satisfies~(\ref{eq:def_holder})
for every $\POSITIVEKERNELSPACE$-decoration $u=\left(U_{e}\right)_{e\text{\ensuremath{\in}}E(H)}$
with $t(H,U_{e})=1$ for every $e\in E(H)$. Then $H$ has the weak
Hölder property.
\end{lem}

\begin{proof}
Suppose that we need to check~(\ref{eq:def_holder}) for a given
$\mathcal{W}_{0}$-decoration $w=\left(W_{e}\right)_{e\text{\ensuremath{\in}}E(H)}$
(or, actually, we will suppose, somewhat more generally, that $w$
is a $\POSITIVEKERNELSPACE$-decoration). Firstly, suppose that $t(H,W_{e})>0$
for every $e\in E(H)$. Then we define a $\POSITIVEKERNELSPACE$-decoration
$u=\left(U_{e}\right)_{e\text{\ensuremath{\in}}E(H)}$ by $U_{e}:=t\left(H,W_{e}\right)^{-1/e(H)}\cdot W_{e}$.
The decoration~$u$ satisfies~(\ref{eq:def_holder}) by the assumption
of the lemma. Hence,
\[
\left(\prod_{e\in E(H)}t\left(H,W_{e}\right)^{-1/e(H)}\:\cdot\:t(H,w)\right)^{e(H)}=t(H,u)^{e(H)}\le\prod_{e\in E(H)}t(H,U_{e})=\prod_{e\in E(H)}1=1,
\]
which is indeed equivalent to~(\ref{eq:def_holder}).

Secondly, suppose that $w$ is a general $\mathcal{W}_{0}$-decoration.
For $\alpha>0$, let $w_{\alpha}=\left(W_{\alpha,e}\right)_{e\text{\ensuremath{\in}}E(H)}$
be a $\POSITIVEKERNELSPACE$-decoration where we add the constant
$\alpha$ to each component, $W_{\alpha,e}=W_{e}+\alpha$. By the
<<firstly>> part, we have $t(H,w_{\alpha})^{e(H)}\le\prod_{e\in E(H)}t(H,W_{\alpha,e})$.
Since the graphons $W_{\alpha,e}$ converge to $W_{e}$ in the cut
norm as alpha tends to $0$, and since the quantities $t(H,x)^{e(H)}$
and $\prod_{e\in E(H)}t\left(H,x_{e}\right)$ are cut norm continuous
on the space of $\POSITIVEKERNELSPACE$-decorations, we obtain the
desired $t(H,w)^{e(H)}\le\prod_{e\in E(H)}t(H,W_{e})$.
\end{proof}
Literature concerning norming and weakly norming graphs seems to be
imprecise when it comes to disconnected graphs. These impressions
penetrated into previous versions of this paper (up to version~3
at arXiv). The following recent result will help us to rescue the
situation.
\begin{fact}[\cite{GaHlLe:Disconnnected}]
\label{fact:disconnGraphNorms}~
\begin{enumerate}[label=(\alph*)]
\item \label{en:disconnectedNorming} Suppose that $G$ is a disconnected
norming graph. Then there exists a connected norming graph $F$ such
that each component of $G$ is either isomorphic to $F$ or is a vertex. 
\item \label{en:disconnectedWeaklyNorming} Suppose that $H$ is a disconnected
weakly norming graph. Then there exists a connected weakly norming
graph $J$ such that each component of $H$ is either isomorphic to
$J$ or is a vertex.
\end{enumerate}
\end{fact}

One of our main results in Section~\ref{subsec:SubgraphDensities},
Theorem~\ref{thm:compatibility->norming}, connects weakly norming
graphs with the concept of the step Sidorenko property introduced
below. To prove Theorem~\ref{thm:compatibility->norming}, we shall
need the following characterization of weakly norming graphs from~\cite{Hat:Siderenko}. 
\begin{thm}[Theorem 2.8 in \cite{Hat:Siderenko}]
\label{thm:norming_iff_holder}A graph is seminorming if and only
if it has the Hölder property. It is weakly norming if and only if
it has the weak Hölder property.
\end{thm}

Another main result in Section~\ref{subsec:SubgraphDensities}, Theorem~\ref{thm:norming->step},
connects norming graphs with the related step forcing property. To
prove Theorem~\ref{thm:norming->step}, we need Proposition~\ref{prop:feelnorming}
below. This proposition was proven in Section 5.2 of \cite{MR3667583}
in the discussion after equations (12) and (13). Formally, it was
proven there only for the case $H=C_{4}$, but the authors noted that
the same approach works in general. In particular, a straightforward
generalization of their proof yields the following proposition. 
\begin{prop}[Section 5.2 in \cite{MR3667583}]
\label{prop:feelnorming}Suppose that $H$ is a norming graph. Then
for each kernel $W\in\KERNELSPACE$ we have $t(H,W)\ge\left(\left\Vert W\right\Vert _{\square}\right)^{e(H)}$.
\end{prop}

\subsubsection{Moduli of convexity of seminorming graphs\label{subsec:ModuliOfConvexityGraphs}}

For every seminorming graph $H$, let $\left\Vert \cdot\right\Vert _{H}$
be the corresponding seminorm on $\KERNELSPACE$. Hatami determined
in Theorem~2.16 of \cite{Hat:Siderenko} the moduli of convexity
of the norms induced by connected seminorming graphs (note that Hatami
wrongly claimed his result for disconnected graphs as well, see~\cite{GaHlLe:Disconnnected}
for a discussion):
\begin{thm}
\label{thm:hatami_moduli_of_convexity}Let $m\in\mathbb{N}$ and let
$\mathfrak{d}_{m}$ be the modulus of convexity of the $L^{m}$ norm.
There exists a constant $C_{m}>0$ such that for any connected seminorming
graph $H$ with $m$ edges we have the following. The modulus $\mathfrak{d}_{H}$
of convexity of the seminorm $\left\Vert \cdot\right\Vert _{H}$ satisfies
\begin{align*}
C_{m}\mathfrak{d}_{m} & \le\mathfrak{d}_{H}\le\mathfrak{d}_{m}\:.
\end{align*}
\end{thm}

In Section~\ref{subsec:ModuliOfConvexity} we mentioned that the
$L^{p}$-norm is uniformly convex, for $p\in\left(1,+\infty\right)$.
Thus, when $H$ is a connected seminorming graph which is not an edge,
the seminorm $\left\Vert \cdot\right\Vert _{H}$ is uniformly convex.

\subsection{Topologies on $\protect\GRAPHONSPACE$}

There are several natural topologies on $\GRAPHONSPACE$ and $\KERNELSPACE$.
The $\left\Vert \cdot\right\Vert _{\infty}$ topology inherited from
the normed space $L^{\infty}(\Omega^{2})$, the $\left\Vert \cdot\right\Vert _{1}$
topology inherited from the normed space $L^{1}(\Omega^{2})$, the
topology given by the $\left\Vert \cdot\right\Vert _{\square}$ norm,
and the weak{*} topology inherited from the weak{*} topology of the
dual Banach space $L^{\infty}(\Omega^{2})$. Note that $\GRAPHONSPACE$
is closed in both $L^{1}(\Omega^{2})$ and $L^{\infty}(\Omega^{2})$.
We write $d_{1}\left(\cdot,\cdot\right)$ for the distance derived
from the $\left\Vert \cdot\right\Vert _{1}$ norm and $d_{\infty}\left(\cdot,\cdot\right)$
for the distance derived from the $\left\Vert \cdot\right\Vert _{\infty}$
norm. The weak{*} topology of the dual Banach space $L^{\infty}(\Omega^{2})$
is generated by elements of its predual $L^{1}(\Omega^{2})$. That
means that the weak{*} topology on $L^{\infty}(\Omega^{2})$ is the
smallest topology on $L^{\infty}(\Omega^{2})$ such that all functionals
of the form $g\in L^{\infty}(\Omega^{2})\mapsto\int_{\Omega^{2}}fg$,
where $f\in L^{1}(\Omega^{2})$ is fixed, are continuous. Recall that
by the Banach\textendash Alaoglu theorem, $\GRAPHONSPACE$ equipped
with the weak{*} topology is compact. Recall also that the weak{*}
topology on $\GRAPHONSPACE$ is metrizable. We shall denote by $d_{\mathrm{w}^{*}}(\cdot,\cdot)$
any metric compatible with this topology. For example, we can take
some countable family $\left\{ A_{n}\right\} _{n\in\mathbb{N}}$ of
measurable subsets of $\Omega$ which forms a dense set in the sigma-algebra
of $\Omega$, and define $d_{\mathrm{w}^{*}}\left(U,W\right):=\sum_{n,k\in\mathbb{N}}2^{-(n+k)}\left|\int_{A_{n}\times A_{k}}(U-W)\right|$.

\subsection{Envelopes, the structuredness order, and the range and degree frequencies\label{subsec:EnvelopesAndStructuredness}}

Here, we recall the key concepts from~\cite{DGHRR:ACCLIM}.

Suppose that $\Gamma_{1},\Gamma_{2},\Gamma_{3},\ldots\in\GRAPHONSPACE$
are graphons. Recall that

\[
\ACC\left(\Gamma_{1},\Gamma_{2},\Gamma_{3},\ldots\right)=\bigcup_{\pi_{1}\pi_{2},\pi_{3},\ldots:\Omega\rightarrow\Omega\,\mathrm{m.p.b.}}\text{weak* accumulation points of}\,\Gamma{}_{1}^{\pi_{1}},\Gamma_{2}^{\pi_{2}},\Gamma{}_{3}^{\pi_{3}},\ldots\;.
\]
Similarly, we define
\[
\LIM\left(\Gamma_{1},\Gamma_{2},\Gamma_{3},\ldots\right)=\bigcup_{\pi_{1}\pi_{2},\pi_{3},\ldots:\Omega\rightarrow\Omega\,\mathrm{m.p.b.}}\text{weak* limit points of}\,\Gamma{}_{1}^{\pi_{1}},\Gamma_{2}^{\pi_{2}},\Gamma{}_{3}^{\pi_{3}},\ldots\;.
\]
For every graphon $W\in\GRAPHONSPACE$ we define the set $\left\langle W\right\rangle :=\LIM\left(W,W,W,\ldots\right)$.
That is, a graphon $U\in\GRAPHONSPACE$ belongs to $\langle W\rangle$
if and only if there are measure preserving bijections $\pi_{1},\pi_{2},\pi_{3},\ldots$
of $\Omega$ such that the sequence $W^{\pi_{1}},W^{\pi_{2}},W^{\pi_{3}},\ldots$
converges to $U$ in the weak{*} topology. We call the set $\langle W\rangle$
the \emph{envelope} of $W$.

We say that a graphon $U$ is \emph{at most as structured }as a graphon\emph{
$W$} if $\left\langle U\right\rangle \subset\left\langle W\right\rangle $.
We write $U\preceq W$ in this case. We write $U\prec W$ if $U\preceq W$
but it does not hold that $W\preceq U$.
\begin{fact}[Lemma 4.2(b) in \cite{DGHRR:ACCLIM}]
\label{fact:steppinginenvelope}We have $W^{\Join\mathcal{P}}\in\langle W\rangle$
for every graphon $W$ and every finite partition $\mathcal{P}$ of
$\Omega$.
\end{fact}

The <<structurdness order>> $\preceq$ fails to be antisymmetric
on the space of graphons equipped with the cut norm. Indeed, if $A$
and $B$ are two graphons, $A=B^{\pi}$ for some measure preserving
bijection $\pi$, then we can have $A\neq B$, $A\preceq B$ and $B\preceq A$.
However, the structurdness order is a proper order on the factorspace
given by the cut distance. This is stated below.
\begin{fact}[Corollary~4.24 in \cite{DGHRR:ACCLIM}]
\label{fact:StrictlyBelowCutDistPos}Suppose that $U\preceq W$.
Then $W\preceq U$ if and only if $\cutDIST(U,W)=0$.
\end{fact}

The next fact tells us that the relation $\preceq$ is (topogically)
closed with respect to the cut distance topology.
\begin{fact}[Lemma 8.1 in \cite{DGHRR:ACCLIM}]
\label{fact:StructerdnessClosed}Suppose that $A,A_{1},A_{2},A_{3},\ldots$
and $B,B_{1},B_{2},B_{3},\ldots$ are graphons, for each $n\in\mathbb{N}$
we have $B_{n}\preceq A_{n}$, $A_{n}\CUTDISTCONV A$ and $B_{n}\CUTDISTCONV B$.
Then $B\preceq A$.
\end{fact}

It follows directly from the definition of the weak{*} topology that
the edge density of a weak{*} limit of a sequence of graphons equals
to the limit of the edge densities of the graphons in the sequence.
Thus, we obtain the following.
\begin{fact}
\label{fact:differentdensitiednotcomparable}If two graphons have
different edge densities then they are incomparable in the structuredness
order.
\end{fact}

\subsubsection{The range frequency $\boldsymbol{\Phi}_{W}$, the degree frequency
$\boldsymbol{\Upsilon}_{W}$, and the flatness order}

Given a graphon $W:\Omega^{2}\rightarrow\left[0,1\right]$, we can
define a pushforward probability measure on $\left[0,1\right]$ by
\begin{equation}
\boldsymbol{\Phi}_{W}\left(A\right):=\nu^{\otimes2}\left(W^{-1}(A)\right)\,,\label{eq:pushforwardValues}
\end{equation}
for every Borel measurable set $A\subset\left[0,1\right]$. The measure
$\boldsymbol{\Phi}_{W}$ gives us the distribution of the values of
$W$. In~\cite{DGHRR:ACCLIM}, $\boldsymbol{\Phi}_{W}$ is called
the \emph{range frequencies of $W$}. Similarly, we can take the pushforward
measure of the degrees, which is called the \emph{degree frequencies
of $W$},
\begin{equation}
\boldsymbol{\Upsilon}_{W}\left(A\right):=\nu\left(\deg_{W}^{-1}\left(A\right)\right)\;,\label{eq:PushforwardDegrees}
\end{equation}
for every Borel measurable set $A\subset\left[0,1\right]$. The measures
$\boldsymbol{\Phi}_{W}$ and $\boldsymbol{\Upsilon}_{W}$ provide
substantial information about the graphon $W$. It is therefore natural
to ask how these measures relate with respect to the structuredness
order. To this end the following <<flatness relation>> on measures
is introduced.
\begin{defn}
\label{def:flatter}Suppose that $\Lambda_{1}$ and $\Lambda_{2}$
are two finite measures on $\left[0,1\right]$. We say that $\Lambda_{1}$
is \emph{at least as flat} as $\Lambda_{2}$ if there exists a finite
measure $\Lambda$ on $\left[0,1\right]^{2}$ such that $\Lambda_{1}$
is the marginal of $\Lambda$ on the first coordinate, $\Lambda_{2}$
is the marginal of $\Lambda$ on the second coordinate, and for each
$D\subset\left[0,1\right]$ we have
\begin{equation}
\int_{\left(x,y\right)\in D\times\left[0,1\right]}x\;\mathrm{d}\Lambda=\int_{\left(x,y\right)\in D\times\left[0,1\right]}y\;\mathrm{d}\Lambda\;.\label{eq:masspreserve}
\end{equation}
We say that $\Lambda_{1}$ is \emph{strictly flatter} than $\Lambda_{2}$
if $\Lambda_{1}$ is at least as flat as $\Lambda_{2}$ and $\Lambda_{1}\neq\Lambda_{2}$.
\end{defn}

\begin{figure}
\includegraphics[scale=0.9]{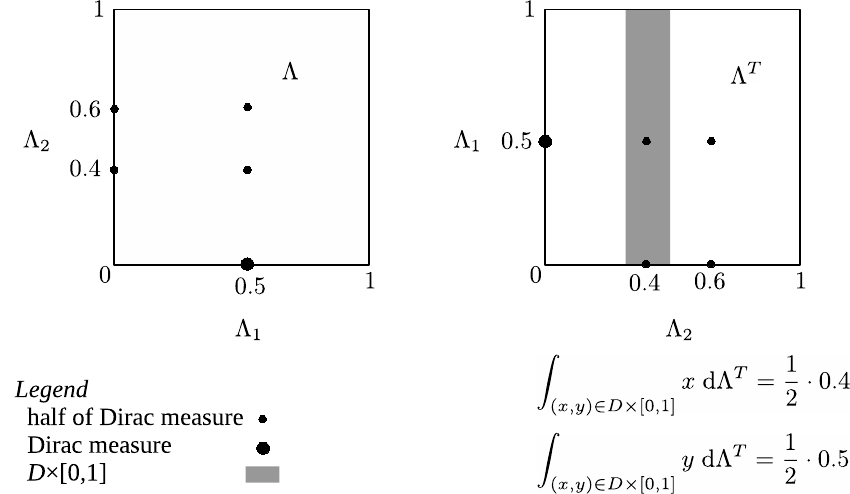}

\caption{An illustration of the concept of flatness order. Let $\Lambda_{1}$
be the Dirac measure on $0.5$. Let $\Lambda_{2}$ be a half of the
sum of the Dirac measures on $0.4$ and on $0.6$. The left-hand side
shows a measure $\Lambda$ which witnesses that $\Lambda_{1}$ is
at least as flat as $\Lambda_{2}$. The right-hand side gives an example
of a set $D$ which proves that the transposed measure $\Lambda^{T}$
fails to witness that $\Lambda_{2}$ is at least as flat as $\Lambda_{1}$
(and indeed no such witness exists). }
\label{fig:flatExample}
\end{figure}

An illustration of the concept of flatness order is given in Figure~\ref{fig:flatExample}.
We can now state the main result of Section~4.4 of~\cite{DGHRR:ACCLIM}.
\begin{prop}
\label{prop:flatter}Suppose that we have two graphons $U\preceq W$.
Then the measure $\boldsymbol{\Phi}_{U}$ is at least as flat as the
measure $\boldsymbol{\Phi}_{W}$. Similarly, the measure $\boldsymbol{\Upsilon}_{U}$
is at least as flat as the measure $\boldsymbol{\Upsilon}_{W}$. Lastly,
if $U\prec W$ then $\boldsymbol{\Phi}_{U}$ is strictly flatter than
$\boldsymbol{\Phi}_{W}$.
\end{prop}

\subsection{Approximating a graphon by versions of a more structured graphon}

In this section we state and prove Lemma~\ref{lem:L1approxByVersions},
which is the key technical step for one of our main results, Theorem~\ref{thm:norming->step}.
Since the proof of Lemma~\ref{lem:L1approxByVersions} is quite complex,
we first state a simplified version in Lemma~\ref{lem:baby}. We
use this simplification to explain some key features of the proof
and motivate some further notation.
\begin{lem}[Simplified version of Proposition~\ref{prop:RandomSEQ}]
\label{lem:baby}Suppose that $U,V\in\GRAPHONSPACE$ and that $\mathcal{R}$
is a finite partition of $\Omega$ such that $V=U^{\Join\mathcal{R}}$.
Then for each $\varepsilon>0$ we can find a number $N$ and measure
preserving bijections $\left(\phi_{i}:\Omega\rightarrow\Omega\right)_{i=1}^{N}$
such that $\left\Vert V-\frac{\sum_{i=1}^{N}U^{\phi_{i}}}{N}\right\Vert _{1}<\varepsilon$.
\end{lem}

While there are several possible proofs, the one which we need (and
which we extend to prove Lemma~\ref{lem:L1approxByVersions}) uses
the probabilistic method. Let us sketch it now. Let $\mathcal{R}=\left\{ \Omega_{1},\Omega_{2},\ldots,\Omega_{k}\right\} $.
Suppose that we are given $\varepsilon$. We now take a large number
$s$, and $N\gg s$. We partition each set $\Omega_{j}$ into sets
$\llbracket\Omega_{j}\rrbracket_{1}^{s}\sqcup\llbracket\Omega_{j}\rrbracket_{2}^{s}\sqcup\ldots\sqcup\llbracket\Omega_{j}\rrbracket_{s}^{s}$
of the same measure (Definition~\ref{def:cutSet} below introduces
this formally). For each $j\in[k]$, we can randomly shuffle the sets
$\left\{ \llbracket\Omega_{j}\rrbracket_{1}^{s},\llbracket\Omega_{j}\rrbracket_{2}^{s},\ldots,\llbracket\Omega_{j}\rrbracket_{s}^{s}\right\} $.
Putting these shuffles together, we obtain a random measure preserving
bijection $\phi:\Omega\rightarrow\Omega$ with the property that 
\begin{equation}
\Omega_{j}=\phi\left(\Omega_{j}\right)\text{ for each \ensuremath{j\in[k]},}\label{eq:stayhome}
\end{equation}
and hence a random version $U^{\phi}$ of $U$ (Definition~\ref{def:RandomReshuffle}
below introduces this formally). Such a random version typically blurs
whatever structure there was in each rectangle $\Omega_{j}\times\Omega_{\ell}$
($j,\ell\in[k]$).\footnote{\label{fn:blur}For example, if $\Omega_{j}\times\Omega_{\ell}$ consisted
of two parts of equal measure, $U$ being $0.1$ on one part and $0.7$
on the other, then $U^{\phi}$ on $\Omega_{j}\times\Omega_{\ell}$
will consist with high probability of a checkerboard with random-like
alternation of $0.1$'s and $0.7$'s. In particular, for each fixed
$X\subset\Omega_{j}\times\Omega_{\ell}$, we will have with high probability
that $\int_{X}U^{\phi}\approx\frac{1}{2}(0.1+0.7)\nu^{\otimes2}(X)$.} Hence, a rather straightforward application of the Law of Large Numbers
gives that with high probability, the mean of independent random versions
$U^{\phi_{1}},U^{\phi_{2}}\ldots,U^{\phi_{N}}$ approximates $V$
in the $L^{1}$-distance.

Our actual Lemma~\ref{lem:L1approxByVersions} strengthens Lemma~\ref{lem:baby}
in two ways. Firstly, it assumes that $V\preceq U$ which is more
general than $V=U^{\Join\mathcal{R}}$ for some finite partition $\mathcal{R}$.
This represents only a minor complication in the proof as these properties
are almost the same (see for example Lemma~\ref{lem:LessStructuredStepping}).
So, we describe the second (and main) strengthening under the notationally
more convenient assumption that $V=U^{\Join\mathcal{R}}$ for $\mathcal{R}=\left\{ \Omega_{1},\Omega_{2},\ldots,\Omega_{k}\right\} $.
In addition to the approximation property as in Lemma~\ref{lem:baby},
we require that many of the pairs $U^{\phi_{2i-1}}$ and $U^{\phi_{2i}}$
are at least as far apart in the cut norm distance, as a constant
multiple of $\delta_{\square}\left(U,V\right)$. (Note that this statement
is void when $\delta_{\square}\left(U,V\right)=0$. Indeed in that
case there is no way we could hope for such a property.) Let us explain
why we expect this to occur for two independent random versions $U^{\phi_{2i-1}}$
and $U^{\phi_{2i}}$ with high probability. To this end, let us fix
$S,T\subset\Omega$ for which we have $\left|\int_{S\times T}(U-V)\right|>\frac{\delta_{\square}\left(U,V\right)}{2}$.
Without loss of generality, let us assume that $\int_{S\times T}(U-V)>\frac{\delta_{\square}\left(U,V\right)}{2}$.
Let us now look at $U^{\phi_{2i-1}}$. We clearly have $\int_{S\times T}U=\int_{\phi_{2i-1}^{-1}(S)\times\phi_{2i-1}^{-1}(T)}U^{\phi_{2i-1}}$.
Since $V$ is a step-function on $\mathcal{R}\times\mathcal{R}$,
for any measure preserving bijection $\phi$ satisfying~(\ref{eq:stayhome})
we have $\int_{S\times T}V=\int_{\phi^{-1}\left(S\right)\times\phi^{-1}\left(T\right)}V$,
and hence 
\begin{equation}
\int_{\phi_{2i-1}^{-1}(S)\times\phi_{2i-1}^{-1}(T)}(U^{\phi_{2i-1}}-V)>\frac{\delta_{\square}\left(U,V\right)}{2}\;.\label{eq:JacobsLadder}
\end{equation}
Let us now look at $\int_{\phi_{2i-1}^{-1}(S)\times\phi_{2i-1}^{-1}(T)}U^{\phi_{2i}}$.
As we said earlier (recall Footnote~\ref{fn:blur}), the version
$U^{\phi_{2i}}$ with high probability blurs any structure on each
rectangle $\Omega_{j}\times\Omega_{\ell}$. Thus, with high probability,
$\int_{\phi_{2i-1}^{-1}(S)\times\phi_{2i-1}^{-1}(T)}U^{\phi_{2i}}\approx\int_{\phi_{2i-1}^{-1}(S)\times\phi_{2i-1}^{-1}(T)}V$.
Combined with~(\ref{eq:JacobsLadder}), this proves that $U^{\phi_{2i-1}}$
and $U^{\phi_{2i}}$ are far apart in the cut norm distance. In the
actual proof, we need to deal with several technical difficulties.
\begin{lem}
\label{lem:L1approxByVersions}Suppose that $U,V\in\GRAPHONSPACE$
and $V\prec U$. Then for any $\varepsilon>0$ we can find an even
number $N$ and measure preserving bijections $\left(\phi_{i}:\Omega\rightarrow\Omega\right)_{i=1}^{N}$
such that 
\begin{equation}
\left\Vert V-\frac{\sum_{i=1}^{N}U^{\phi_{i}}}{N}\right\Vert _{1}<\varepsilon\;.\label{eq:LampaSviti}
\end{equation}
Moreover, for at least half of the indices $i\in\left\{ 1,2,\ldots,\frac{N}{2}\right\} $
we have 
\begin{equation}
\left\Vert U^{\phi_{2i-1}}-U^{\phi_{2i}}\right\Vert _{\square}>\frac{\delta_{\square}\left(U,V\right)}{32}\;.\label{eq:farapartmycka}
\end{equation}
\end{lem}

\begin{rem}
\label{rem:RescalingL1approxByVersions}Obviously, by rescaling, Lemma~\ref{lem:L1approxByVersions}
can be extended to $U,V\in\POSITIVEKERNELSPACE$.
\end{rem}

The rest of this section is devoted to proving Lemma~\ref{lem:L1approxByVersions}.
The key construction in the proof of Lemma~\ref{lem:L1approxByVersions}
is very similar to the proof of Lemma~9 in~\cite{DH:WeakStar} (however,
our proof is substantially more complex due to the additional property~(\ref{eq:farapartmycka})).
We borrow the following two definitions from~\cite{DH:WeakStar}.
\begin{defn}
\label{def:cutSet}Given a set $A\subset[0,1]$ of positive measure
and a number $s\in\mathbb{N}$, we can consider a partition $A=\llbracket A\rrbracket_{1}^{s}\sqcup\llbracket A\rrbracket_{2}^{s}\sqcup\ldots\sqcup\llbracket A\rrbracket_{s}^{s}$,
where each set $\llbracket A\rrbracket_{i}^{s}$ has measure $\frac{\lambda(A)}{s}$
and for each $1\le i<j\le s$, the set $\llbracket A\rrbracket_{i}^{s}$
is entirely to the left of $\llbracket A\rrbracket_{j}^{s}$. These
conditions define the partition $A=\llbracket A\rrbracket_{1}^{s}\sqcup\llbracket A\rrbracket_{2}^{s}\sqcup\ldots\sqcup\llbracket A\rrbracket_{s}^{s}$
uniquely, up to null sets. For each $i,j\in[s]$ there is a natural,
uniquely defined (up to null sets), measure preserving almost-bijection
$\chi_{i,j}^{A,s}:\llbracket A\rrbracket_{i}^{s}\rightarrow\llbracket A\rrbracket_{j}^{s}$
which preserves the order on the real line. 
\end{defn}

We can now give a definition of the model of random versions of a
given graphon which we need. An illustration is given in Figure~\ref{fig:Stripes}.
\begin{figure}
\includegraphics[scale=0.87]{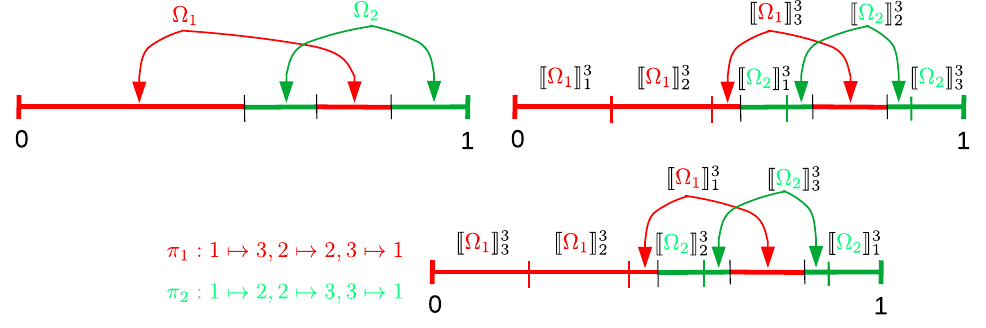}

\caption{An illustration to Definitions~\ref{def:cutSet} and~\ref{def:RandomReshuffle}.
A sample partition $\mathcal{R}=\left\{ \Omega_{1},\Omega_{2}\right\} $
of $[0,1]$ in the top left figure. To create $\mathbb{W}(\Gamma,\mathcal{R},s=3)$,
the unit interval is partitioned as in the top right figure. Finally,
for randomly chosen permutations $\pi_{1},\pi_{2}$, the unit interval
is effectively reshuffled as in the bottom figure.}
\label{fig:Stripes}
\end{figure}

\begin{defn}
\label{def:RandomReshuffle}Suppose that $\Gamma:[0,1]^{2}\rightarrow[0,1]$
is a graphon. For a finite partition $\mathcal{R}=\left\{ \Omega_{1},\Omega_{2},\ldots,\Omega_{k}\right\} $
of $[0,1]$ and for $s\in\mathbb{N}$, we define a discrete probability
distribution $\mathbb{W}(\Gamma,\mathcal{R},s)$ on versions of $\Gamma$
as follows. We take $\pi_{1},\ldots,\pi_{k}:[s]\rightarrow[s]$ independent
uniformly random permutations. After these are fixed, we define a
sample $W\sim\mathbb{W}(\Gamma,\mathcal{R},s)$ by 
\[
W(x,y)=\Gamma\left(\chi_{p,\pi_{i}(p)}^{\Omega_{i},s}(x),\chi_{q,\pi_{j}(q)}^{\Omega_{j},s}(y)\right)\quad\mbox{when \ensuremath{x\in\llbracket\Omega_{i}\rrbracket_{p}^{s}}, \ensuremath{y\in\llbracket\Omega_{j}\rrbracket_{q}^{s}}, \ensuremath{i,j\in[k]}, \ensuremath{p,q\in[s]}}\;.
\]
This defines the sample $W:[0,1]^{2}\rightarrow[0,1]$ uniquely up
to null sets, and thus defines the whole distribution $\mathbb{W}(\Gamma,\mathcal{R},s)$.
Observe that $\mathbb{W}(\Gamma,\mathcal{R},s)$ is supported on (some)
versions of $\Gamma$. We call the sets $\llbracket\Omega_{j}\rrbracket_{q}^{s}$
\emph{stripes}.
\end{defn}

We shall also need the following technical lemmas.
\begin{lem}
\label{lem:LessStructuredStepping}Let $V\preceq U$ be two graphons
on $\Omega^{2}$ and $\varepsilon>0$. Then there is a measure preserving
bijection $\varphi:\Omega\rightarrow\Omega$ and a finite partition
$\mathcal{R}$ of $\Omega$ such that 
\[
\left\Vert V-\left(U^{\varphi}\right)^{\Join\mathcal{R}}\right\Vert _{1}<\varepsilon.
\]
\end{lem}

\begin{proof}
Use Lemma~\ref{lem:approx} to find a finite partition $\mathcal{R}$
such that $\left\Vert V-V^{\Join\mathcal{R}}\right\Vert _{1}<\frac{\varepsilon}{2}$.
By the definition of $V\preceq U$ we may find a sequence $\left\{ U^{\varphi_{i}}\right\} _{i\in\mathbb{N}}$
of versions of $U$ such that $U^{\varphi_{i}}\WEAKCONV V$. By Lemma~\ref{lem:weakstarstepping},
we have $\left(U^{\varphi_{i}}\right)^{\Join\mathcal{R}}\WEAKCONV V^{\Join\mathcal{R}}$.
Since on the left-hand side, we have step-functions on the same grid
$\mathcal{R}\times\mathcal{R}$, the weak{*} convergence is in this
case equivalent to the convergence in $\left\Vert \cdot\right\Vert _{1}$,
$\left(U^{\varphi_{i}}\right)^{\Join\mathcal{R}}\LONECONV V^{\Join\mathcal{R}}$.
Now, we can find $i\in\mathbb{N}$ such that $\left\Vert V-\left(U^{\varphi_{i}}\right)^{\Join\mathcal{R}}\right\Vert _{1}\le\left\Vert V-V^{\Join\mathcal{R}}\right\Vert _{1}+\left\Vert V^{\Join\mathcal{R}}-\left(U^{\varphi_{i}}\right)^{\Join\mathcal{R}}\right\Vert _{1}<\varepsilon$.
\end{proof}
\begin{lem}
\label{lem:ApproxByStripes}Let $A\subset\left[0,1\right]$ be a measurable
set, $\mathcal{R}=\left\{ \Omega_{1},\dots,\Omega_{k}\right\} $ be
a finite partition of $\left[0,1\right]$ and $\varepsilon>0$. Then
there is $s_{0}\in\mathbb{N}$ such that for every $s\ge s_{0}$ we
may find a collection $\mathcal{I_{A}}\subset\left[k\right]\times\left[s\right]$
such that for the symmetric difference of $A$ and $\bigcup_{\left(i,x\right)\in\mathcal{I}_{A}}\left\llbracket \Omega_{i}\right\rrbracket _{x}^{s}$
we have
\[
\nu\left(A\diamond\bigcup_{\left(i,x\right)\in\mathcal{I}_{A}}\left\llbracket \Omega_{i}\right\rrbracket _{x}^{s}\right)<\varepsilon.
\]
\end{lem}

\begin{proof}
First we demonstrate that it is enough to show the lemma for the special
case when $\mathcal{R}$ is an interval partition. Suppose that $A\subset\left[0,1\right]$
and $\varepsilon>0$ is given. We may find a measurable almost-bijection
$\varphi$ such that the restriction of $\varphi$ to each $\Omega_{j}$
preserves the order of the real line and such that $\varphi\left(\mathcal{R}\right)$
is an interval partition. Then we find the correct $s_{0}\in\mathbb{N}$
when applied for $\varphi\left(A\right)$. It is clear that this $s_{0}\in\mathbb{N}$
works because $\varphi\left(\left\llbracket \Omega_{i}\right\rrbracket _{x}^{s}\right)=\left\llbracket \varphi\left(\Omega_{i}\right)\right\rrbracket _{x}^{s}$.

If $\mathcal{R}$ is a finite interval partition then we can work
in each interval separately. This implies that we may restrict ourselves
to the case where $\mathcal{R}=\left\{ \left[0,1\right]\right\} $.
The latter is a basic fact about the Lebesgue measure.
\end{proof}
\begin{lem}
\label{lem:StripAPP} Let $U\in\mathcal{W}_{0}$, $\mathcal{R}=\left\{ \Omega_{1},\dots,\Omega_{k}\right\} $
be a finite partition and $\varepsilon>0$. Then there is $s_{0}\in\mathbb{N}$
such that for every $s\ge s_{0}$ we have $\left\Vert U-U^{\Join\left\llbracket \mathcal{R}\right\rrbracket ^{s}}\right\Vert _{1}<\varepsilon$
where 
\[
\left\llbracket \mathcal{R}\right\rrbracket ^{s}=\left\{ \left\llbracket \Omega_{i}\right\rrbracket _{x}^{s}:\left(i,x\right)\in\left[k\right]\times\left[s\right]\right\} .
\]
\end{lem}

\begin{proof}
Using Lemma~\ref{lem:ApproxByStripes}, it is straightforward to
show that the assertion is true if there is a finite partition $\mathcal{P}=\{P_{1},\ldots,P_{\ell}\}$
of $\Omega$ such that $U$ is constant on each $P_{i}\times P_{j}$.
To prove the general case, use Lemma~\ref{lem:approx} to find a
suitable partition $\mathcal{P}$ of $\Omega$, then apply the previous
observation to $U^{\Join\mathcal{P}}$ and use the inequality
\[
\|U-U{}^{\Join\left\llbracket \mathcal{R}\right\rrbracket ^{s}}\|_{1}\le\|U-U^{\Join\mathcal{P}}\|_{1}+\|U^{\Join\mathcal{P}}-(U^{\Join\mathcal{P}})^{\Join\left\llbracket \mathcal{R}\right\rrbracket ^{s}}\|_{1}+\|(U^{\Join\mathcal{P}})^{\Join\left\llbracket \mathcal{R}\right\rrbracket ^{s}}-U{}^{\Join\left\llbracket \mathcal{R}\right\rrbracket ^{s}}\|_{1}\:.
\]
\end{proof}
\begin{prop}
\label{prop:RandomSEQ}Let $U$ be a graphon, $\mathcal{R}=\left\{ \Omega_{1},\dots,\Omega_{k}\right\} $
a finite partition of $\left[0,1\right]$ and $\varepsilon>0$. Then
there is $s_{0}\in\mathbb{N}$ such that for every $s\ge s_{0}$ there
is $N_{0}\in\mathbb{N}$ such that for every $N\ge N_{0}$, an independently
chosen random $N$-tuple $\left(W_{i}\sim\mathbb{W}\left(U,\mathcal{R},s\right)\right)_{i=1}^{N}$
satisfies 
\begin{equation}
\left\Vert U^{\Join\mathcal{R}}-\frac{\sum_{i=1}^{N}W_{i}}{N}\right\Vert _{1}<\varepsilon\label{eq:Hamm}
\end{equation}
with probability at least $0.9$.
\end{prop}

\begin{proof}
We will set $s_{0},s,N_{0},N$ later, and start with some bounds that
hold for all $s$ and $N$, which we from now on suppose to be fixed.
The idea is to split~(\ref{eq:Hamm}) to the contributions of individual
parts $\left(\llbracket\Omega_{i}\rrbracket_{x}^{s}\times\llbracket\Omega_{j}\rrbracket_{y}^{s}\right)_{i,j\in[k],x,y\in[s]}$. 

We call a quadruple $(i,j,x,y)\in[k]\times[k]\times[s]\times[s]$
a \emph{diagonal quadruple} if $i=j$ and $x=y$. A quadruple $(i,j,p,q)\in[k]\times[k]\times[s]\times[s]$
is \emph{off-diagonal of Type~I} if $i\neq j$, and it is \emph{off-diagonal
of Type~II} if $i=j$ and $x\neq y$.
\begin{claim*}[Claim~A]
Suppose that $(i,j,x,y)\in[k]\times[k]\times[s]\times[s]$ is an
off-diagonal quadruple. Then for each $a>0$ we have with probability
at least $1-2\exp\left(-\frac{a^{2}N}{2}\right)$ that
\[
\left|\frac{1}{s^{2}}\int_{\Omega_{i}\times\Omega_{j}}U\mathrm{d}\nu^{\otimes2}-\frac{1}{N}\sum_{\ell=1}^{N}\int_{\llbracket\Omega_{i}\rrbracket_{x}^{s}\times\llbracket\Omega_{j}\rrbracket_{y}^{s}}W_{\ell}\mathrm{d}\nu^{\otimes2}\right|\le a+\frac{2\nu^{\otimes2}(\Omega_{i}\times\Omega_{j})}{s^{2}\left(s-1\right)}\;.
\]
\end{claim*}
\begin{proof}[Proof of Claim~A]
 For each $\ell\in[N]$, we define 
\[
Y_{\ell}:=\int_{\llbracket\Omega_{i}\rrbracket_{x}^{s}\times\llbracket\Omega_{j}\rrbracket_{y}^{s}}W_{\ell}\mathrm{d}\nu^{\otimes2}\;.
\]
Then $Y_{\ell}:\mathbb{W}\left(U,\mathcal{R},s\right)\to\left[0,1\right]$
is a random variable with the expectation 
\[
e_{1}:=\frac{1}{s^{2}}\cdot\int_{\Omega_{i}\times\Omega_{j}}U\mathrm{d}\nu^{\otimes2}
\]
if $(i,j,x,y)$ if off-diagonal of Type~I, or 
\[
e_{2}:=\frac{1}{s\left(s-1\right)}\int_{\Omega_{i}\times\Omega_{j}\setminus\cup_{d=1}^{s}\llbracket\Omega_{i}\rrbracket_{d}^{s}\times\llbracket\Omega_{j}\rrbracket_{d}^{s}}U\mathrm{d}\nu^{\otimes2}=\left(\frac{1}{s^{2}}\int_{\Omega_{i}\times\Omega_{j}}U\mathrm{d}\nu^{\otimes2}\right)\pm\frac{2\nu^{\otimes2}\left(\Omega_{i}\times\Omega_{j}\right)}{s^{2}\left(s-1\right)}
\]
if $(i,j,x,y)$ if off-diagonal of Type~II. To see this, consider
first the case of Type~I. In that case, using the notation from Definition~\ref{def:RandomReshuffle},
we have a random permutation $\pi_{i}$ permuting stripes of $\Omega_{i}$
and a different random permutation $\pi_{j}$ permuting stripes of
$\Omega_{j}$. Such a pair of random permutations induces a permutation
of the grid on $\Omega_{i}\times\Omega_{j}$ such that the probability
that any given cell is placed onto the cell $\llbracket\Omega_{i}\rrbracket_{x}^{s}\times\llbracket\Omega_{j}\rrbracket_{y}^{s}$
is $\frac{1}{s^{2}}$, which justifies that the average is $e_{1}$
in this case. Similarly, in the case of Type~II (i.e., $i=j$), we
have one permutation $\pi_{i}$ which permutes simultaneously rows
and columns of the grid on $\Omega_{i}\times\Omega_{i}$. In that
case, the probability that any given off-diagonal cell is placed onto
the cell $\llbracket\Omega_{i}\rrbracket_{x}^{s}\times\llbracket\Omega_{i}\rrbracket_{y}^{s}$
is $\frac{1}{s(s-1)}$. 

Observe that $\left(Y_{\ell}\right)_{\ell=1}^{N}$ are independent
random variables. Thus, the Chernoff bound (Lemma~\ref{lem:Chernoff})
gives us that for each $a>0$ we have with probability at least $1-2\exp\left(-\frac{a^{2}N^{2}}{2N}\right)=1-2\exp\left(-\frac{a^{2}N}{2}\right)$
that 
\begin{equation}
\left|\frac{1}{s^{2}}\int_{\Omega_{i}\times\Omega_{j}}U\mathrm{d}\nu^{\otimes2}-\frac{1}{N}\sum_{\ell=1}^{N}\int_{\llbracket\Omega_{i}\rrbracket_{x}^{s}\times\llbracket\Omega_{j}\rrbracket_{y}^{s}}W_{\ell}\mathrm{d}\nu^{\otimes2}\right|=\left|e_{1}-\left(\frac{1}{N}\sum_{\ell=1}^{N}Y_{\ell}\right)\right|\le a\label{eq:ChernobylBound-1}
\end{equation}
in the case of off-diagonal quadruple of Type~I and 
\begin{align}
\left|\frac{1}{s^{2}}\int_{\Omega_{i}\times\Omega_{j}}U\mathrm{d}\nu^{\otimes2}-\frac{1}{N}\sum_{\ell=1}^{N}\int_{\llbracket\Omega_{i}\rrbracket_{x}^{s}\times\llbracket\Omega_{j}\rrbracket_{y}^{s}}W_{\ell}\mathrm{d}\nu^{\otimes2}\right| & \le\label{eq:ChernobylBound-1-1}\\
\le\left|e_{2}-\left(\frac{1}{N}\sum_{\ell=1}^{N}Y_{\ell}\right)\right|+\frac{2\nu^{\otimes2}\left(\Omega_{i}\times\Omega_{j}\right)}{s^{2}\left(s-1\right)} & \le a+\frac{2\nu^{\otimes2}\left(\Omega_{i}\times\Omega_{j}\right)}{s^{2}\left(s-1\right)}\nonumber 
\end{align}
in the case of off-diagonal quadruple of Type~II. Hence, (\ref{eq:ChernobylBound-1})
and (\ref{eq:ChernobylBound-1-1}) give the statement of the claim.
\end{proof}
Take $s_{0}\ge3$ that satisfies Lemma~\ref{lem:StripAPP} with $\frac{\varepsilon}{2}$,
and such that $\frac{k^{2}+3}{s_{0}-1}<\frac{\varepsilon}{2}$. Then
for every $s\ge s_{0}$ we may find $N_{0}$ such that for every $N\ge N_{0}$
we have $1-2\exp\left(-\frac{\left(\frac{1}{s^{3}}\right)^{2}N}{2}\right)>1-\frac{1}{10k^{2}s^{2}}$.
We have
\begin{align*}
\left\Vert U^{\Join\mathcal{R}}-\frac{1}{N}\sum_{\ell=1}^{N}W_{\ell}\right\Vert _{1} & \le\left\Vert U^{\Join\mathcal{R}}-\frac{1}{N}\sum_{\ell=1}^{N}\left(W_{\ell}\right)^{\Join\left\llbracket \mathcal{R}\right\rrbracket ^{s}}\right\Vert _{1}+\left\Vert \frac{1}{N}\sum_{\ell=1}^{N}\left(W_{\ell}\right)^{\Join\left\llbracket \mathcal{R}\right\rrbracket ^{s}}-\frac{1}{N}\sum_{\ell=1}^{N}W_{\ell}\right\Vert _{1}\\
 & \le\left\Vert U^{\Join\mathcal{R}}-\frac{1}{N}\sum_{\ell=1}^{N}\left(W_{\ell}\right)^{\Join\left\llbracket \mathcal{R}\right\rrbracket ^{s}}\right\Vert _{1}+\frac{\varepsilon}{2}.
\end{align*}
The following simple observation allows us to use Claim~A:
\begin{align*}
 & \left\Vert U^{\Join\mathcal{R}}-\frac{1}{N}\sum_{\ell=1}^{N}\left(W_{\ell}\right)^{\Join\left\llbracket \mathcal{R}\right\rrbracket ^{s}}\right\Vert _{1}\\
= & \sum_{\left(i,j,x,y\right)\in\left[k\right]\times\left[k\right]\times\left[s\right]\times\left[s\right]}\int_{\llbracket\Omega_{i}\rrbracket_{x}^{s}\times\llbracket\Omega_{j}\rrbracket_{y}^{s}}\\
 & \,\,\,\,\,\left|\frac{1}{\nu^{\otimes2}(\Omega_{i}\times\Omega_{j})}\int_{\Omega_{i}\times\Omega_{j}}U\mathrm{d}\nu^{\otimes2}-\frac{1}{N}\sum_{\ell=1}^{N}\frac{s^{2}}{\nu^{\otimes2}(\Omega_{i}\times\Omega_{j})}\int_{\llbracket\Omega_{i}\rrbracket_{x}^{s}\times\llbracket\Omega_{j}\rrbracket_{y}^{s}}W_{\ell}\mathrm{d}\nu^{\otimes2}\right|\mathrm{d}\nu^{\otimes2}\\
= & \sum_{\left(i,j,x,y\right)\in\left[k\right]\times\left[k\right]\times\left[s\right]\times\left[s\right]}\left|\frac{1}{s^{2}}\int_{\Omega_{i}\times\Omega_{j}}U\mathrm{d}\nu^{\otimes2}-\frac{1}{N}\sum_{\ell=1}^{N}\int_{\llbracket\Omega_{i}\rrbracket_{x}^{s}\times\llbracket\Omega_{j}\rrbracket_{y}^{s}}W_{\ell}\mathrm{d}\nu^{\otimes2}\right|\;.
\end{align*}
Claim A applied with $a=\frac{1}{s^{3}}$ gives that with probability
at least $1-\left(\sum_{i,j\in\left[k\right]}\sum_{x,y\in[s]}\frac{1}{10k^{2}s^{2}}\right)=0.9$
we have
\begin{align*}
 & \sum_{\left(i,j,x,y\right)\in\left[k\right]\times\left[k\right]\times\left[s\right]\times\left[s\right]}\left|\frac{1}{s^{2}}\int_{\Omega_{i}\times\Omega_{j}}U\mathrm{d}\nu^{\otimes2}-\frac{1}{N}\sum_{\ell=1}^{N}\int_{\llbracket\Omega_{i}\rrbracket_{x}^{s}\times\llbracket\Omega_{j}\rrbracket_{y}^{s}}W_{\ell}\mathrm{d}\nu^{\otimes2}\right|\\
\le & \sum_{i,j\in\left[k\right]\times\left[k\right]}s^{2}\left(a+\frac{2\nu^{\otimes2}\left(\Omega_{i}\times\Omega_{j}\right)}{s^{2}\left(s-1\right)}\right)+\sum_{(i,x)\in\left[k\right]\times\left[s\right]}\nu^{\otimes2}\left(\llbracket\Omega_{i}\rrbracket_{x}^{s}\times\llbracket\Omega_{i}\rrbracket_{x}^{s}\right)\\
\le & k^{2}s^{2}a+\frac{2}{s-1}+\frac{1}{s}\le\frac{k^{2}+3}{s-1}<\frac{\varepsilon}{2}\;.
\end{align*}
\end{proof}
\begin{prop}
\label{prop:RandomCUT}Let $U$ be a graphon, $\mathcal{R}=\left\{ \Omega_{1},\dots,\Omega_{k}\right\} $
be a finite partition of $\left[0,1\right]$. Then there is $s_{0}\in\mathbb{N}$
such that for every $s\ge s_{0}$, a random graphon $W\sim\mathbb{W}\left(U,\mathcal{R},s\right)$
satisfies
\[
\left\Vert U-W\right\Vert _{\square}>\frac{\delta_{\square}\left(U,U^{\Join\mathcal{R}}\right)}{8}-\frac{2}{s^{\frac{1}{4}}}
\]
with probability at least  $1-\exp\left(-\frac{\sqrt{s}}{8k}\right)$.
\end{prop}

\begin{proof}
Put $V=U^{\Join\mathcal{R}}$ and write $\delta=\delta_{\square}\left(U,V\right)$.
By Lemma~\ref{lem:DisjointWitness} we may find $S_{0},T_{0}\subset\Omega$
such that $S_{0}\cap T_{0}=\emptyset$ and 
\[
\left|\int_{S_{0}\times T_{0}}(U-V)\right|\ge\frac{\delta}{4}\;.
\]
By a slight modification of Lemma~\ref{lem:ApproxByStripes} there
is $s_{0}\ge3$ such that for every $s\ge s_{0}$ there are $S,T\subset\Omega$
that are unions of stripes such that 
\[
\left|\int_{S\times T}(U-V)\right|\ge\frac{\delta}{8}
\]
and $S\cap T=\emptyset$. Now fix $s\ge s_{0}$ and denote $\mathcal{I}_{S},\mathcal{I}_{T}\subset\left[k\right]\times\left[s\right]$
the sets of indices giving the corresponding stripes, i.e., 
\[
S=\bigcup_{\left(i,x\right)\in\mathcal{I}_{S}}\left\llbracket \Omega_{i}\right\rrbracket _{x}^{s}
\]
and similarly for $T$. Define $\mathcal{I}_{S}^{i}=\left\{ x\in\left[s\right]:\left(i,x\right)\in\mathcal{I}_{S}\right\} $,
$\mathcal{C}_{S}=\left\{ \left\llbracket \Omega_{i}\right\rrbracket _{x}^{s}\right\} _{\left(i,x\right)\in\mathcal{I}_{S}}$
and $\mathcal{C}_{S}^{i}=\left\{ \left\llbracket \Omega_{i}\right\rrbracket _{x}^{s}\right\} _{x\in\mathcal{I}_{S}^{i}}$,
similarly define $\mathcal{I}_{T}^{i}$, $\mathcal{C}_{T}$ and $\mathcal{C}_{T}^{i}$.

We may assume that $S_{i}:=\bigcup\mathcal{C}_{S}^{i}$ is on the
left side of $\Omega_{i}$ and that $T_{i}:=\bigcup\mathcal{C}_{T}^{i}$
is exactly next to it. To see this note that if $\mathcal{C}_{S}^{i}$
and $\mathcal{C}_{T}^{i}$ are in some general position, then we may
find a measure preserving bijection $\varphi$ that is invariant on
each $\Omega_{i}$ and permutes the stripes accordingly. Note that
this is possible because $S,T$ are disjoint. Then for $\mathcal{C}_{S}^{i},\mathcal{C}_{T}^{i}$
in the general position use the same argument with conjugation by
$\varphi$ as in Lemma~\ref{lem:ApproxByStripes}. 

Define the random variable $Z:\mathbb{W}\left(U,\mathcal{R},s\right)\to\mathbb{R}$
as
\[
Z\left(W\right)=\int_{S\times T}W\mathrm{d}\nu^{\otimes2}.
\]
\begin{claim*}[Claim~B]
We have $\EXP\left[Z\right]=\int_{S\times T}V\mathrm{d}\nu^{\otimes2}\pm\frac{1}{s-1}$.
\end{claim*}
\begin{proof}[Proof of Claim~B]
For fixed $i,j\in[k]$ define $Z_{i,j}\left(W\right)=\int_{S_{i}\times T_{j}}W\mathrm{d}\nu^{\otimes2}$.
Then $Z=\sum_{i,j\in[k]}Z_{i,j}$ and also $\EXP\left[Z\right]=\sum_{i,j\in[k]}\EXP\left[Z_{i,j}\right]$.
It suffices to show that $\EXP\left[Z_{i,j}\right]=\int_{S_{i}\times T_{j}}V\mathrm{d}\nu^{\otimes2}$
if $i\not=j$ and $\EXP\left[Z_{i,i}\right]=\int_{S_{i}\times T_{i}}V\mathrm{d}\nu^{\otimes2}\pm\frac{\nu^{\otimes2}\left(\Omega_{i}\times\Omega_{i}\right)}{s-1}$
if $i=j$.

We use the notation from the proof of Proposition~\ref{prop:RandomSEQ}.
Take an off-diagonal quadruple $\left(i,j,x,y\right)\in\left[k\right]\times\left[k\right]\times\left[s\right]\times\left[s\right]$
and define
\[
Y_{i,j,x,y}=\int_{\llbracket\Omega_{i}\rrbracket_{x}^{s}\times\llbracket\Omega_{j}\rrbracket_{y}^{s}}W\mathrm{d}\nu^{\otimes2}\ .
\]
There are two cases depending on the type of $\left(i,j,x,y\right)$.
Suppose that $\left(i,j,x,y\right)$ is of Type I. Then we have
\[
\EXP\left[Y_{i,j,x,y}\right]=\frac{1}{s^{2}}\cdot\int_{\Omega_{i}\times\Omega_{j}}U\mathrm{d}\nu^{\otimes2}=\frac{1}{s^{2}}\cdot\int_{\Omega_{i}\times\Omega_{j}}V\mathrm{d}\nu^{\otimes2}\;,
\]
and summing over all $\left(x,y\right)\in\mathcal{I}_{S}^{i}\times\mathcal{I}_{T}^{j}$
we get for $i\neq j$ that
\[
\EXP\left[Z_{i,j}\right]=\frac{\left|\mathcal{I}_{S}^{i}\right|\times\left|\mathcal{I}_{T}^{j}\right|}{s^{2}}\int_{\Omega_{i}\times\Omega_{j}}V\mathrm{d}\nu^{\otimes2}=\int_{S_{i}\times T_{j}}V\mathrm{d}\nu^{\otimes2}.
\]

Suppose that $\left(i,i,x,y\right)$ is of Type II. Then we have 
\[
\EXP\left[Y_{i,i,x,y}\right]=\left(\frac{1}{s^{2}}\int_{\Omega_{i}\times\Omega_{i}}V\mathrm{d}\nu^{\otimes2}\right)\pm\frac{2\nu^{\otimes2}\left(\Omega_{i}\times\Omega_{i}\right)}{s^{2}\left(s-1\right)}\;,
\]
and summing over all $\left(x,y\right)\in\mathcal{I}_{S}^{i}\times\mathcal{I}_{T}^{i}$
(with $x$ and $y$ distinct) we get for every $i$ that
\begin{align*}
\EXP\left[Z_{i,i}\right] & =\frac{\left|\mathcal{I}_{S}^{i}\right|\times\left|\mathcal{I}_{T}^{i}\right|}{s^{2}}\left(\int_{\Omega_{i}\times\Omega_{i}}V\mathrm{d}\nu^{\otimes2}\pm\frac{2\nu^{\otimes2}\left(\Omega_{i}\times\Omega_{i}\right)}{\left(s-1\right)}\right)\\
 & =\int_{S_{i}\times T_{i}}V\mathrm{d}\nu^{\otimes2}\pm\frac{\nu^{\otimes2}\left(\Omega_{i}\times\Omega_{i}\right)}{s-1}
\end{align*}
because $\frac{\left|\mathcal{I}_{S}^{i}\right|\times\left|\mathcal{I}_{T}^{i}\right|}{s^{2}}\le\frac{1}{2}$.
\end{proof}
In order to use the Method of Bounded Differences we introduce the
following correspondence between permutations that induce $\mathbb{W}\left(U,\mathcal{R},s\right)$
and $\left[0,1\right]^{k\times s}$ (which we view as functions from
$[k]\times[s]$ to $[0,1]$). Namely, for each $f\in[k]\times[s]\rightarrow[0,1]$
which is injective we define a permutation $\alpha_{f}$ of $[k]\times[s]$
such that $\alpha_{f}(\{i\}\times[s])=\{i\}\times[s]$ for each $i$,
and such that the relative position of $\alpha_{f}(i,x)$ inside the
block $\{i\}\times[s]$ is the same as the relative position of $f\left(i,x\right)$
inside the set of numbers $\left\{ f\left(i,y\right)\right\} _{y\in\left[s\right]}$.
We leave $\alpha_{f}$ undefined for non-injective functions, which
form a nullset on $\left[0,1\right]^{k\times s}$. One can verify
that the assignment $f\mapsto(\alpha_{f})^{-1}$, which maps each
$f$ to the inverse of $\alpha_{f}$, is measure preserving where
we have the Lebesgue measure on $\left[0,1\right]^{k\times s}$ and
the uniform measure on the permutations of $\left[k\right]\times\left[s\right]$
that fix the first coordinate. Note that these are exactly the permutations
that naturally induce $\mathbb{W}\left(U,\mathcal{R},s\right)$. Hence,
we may consider the random variable $Z$ to be defined on $\left[0,1\right]^{k\times s}$
with values in $\mathbb{R}$. 

We show that $Z$ satisfies the assumptions of Lemma~\ref{lem:MethodBoundedDifferences}\ref{enu:MCDiarmidBasic}.
Recall that we assume that each $S_{i}$ is concentrated on the left-most
part of the interval $\Omega_{i}$ and $T_{i}$ is exactly next to
it. Suppose that $f,f'\in\left[0,1\right]^{k\times s}$ differ in
at most one coordinate in the $i$-th block. Then 
\[
\left|(\alpha_{f})^{-1}\left(\mathcal{I}_{S}^{i}\right)\diamond(\alpha_{f'})^{-1}\left(\mathcal{I}_{S}^{i}\right)\right|\le2\quad\mbox{and}\quad\left|(\alpha_{f})^{-1}\left(\mathcal{I}_{T}^{i}\right)\diamond(\alpha_{f'})^{-1}\left(\mathcal{I}_{T}^{i}\right)\right|\le2.
\]
Then we may compute
\begin{align*}
\left|Z\left(f\right)-Z\left(f'\right)\right| & =\left|\int_{S\times T}U^{(\alpha_{f})^{-1}}\mathrm{d}\nu^{\otimes2}-\int_{S\times T}U^{(\alpha_{f'})^{-1}}\mathrm{d}\nu^{\otimes2}\right|\\
 & =\left|\int_{(\alpha_{f})^{-1}\left(S\right)\times(\alpha_{f})^{-1}\left(T\right)}U\mathrm{d}\nu^{\otimes2}-\int_{(\alpha_{f'})^{-1}\left(S\right)\times(\alpha_{f'})^{-1}\left(T\right)}U\mathrm{d}\nu^{\otimes2}\right|\\
 & \le\nu^{\otimes2}\left(\left((\alpha_{f})^{-1}\left(S\right)\times(\alpha_{f})^{-1}\left(T\right)\right)\diamond\left((\alpha_{f'})^{-1}\left(S\right)\times(\alpha_{f'})^{-1}\left(T\right)\right)\right)\\
 & \le4\frac{\nu\left(\Omega_{i}\right)}{s}\le\frac{4}{s}.
\end{align*}
By Lemma~\ref{lem:MethodBoundedDifferences}\ref{enu:MCDiarmidBasic}
we have
\[
\PROB\left[\left|Z-\EXP Z\right|>d\right]\le\exp\left(-\frac{d^{2}s}{8k}\right)\;,\text{for any \ensuremath{d>0}.}
\]
In particular, taking $d=s^{-\frac{1}{4}}$ we have 
\[
\PROB\left[\left|Z-\EXP Z\right|>s^{-\frac{1}{4}}\right]\le\exp\left(-\frac{\sqrt{s}}{8k}\right)
\]
and therefore with probability at least $1-\exp\left(-\frac{\sqrt{s}}{8k}\right)$
we have that 
\begin{equation}
\left|Z\left(W\right)-\EXP Z\right|\le s^{-\frac{1}{4}}\label{eq:MilujuJirkuOvcacka}
\end{equation}
for $W\in\mathbb{W}\left(U,\mathcal{R},s\right)$. We conclude that
with probability at least $1-\exp\left(-\frac{\sqrt{s}}{8k}\right)$
we have that 
\begin{align*}
\left\Vert W-U\right\Vert _{\square} & \ge\left|\int_{S\times T}U\mathrm{d}\nu^{\otimes2}-\int_{S\times T}W\mathrm{d}\nu^{\otimes2}\right|\\
 & \ge\left|\int_{S\times T}U\mathrm{d}\nu^{\otimes2}-\int_{S\times T}V\mathrm{d}\nu^{\otimes2}\right|-\left|\int_{S\times T}V\mathrm{d}\nu^{\otimes2}-\int_{S\times T}W\mathrm{d}\nu^{\otimes2}\right|\\
 & \ge\frac{\delta}{8}-\left|Z\left(W\right)-\left(\EXP\left[Z\right]\pm\frac{1}{s-1}\right)\right|\\
\JUSTIFY{by\,\eqref{eq:MilujuJirkuOvcacka}} & \ge\frac{\delta}{8}-\frac{1}{s-1}-\frac{1}{s^{\frac{1}{4}}}>\frac{\delta}{8}-\frac{2}{s^{\frac{1}{4}}}\;,
\end{align*}
as was needed (we used that $s\ge s_{0}\ge3$ in the last inequality). 
\end{proof}
Now we are ready to prove Lemma~\ref{lem:L1approxByVersions}.
\begin{proof}[Proof of Lemma~\ref{lem:L1approxByVersions}]
Let $\delta:=\delta_{\square}\left(V,U\right)$. By Fact~\ref{fact:StrictlyBelowCutDistPos},
we have that $\delta>0$. First use Lemma~\ref{lem:LessStructuredStepping}
to approximate $V$ by some $\left(U^{\varphi}\right)^{\Join\mathcal{R}}$
such that $\left\Vert V-\left(U^{\varphi}\right)^{\Join\mathcal{R}}\right\Vert _{1}<\min\left(\frac{\varepsilon}{2},\frac{\delta}{2}\right)$.
We may assume without loss of generality that $\varphi$ is the identity
and therefore work with $U^{\Join\mathcal{R}}$ instead of $\left(U^{\varphi}\right)^{\Join\mathcal{R}}$.
We have 
\[
\delta_{\square}\left(U,U^{\Join\mathcal{R}}\right)\ge\delta_{\square}\left(V,U\right)-\delta_{\square}\left(V,U^{\Join\mathcal{R}}\right)\ge\delta_{\square}\left(V,U\right)-\left\Vert V-U^{\Join\mathcal{R}}\right\Vert _{1}>\frac{\delta}{2}\;.
\]
We use Proposition~\ref{prop:RandomSEQ} and Proposition~\ref{prop:RandomCUT}
to find $s\in\mathbb{N}$ and an even number $N$ such that $\left\Vert U-W\right\Vert _{\square}>\frac{\delta}{32}$
for $W\in\mathbb{W}\left(U,\mathcal{R},s\right)$ with probability
at least $0.9$, and also $\left\Vert U^{\Join\mathcal{R}}-\frac{\sum_{k=1}^{N}W_{k}}{N}\right\Vert _{1}<\frac{\varepsilon}{2}$
for $\left(W_{k}\right)_{k=1}^{N}\in\left(\mathbb{W}\left(U,\mathcal{R},s\right)\right)^{N}$
with probability at least $0.9$. 

Define a random variable $Q:\left(\mathbb{W}\left(U,\mathcal{R},s\right)\right)^{N}\rightarrow\mathbb{R}$,
\[
Q:\left(W_{k}\right)_{k=1}^{N}\mapsto\frac{2}{N}\cdot\left|\left\{ k\in[N/2]:\left\Vert W_{2k-1}-W_{2k}\right\Vert _{\square}\le\frac{\delta}{32}\right\} \right|.
\]
Note that for any $U^{*}\in\mathbb{W}(U,\mathcal{R},s)$ , the distributions
$\mathbb{W}(U,\mathcal{R},s)$ and $\mathbb{W}(U^{*},\mathcal{R},s)$
on versions of $U$ coincide with probability $1$. So for every $k\in[N/2]$,
we can equivalently first sample $W_{2k-1}\sim\mathbb{W}(U,\mathcal{R},s)$,
and then sample $W_{2k}\sim\mathbb{W}(W_{2k-1},\mathcal{R},s)$. Thus,
for a fixed $k\in[N/2]$ the probability that $\left\Vert W_{2k-1}-W_{2k}\right\Vert _{\square}>\frac{\delta}{32}$
is at least $0.9$ due to Proposition~\ref{prop:RandomCUT}. Hence,
$\EXP\left[Q\right]\le0.1$. By Markov's inequality, $\PROB\left[Q\ge0.4\right]\le0.25$.
By the union bound, with probability at most $0.35$ we have that
$Q\ge0.4$ or that $\left\Vert U^{\Join\mathcal{R}}-\frac{\sum_{k=1}^{N}W_{k}}{N}\right\Vert _{1}\ge\frac{\varepsilon}{2}$.
In particular, there exists a choice $\left(U^{\phi_{k}}\right)_{k=1}^{N}=\left(W_{k}\right)_{k=1}^{N}$
of an $N$-tuple of versions of $U$ which does not have any of these
two <<bad>> properties. Such an $N$-tuple $\left(U^{\phi_{k}}\right)_{k=1}^{N}$
satisfies~(\ref{eq:farapartmycka}) since $Q<0.4<0.5$. It also satisfies~(\ref{eq:LampaSviti})
since
\[
\left\Vert V-\frac{\sum_{i=1}^{N}U^{\phi_{i}}}{N}\right\Vert _{1}\le\left\Vert V-U^{\Join\mathcal{R}}\right\Vert _{1}+\left\Vert U^{\Join\mathcal{R}}-\frac{\sum_{i=1}^{N}U^{\phi_{i}}}{N}\right\Vert _{1}<\frac{\varepsilon}{2}+\frac{\varepsilon}{2}=\varepsilon\;.
\]
\end{proof}

\section{Cut distance identifying graphon parameters\label{sec:Parameters}}

\subsection{Basics\label{subsec:BasicsOfCDIGP}}

In~\cite{DGHRR:ACCLIM}, we based our treatment of the cut distance
on $\ACC\left(W_{1},W_{2},W_{3},\ldots\right)$ and $\LIM\left(W_{1},W_{2},W_{3},\ldots\right)$,
which are sets of functions. In contrast, the key objects in~\cite{DH:WeakStar}
are the sets of numerical values 
\[
\left\{ \INT_{f}(W):W\in\ACC\left(W_{1},W_{2},W_{3},\ldots\right)\right\} \text{ and }\left\{ \INT_{f}(W):W\in\LIM\left(W_{1},W_{2},W_{3},\ldots\right)\right\} ,
\]
with notation taken from~(\ref{eq:FDolHl}). In this section, we
introduce an abstract framework to approaching the cut distance via
similar optimization problems. Our key definitions of cut distance
identifying graphon parameters and cut distance compatible graphon
parameters use $\mathbb{R}^{n}$ together with lexicographical ordering
and Euclidean metric, and $\mathbb{R}^{\mathbb{N}}$ together with
lexicographical ordering which we denote just $\le$. 

By a \emph{graphon parameter} we mean any function $\theta:\GRAPHONSPACE\rightarrow\mathbb{R}$,
$\theta:\GRAPHONSPACE\rightarrow\mathbb{R}^{n}$ (for some $n\in\mathbb{N}$),
or $\theta:\GRAPHONSPACE\rightarrow\mathbb{R}^{\mathbb{N}}$, such
that $\theta(W_{1})=\theta(W_{2})$ for any two graphons $W_{1}$
and $W_{2}$ with $\delta_{\square}(W_{1},W_{2})=0$. By a \emph{graphon
order}, we mean a preorder $\leftslice$ on $\GRAPHONSPACE$ which
does not change within the weak isomorphism classes, i.e., for each
$W_{1},W_{2}\in\GRAPHONSPACE$ with $\delta_{\square}(W_{1},W_{2})=0$
we have $W_{1}\leftslice W_{2}$ and $W_{2}\leftslice W_{1}$. With
these preliminary definitions, we can introduce the central concept
of this paper, which we do in four variants, the distinction being
whether we require strict monotonicity or not, and whether we work
in the setting of graphon parameters or graphon orders.
\begin{defn}
~
\begin{itemize}
\item We say that a graphon parameter $\theta$ is a \emph{cut distance
identifying graphon parameter} if we have that $W_{1}\prec W_{2}$
implies $\theta\left(W_{1}\right)<\theta\left(W_{2}\right)$ (here,
by $<$ we understand the usual Euclidean order on $\mathbb{R}$ in
case $\theta:\GRAPHONSPACE\rightarrow\mathbb{R}$ and the lexicographic
order in case $\theta:\GRAPHONSPACE\rightarrow\mathbb{R}^{n}$ or
$\theta:\GRAPHONSPACE\rightarrow\mathbb{R}^{\mathbb{N}}$).
\item We say that a graphon parameter $\psi$ is a \emph{cut distance compatible
graphon parameter} if we have that $W_{1}\preceq W_{2}$ implies $\psi\left(W_{1}\right)\le\psi\left(W_{2}\right)$.
\item We say that a graphon order $\vartriangleleft$ is a \emph{cut distance
identifying graphon order} if $W_{1}\prec W_{2}$ implies $W_{1}\vartriangleleft W_{2}$
and $W_{2}\ntriangleleft W_{1}$. 
\item We say that a graphon order $\lessdot$ is a \emph{cut distance compatible
graphon order} if $W_{1}\preceq W_{2}$ implies $W_{1}\lessdot W_{2}$. 
\end{itemize}
\end{defn}

Cut distance identifying/compatible orders are abstract versions of
their parameter counterparts. Indeed, given a graphon parameter $\theta$,
we have that $\theta$ is cut distance identifying if and only if
a graphon order $\vartriangleleft_{\theta}$ defined by
\[
U\vartriangleleft_{\theta}W\quad\text{if and only if \ensuremath{\theta(U)<\theta(W)}}
\]
is cut distance identifying. Similarly, a graphon parameter $\psi$
is cut distance compatible if and only if a graphon order $\lessdot_{\psi}$
defined by
\begin{equation}
U\lessdot_{\psi}W\quad\text{if and only if \ensuremath{\psi(U)\le\psi(W)}}\label{eq:IAAT}
\end{equation}
is cut distance compatible. However, there are cut distance identifying
and cut distance compatible graphon orders that do not arise from
graphon parameters. Indeed, Proposition~\ref{prop:flatter} tells
us that the at-least-as-flat relation on degree frequencies induces
a cut distance compatible graphon order and that the strictly-flatter
relation on range frequencies induces a cut distance identifying graphon
order.\footnote{For this argument to make sense, we need the flatness relation to
be transitive. This follows from Lemma~4.13 in~\cite{DGHRR:ACCLIM}.}

\medskip{}

The following proposition provides a useful criterion for cut distance
compatible graphon parameters. In this criterion, we restrict ourselves
to $L^{1}$-continuous parameters. This is only a mild restriction.
Indeed, many prominent graphon parameters such as homomorphism densities
are even continuous with respect to the cut norm (which is a coarser
topology). As another important example, the graphon parameter $\INT_{f}(\cdot)$
is $L^{1}$-continuous if $f$ is a continuous function.
\begin{prop}
\label{prop:CDIPstep}Suppose that $\theta$ is a graphon parameter
that is continuous with respect to the $L^{1}$ norm. Then $\theta$
is cut distance compatible if and only if for each graphon $W:\Omega^{2}\rightarrow[0,1]$
and each finite partition $\mathcal{P}$ of $\Omega$ we have $\theta\left(W^{\Join\mathcal{P}}\right)\le\theta\left(W\right)$. 
\end{prop}

\begin{proof}
The $\Rightarrow$ direction is obvious, since $W^{\Join\mathcal{P}}\preceq W$
by Fact~\ref{fact:steppinginenvelope} ($L^{1}$continuity is not
needed for this direction). For the reverse direction, suppose that
$\theta$ is not cut distance compatible. That is, there exist two
graphons $U\preceq W$ so that $\theta(U)>\theta(W)$. Since $\theta$
is $L^{1}$-continuous at $U$ we can use Lemma~\ref{lem:approx}
to find a finite partition $\mathcal{Q}$ such that 
\begin{equation}
\theta\left(U^{\Join\mathcal{Q}}\right)>\theta(W).\label{eq:FolsomPrison}
\end{equation}
As $U\preceq W$, there exist measure preserving bijections $\pi_{1},\pi_{2},\pi_{3},\ldots$
so that $W^{\pi_{n}}\WEAKCONV U$. In particular, the sequence $\left(\left(W^{\pi_{n}}\right)^{\Join\mathcal{Q}}\right)_{n}$
converges to $U^{\Join\mathcal{Q}}$ in $L^{1}$. Thus the $L^{1}$-continuity
of $\theta$ at $U^{\Join\mathcal{Q}}$ gives us that for some $n$,
$\theta\left(\left(W^{\pi_{n}}\right)^{\Join\mathcal{Q}}\right)$
is nearly as big as $\theta\left(U^{\Join\mathcal{Q}}\right)$. In
particular, using~(\ref{eq:FolsomPrison}) we have that $\theta\left(\left(W^{\pi_{n}}\right)^{\Join\mathcal{Q}}\right)>\theta(W)$.
We let $\pi_{n}$ act on the partition $\mathcal{Q}$, $\mathcal{P}:=\pi_{n}(\mathcal{Q})$.
Obviously, $\left(W^{\pi_{n}}\right)^{\Join\mathcal{Q}}$ is a version
of $W^{\Join\mathcal{P}}$, and thus $\theta\left(W^{\Join\mathcal{P}}\right)=\theta\left(\left(W^{\pi_{n}}\right)^{\Join\mathcal{Q}}\right)>\theta\left(W\right)$,
as was needed.
\end{proof}
It is natural to believe that there is a similar characterization
for cut distance identifying parameters. We were however unable to
prove it, so we leave it as a conjecture. 
\begin{conjecture}
\label{conj:identifying_characterization}Suppose that $\theta$ is
a graphon parameter that is continuous with respect to the $L^{1}$
norm. Then $\theta$ is cut distance identifying if and only if for
each graphon $W$ and each finite partition $\mathcal{P}$ of $\Omega$
for which $W^{\Join\mathcal{P}}\neq W$ we have $\theta\left(W^{\Join\mathcal{P}}\right)<\theta\left(W\right)$.
\end{conjecture}

Note that the $\Rightarrow$ direction is obvious as in Proposition~\ref{prop:CDIPstep}. 

Cut distance identifying graphon parameters/orders can be used to
prove compactness of the graphon space. This is stated in the next
two theorems. 
\begin{thm}
\label{thm:subsequencegeneral}Let $\Gamma_{1},\Gamma_{2},\Gamma_{3},\ldots$
be a sequence of graphons.
\begin{description}
\item [{For~orders}] Suppose that $\vartriangleleft$ is a cut distance
compatible graphon order. Then there exists a subsequence $\Gamma_{n_{1}},\Gamma_{n_{2}},\Gamma_{n_{3}},\ldots$
such that $\ACC\left(\Gamma_{n_{1}},\Gamma_{n_{2}},\Gamma_{n_{3}},\ldots\right)$
contains an element $\Gamma$ with $W\vartriangleleft\Gamma$ for
each $W\in\ACC\left(\Gamma_{n_{1}},\Gamma_{n_{2}},\Gamma_{n_{3}},\ldots\right)$.
\item [{For~parameters}] Suppose that $\theta$ is a cut distance compatible
graphon parameter. Then there exists a subsequence $\Gamma_{n_{1}},\Gamma_{n_{2}},\Gamma_{n_{3}},\ldots$
such that $\ACC\left(\Gamma_{n_{1}},\Gamma_{n_{2}},\Gamma_{n_{3}},\ldots\right)$
contains an element $\Gamma$ with $\theta\left(\Gamma\right)=\sup\left\{ \theta(W):W\in\ACC\left(\Gamma_{n_{1}},\Gamma_{n_{2}},\Gamma_{n_{3}},\ldots\right)\right\} $.
\end{description}
\end{thm}

In both cases this follows immediately from~\cite[Theorem 3.3]{DGHRR:ACCLIM}
and~\cite[Lemma 4.9]{DGHRR:ACCLIM}. Note that the version for orders
is more general, since the parameter version can be reduced by~(\ref{eq:IAAT}).
Let us note that one could use the ideas from the proof of Lemma~16
from~\cite{DH:WeakStar} to obtain an alternative proof of the parameter
version of Theorem~\ref{thm:subsequencegeneral}. This latter proof
is more elementary and does not need transfinite induction or any
appeal to the Vietoris topology, which the machinery from~\cite{DGHRR:ACCLIM}
does. However, one needs to be a little careful while doing so because
not every subset of $\mathbb{R}^{N}$ (or $\mathbb{R}^{\mathbb{N}}$)
has a supremum in the lexicographical ordering. On the other hand,
the parameter version of Theorem~\ref{thm:subsequencegeneral} implicitly
says that the supremum of the set $\left\{ \theta(W):W\in\ACC\left(\Gamma_{n_{1}},\Gamma_{n_{2}},\Gamma_{n_{3}},\ldots\right)\right\} $
exists.
\begin{thm}
\label{thm:subsequencemaxbest}Let $W_{1},W_{2},W_{3},\ldots$ be
a sequence of graphons.\emph{ }
\begin{description}
\item [{For~orders}] Suppose that $\vartriangleleft$ is a cut distance
identifying graphon order. Suppose that $\Gamma\in\LIM\left(W_{1},W_{2},W_{3},\ldots\right)$
is such that $W\vartriangleleft\Gamma$ for each $W\in\ACC\left(W_{1},W_{2},W_{3},\ldots\right)$.
Then $W_{1},W_{2},W_{3},\ldots$ converges to $\Gamma$ in the cut
distance.
\item [{For~parameters}] Suppose that $\theta$ is a cut distance identifying
graphon parameter. Suppose that $\Gamma\in\LIM\left(W_{1},W_{2},W_{3},\ldots\right)$
is such that $\theta\left(\Gamma\right)=\sup\left\{ \theta(W):W\in\ACC\left(W_{1},W_{2},W_{3},\ldots\right)\right\} $.
Then $W_{1},W_{2},W_{3},\ldots$ converges to $\Gamma$ in the cut
distance.
\end{description}
\end{thm}

\begin{proof}
As the first step, we show that $\left\langle \Gamma\right\rangle =\ACC\left(W_{1},W_{2},\dots\right)=\LIM\left(W_{1,}W_{2},\dots\right)$.
Let $U\in\ACC\left(W_{1},W_{2},\dots\right)$. By Theorem~3.3 from~\cite{DGHRR:ACCLIM}
we can find a subsequence $W_{n_{1}},W_{n_{2}},\dots$ such that $\LIM\left(W_{n_{1}},W_{n_{2}},\dots\right)=\ACC\left(W_{n_{1}},W_{n_{2}},\dots\right)$
and $U\in\LIM\left(W_{n_{1}},W_{n_{2}},\dots\right)$. Note that $\Gamma\in\LIM\left(W_{n_{1}},W_{n_{2}},\dots\right)$.
Using Lemma~4.9 from~\cite{DGHRR:ACCLIM}, we can find a maximum
element $W\in\LIM\left(W_{n_{1}},W_{n_{2}},\dots\right)$ with respect
to the structuredness order. It follows that $\Gamma\preceq W$. Therefore
$\Gamma\vartriangleleft W$ or $\theta\left(\Gamma\right)\le\theta\left(W\right)$,
respectively. Using our assumption on $\Gamma$ and the fact that
$\vartriangleleft$ is a cut distance identifying graphon order or
that $\theta$ is a cut distance identifying graphon parameter, respectively,
we must have $\left\langle \Gamma\right\rangle =\left\langle W\right\rangle $.
This implies that $U\in\left\langle W\right\rangle =\left\langle \Gamma\right\rangle \subset\LIM\left(W_{1},W_{2},\dots\right)$
where we used the fact that $\LIM\left(W_{1},W_{2},\dots\right)$
is weak{*} closed (see \cite[Lemma 3.1]{DGHRR:ACCLIM}). This immediately
finishes the first step.

We may suppose that $W_{n}\WEAKCONV\Gamma$. To show that in fact
$W_{n}\CUTDISTCONV\Gamma$, we can mimic the proof of Theorem~3.5
(b)$\implies$(a) from~\cite{DGHRR:ACCLIM}. 
\end{proof}
So, while the concepts of cut distance identifying graphon parameters
or orders do not bring any new tools compared to the structuredness
order, knowing that a particular parameter or order is cut distance
identifying allows calculations that are often more direct than working
with the structuredness order.

\subsubsection{Relation to quasirandomness\label{subsec:RelationToQuasirandomness}}

Recall that dense quasi-random finite graphs correspond to constant
graphons. Thus, the key question in the area of quasirandomness is
which graphon parameters can be used to characterize constant graphons.\footnote{Strictly speaking, only parameters that are continuous with respect
to the cut distance are relevant for characterizing sequences of quasi-random
graphs. Indeed, the assumption of continuity is used to transfer between
finite graphs and their limits. The two main parameters we treat below
\textemdash{} homomorphism densities $t(H,\cdot)$ and spectrum \textemdash{}
are indeed well-known to be cut distance continuous (see Theorems~11.3
and~11.53 in~\cite{Lovasz2012}). The parameter $\INT_{f}(\cdot)$
is not cut distance continuous, and hence does not admit such a transference.}

The Chung\textendash Graham\textendash Wilson Theorem~\cite{Chung1989},
a version of which we state below, provides the most classical parameters
whose minimizer in $\mathcal{G}_{p}$ is the constant-$p$ graphon.
\begin{thm}
\label{thm:CGW}Let $p\in[0,1]$. Then the constant-$p$ graphon is
the only graphon $U$ in the family $\mathcal{G}_{p}$ satisfying
any of the following conditions.
\begin{enumerate}[label=(\alph*)]
\item \label{enu:EvenCyclesCGW}We have $t(C_{2\ell},U)\le p^{2\ell}$
for a fixed $\ell\in\left\{ 2,3,4,\ldots\right\} $.
\item \label{enu:EigenvaluesCGW}The largest eigenvalue of $U$ is at most
$p$ and all other eigenvalues are zero. 
\end{enumerate}
\end{thm}

Such characterizations of quasirandomness fit very nicely our framework
of cut distance identifying graphon parameters. Indeed, constant graphons
are exactly the minimal elements in the structuredness order; we refer
to~\cite[Proposition 8.5]{DGHRR:ACCLIM} for an easy proof. Thus,
each cut distance identifying graphon parameter can be used to characterize
constant graphons.

In the opposite direction, we show in Sections~\ref{subsec:Spectrum}
and~\ref{subsec:SubgraphDensities} that the graphon parameters considered
in Theorem~\ref{thm:CGW} are actually cut distance identifying.
Such a strengthening is not automatic (even for reasonable graphon
parameters); for example the parameter $t(C_{4}^{+},\cdot)$ (here,
$C_{4}^{+}$ is a 4-cycle with a pendant edge) is shown in~\cite[Section 2]{KMPW:StepSidorenko}
to be minimized on constant graphons but not to be cut distance identifying.\footnote{See Remark~\ref{rem:stepSidorenkoAreQuiteRestricted} for a more
general result.}

\subsubsection{Uniformity of cut distance identifying graphon parameters\label{subsec:Uniformity-CDI}}

If $\theta$ is a cut distance identifying graphon parameter and $U\prec W$
are two graphons then we know that $\theta(U)<\theta(W)$. In Proposition~\ref{prop:UniformityCDIGP}
below we prove that this relation can be made uniform (if $\theta$
is assumed to be cut distance continuous). That is, if $\cutDIST(U,W)\ge\varepsilon$
then $\theta(U)\le\theta(W)-b_{\varepsilon}$, where $b_{\varepsilon}>0$
depends only on $\theta$ and $\varepsilon$.

We shall make use of Proposition~\ref{prop:UniformityCDIGP} in Section~\ref{subsec:IndexPumping}.
\begin{prop}
\label{prop:UniformityCDIGP}Suppose that $\theta:\GRAPHONSPACE\rightarrow\mathbb{R}$
is an arbitrary cut distance identifying graphon parameter that is
continuous with respect to the cut distance. For every $\varepsilon>0$
there exists a $b_{\varepsilon}>0$ such that the following holds.
Suppose that $U,W:\Omega^{2}\rightarrow[0,1]$ are graphons such that
$U\preceq W$ and $\cutDIST(U,W)\ge\varepsilon$. Then $\theta\left(U\right)\le\theta\left(W\right)-b_{\varepsilon}$.
\end{prop}

\begin{proof}
Suppose that the claim fails for some $\varepsilon>0$. That is, for
each $n\in\mathbb{N}$, there exist graphons $U_{n},W_{n}:\Omega^{2}\rightarrow[0,1]$,
$U_{n}\preceq W_{n}$, 
\begin{equation}
\cutDIST\left(U_{n},W_{n}\right)\ge\varepsilon\;,\label{eq:hallo}
\end{equation}
and yet
\begin{equation}
\theta\left(W_{n}\right)\le\theta\left(U_{n}\right)+\frac{1}{n}\;.\label{eq:thetaGapsSmaller}
\end{equation}

As the square of the metric space $\left(\GRAPHONSPACE,\cutDIST\right)$
is compact, there exists a pair $(U,W)$ of graphons and a sequence
$i_{1}<i_{2}<i_{3}<\ldots$ so that $U_{i_{\ell}}\CUTDISTCONV U$,
$W_{i_{\ell}}\CUTDISTCONV W$. By the continuity of $\theta$, we
get from~(\ref{eq:thetaGapsSmaller}) that $\theta(W)\le\theta(U)$.
Also, by~(\ref{eq:hallo}) we get that 
\begin{equation}
\cutDIST\left(U,W\right)\ge\varepsilon>0\;.\label{eq:ABpos-1}
\end{equation}
Further, using Fact~\ref{fact:StructerdnessClosed}, we infer that
\begin{equation}
U\preceq W\;.\label{eq:UleW}
\end{equation}
Combined with~(\ref{eq:ABpos-1}), we get that $U\prec W$. Since
$\theta$ is cut distance identifying, we should have $\theta(U)<\theta(W)$,
a contradiction.
\end{proof}

\subsection{Using cut distance identifying graphon parameters for index-pumping\label{subsec:IndexPumping}}

In this section, we show that any cut distance identifying graphon
parameter that is continuous with respect to the cut distance can
replace the <<index>>, in the Frieze\textendash Kannan regularity
lemma. In particular, by Theorem~\ref{thm:norming->step} below,
any norming graph can be used for index-pumping. 

We state a graphon version of the Frieze\textendash Kannan regularity
lemma~\cite{Frieze1999} in Theorem~\ref{thm:FriezeKannaRL} below. 
\begin{thm}[{\cite[Corollary 9.13]{Lovasz2012}}]
\label{thm:FriezeKannaRL}For every $\varepsilon>0$ there exists
a number $M\in\mathbb{N}$ so that for each graphon $W:\Omega^{2}\rightarrow[0,1]$
there exists a partition $\mathcal{P}$ of $\Omega$ with at most
$M$ parts so that $\|W-W^{\Join\mathcal{P}}\|_{\square}\le\varepsilon$.
\end{thm}

The number $M$ in Theorem~\ref{thm:FriezeKannaRL} can be taken
as $M=2^{O(1/\varepsilon^{2})}$ and this is essentially optimal,~\cite{MR3811511}.
Let us recall the main steps of the proof of Theorem~\ref{thm:FriezeKannaRL}.
\begin{enumerate}[{label=[FK\arabic*]}]
\item \label{enu:TrivialPart}We start with the trivial partition $\mathcal{P}_{1}=\left\{ \Omega\right\} $.
\item At any given step $i=1,2,\ldots$, if $\left\Vert W-W^{\Join\mathcal{P}_{i}}\right\Vert _{\square}\le\varepsilon$,
then we output the partition $\mathcal{P}_{i}$, but \dots{}
\item \label{enu:IterationRL}\dots{} if $\left\Vert W-W^{\Join\mathcal{P}_{i}}\right\Vert _{\square}>\varepsilon$
let us take a set $X\subset\Omega$ which is a witness for this (c.f.
Lemma~\ref{lem:DisjointWitness}), that is, $\left|\int_{X\times X}(W-W^{\Join\mathcal{P}_{i}})\right|>\frac{\varepsilon}{4}$.
The so-called index pumping lemma asserts that defining a new partition
$\mathcal{P}_{i+1}:=\left\{ C\cap X,C\setminus X:C\in\mathcal{P}_{i}\right\} $
we have $\INT_{x\mapsto x^{2}}(W^{\Join\mathcal{P}_{i+1}})>\INT_{x\mapsto x^{2}}(W^{\Join\mathcal{P}_{i}})+c_{\varepsilon}$,
where $c_{\varepsilon}>0$ depends on $\varepsilon$ only.
\item \label{enu:BoundedIter}Since the mapping $\INT_{x\mapsto x^{2}}(\cdot)$
takes values in the interval $[0,1]$, we conclude that the above
iteration in~\ref{enu:IterationRL} cannot occur more than $\frac{1}{c_{\varepsilon}}$-many
times. Since $\left|\mathcal{P}_{i+1}\right|\le2\left|\mathcal{P}_{i}\right|$,
we conclude that the theorem holds with $M:=2^{\left\lceil 1/c_{\varepsilon}\right\rceil }$.
\end{enumerate}
Our approach is as follows, in the first step, we replace the index
$\INT_{x\mapsto x^{2}}(\cdot)$ by the $C_{4}$-density, and in the
second step, using Proposition~\ref{prop:UniformityCDIGP}, we obtain
a general result for any cut distance identifying graphon parameter
$\theta$ that is continuous with respect to the cut distance.\footnote{Note that a tempting shortcut in which we would deduce the pumping-up
property of $\theta$ directly from the pumping-up property of $\INT_{x\mapsto x^{2}}(\cdot)$
does not work. The reason for this is that $\INT_{x\mapsto x^{2}}(\cdot)$
is not cut distance continuous.} Let us state the result about the $C_{4}$-density first.
\begin{prop}
\label{prop:pumpingC4}Suppose that $\varepsilon>0$, $W:\Omega^{2}\rightarrow[0,1]$
is a graphon, $\mathcal{P}$ is a finite partition of $\Omega$, and
$X\subset\Omega$ is such that
\begin{equation}
\left|\int_{X\times X}(W-W^{\Join\mathcal{P}})\right|>\varepsilon\;.\label{eq:Ass314}
\end{equation}
Define $\mathcal{P}^{*}:=\left\{ C\cap X,C\setminus X:C\in\mathcal{P}\right\} $.
Then $t\left(C_{4},W^{\Join\mathcal{P^{*}}}\right)>t\left(C_{4},W^{\Join\mathcal{P}}\right)+\frac{\varepsilon^{4}}{100}$.
\end{prop}

Our proof of Proposition~\ref{prop:pumpingC4} is based on an extension
of an auxiliary but technical result from~\cite{CoKrMa:FinitelyForcibleUniversal},
which we now state.
\begin{lem}[Lemma 11 in \cite{CoKrMa:FinitelyForcibleUniversal}]
\label{lem:KralLemma11}Suppose that $Q$ and $R$ are step graphons
with respect to equipartitions $\mathcal{Q}$ and $\mathcal{R}$,
respectively. Suppose further that $\mathcal{Q}$ refines $\mathcal{R}$
and that $R=Q^{\Join\mathcal{R}}$. Then $t\left(C_{4},Q\right)\ge t\left(C_{4},R\right)+\frac{\left\Vert Q-R\right\Vert _{\square}^{4}}{8}$.
\end{lem}

\begin{proof}[Proof of Proposition~\ref{prop:pumpingC4}]
Suppose first that there is an equipartition partition $\mathcal{Q}=\left\{ Q_{1},Q_{2},\ldots,Q_{k}\right\} $
of $[0,1]$ that refines $\mathcal{P}^{*}$ and $N\in\mathbb{N}$
such that $\left|I_{A}\right|$ is a multiple of $N$ for every $A\in\mathcal{P}^{*}$
where $I_{A}=\left\{ i\in\left[k\right]:Q_{i}\subseteq A\right\} $.
Note that if such an equipartition $\mathcal{Q}$ and $N\in\mathbb{N}$
exists, then we may assume that $N$ is arbitrarily big and $\sqrt{N}\in\mathbb{N}$.
Let $k_{A}=\frac{\left|I_{A}\right|}{\sqrt{N}}$ and partition each
$I_{A}$ into $I_{A,1}\dots,I_{A,k_{A}}$ with $\sqrt{N}$-elements
each as in Lemma~\ref{lem:concentrationrandomgrouping}. Denote as
$\mathcal{D}$ the partition of $\Omega$ with pieces $\bigcup I_{A,i}$
where $A\in\mathcal{P}^{*}$ and $i\in\left[k_{A}\right]$. Note that
$\mathcal{D}$ is an equipartition that refines $\mathcal{P^{*}}$.
Then we have
\begin{align}
\left\Vert W^{\Join\mathcal{P}^{*}}-W^{\Join\mathcal{D}}\right\Vert _{1} & \le2N^{-\frac{1}{8}}\label{eq:cernovice1}
\end{align}
by Lemma~\ref{lem:concentrationrandomgrouping}. For each $B\in\mathcal{P}$
we denote as $A_{1},A_{2}\in\mathcal{P}^{*}$ the unique elements
such that $B=A_{1}\cup A_{2}$. We define $J_{B}=\left\{ X\in\mathcal{D}:X\subseteq A_{1}\!\vee\!X\subseteq A_{2}\right\} $.
Note that it follows from the assumption on $N$ that for the number
$r_{B}:=\frac{\left|J_{B}\right|}{\sqrt{N}}$ we have $r_{B}\in\mathbb{N}$.
Partition each $J_{B}$ into $J_{B,1}\dots,J_{B,r_{B}}$ groups with
$\sqrt{N}$-elements each as in Lemma~\ref{lem:concentrationrandomgrouping}
and define $\mathcal{T}$ as the partition of $\Omega$ with pieces
$\bigcup J_{B,i}$ where $B\in\mathcal{P}$ and $i\in\left[r_{B}\right]$.
We have
\begin{align}
\left\Vert W^{\Join\mathcal{P}}-W^{\Join\mathcal{T}}\right\Vert _{1} & \le2N^{-\frac{1}{8}}\label{eq:cernovice2}
\end{align}
by Lemma~\ref{lem:concentrationrandomgrouping}. Observe that $\mathcal{D}$
is an equipartition that refines $\mathcal{T}$. Using~(\ref{eq:cernovice1}),~(\ref{eq:cernovice2})
and the the fact that $\|W^{\Join\mathcal{P}}-W^{\Join\mathcal{P^{*}}}\|_{\square}>\varepsilon$
(by~\ref{eq:Ass314}) we get
\begin{align*}
\left\Vert W^{\Join\mathcal{D}}-W^{\Join\mathcal{T}}\right\Vert _{\square} & >\varepsilon-4N^{-\frac{1}{8}}.
\end{align*}
By Lemma~\ref{lem:KralLemma11}, we have $t\left(C_{4},W^{\Join\mathcal{D}}\right)\ge t\left(C_{4},W^{\Join\mathcal{T}}\right)+\frac{\left(\varepsilon-4N^{-\frac{1}{8}}\right)^{4}}{8}$.
Further, by the Lemma~\ref{lem:countinglemma} (using~(\ref{eq:cernovice1})
and~(\ref{eq:cernovice2})), we have $t\left(C_{4},W^{\Join\mathcal{D}}\right)=t\left(C_{4},W^{\Join\mathcal{P}^{*}}\right)\pm32N^{-\frac{1}{8}}$
and $t\left(C_{4},W^{\Join\mathcal{T}}\right)=t\left(C_{4},W^{\Join\mathcal{P}}\right)\pm32N^{-\frac{1}{8}}$.
Taking $N$ large enough finishes the proof in the special case.

In the general case we assign to each $A\in\mathcal{P}^{*}$ a measurable
set $A'$ such that $\lambda\left(A'\right)\in\mathbb{Q}$, $\lambda\left(A\diamond A'\right)$
is arbitrarily small and the collection $\mathcal{R}=\left\{ A':A\in\mathcal{P}^{*}\right\} $
is a partition of $\Omega$. We use $\mathcal{R}$ in an obvious way
to build $\mathcal{S}$ that approximate $\mathcal{P}$, i.e., if
$B\in\mathcal{P}$ and $A_{1}\cup A_{2}=B$ where $A_{1},A_{2}\in\mathcal{P^{*}}$,
then define $B'=A'_{1}\cup A'_{2}$. It is easy to see that since
$W$ is a bounded function we can always find $\mathcal{R}$ such
that $\left\Vert W^{\Join\mathcal{P^{*}}}-W^{\Join\mathcal{R}}\right\Vert _{1}$
and $\left\Vert W^{\Join\mathcal{P}}-W^{\Join\mathcal{S}}\right\Vert _{1}$
are arbitrary small. The rest is an easy application of the triangle
inequality.
\end{proof}
We now show how to extend Proposition~\ref{prop:pumpingC4} to all
continuous cut distance identifying graphon parameters.
\begin{prop}
\label{prop:pumpingCDIGP}Suppose that $\theta:\GRAPHONSPACE\rightarrow\mathbb{R}$
is a cut distance identifying graphon parameter that is continuous
with respect to the cut distance. For every $\varepsilon>0$ there
exists a $b_{\varepsilon}>0$ such that the following holds. Suppose
that $W:\Omega^{2}\rightarrow[0,1]$ is a graphon, $\mathcal{P}$
is a finite partition of $\Omega$, and $X\subset\Omega$ is such
that $\left|\int_{X\times X}(W-W^{\Join\mathcal{P}})\right|>\frac{\varepsilon}{4}$.
Define $\mathcal{P}^{*}:=\left\{ C\cap X,C\setminus X:C\in\mathcal{P}\right\} $.
Then $\theta\left(W^{\Join\mathcal{P^{*}}}\right)>\theta\left(W^{\Join\mathcal{P}}\right)+b_{\varepsilon}$.
\end{prop}

\begin{proof}
By Proposition~\ref{prop:pumpingC4}, we have $t\left(C_{4},W^{\Join\mathcal{P^{*}}}\right)>t\left(C_{4},W^{\Join\mathcal{P}}\right)+\frac{\varepsilon^{4}}{25600}$.
By Lemma~\ref{lem:countinglemma}, we have $\cutDIST\left(W^{\Join\mathcal{P^{*}}},W^{\Join\mathcal{P}}\right)\ge\frac{\varepsilon^{4}}{409600}$.
Finally, Proposition~\ref{prop:UniformityCDIGP} gives $\theta\left(W^{\Join\mathcal{P^{*}}}\right)>\theta\left(W^{\Join\mathcal{P}}\right)+b_{\varepsilon}$,
for some $b_{\varepsilon}$ that depends only on $\varepsilon$ and
$\theta$.
\end{proof}
\begin{rem}
We would like to emphasize that in this section we showed that any
continuous cut distance identitifying graphon parameter has a similar
<<pumping property>> as the index, but did not obtain any new self-contained
proof of the Frieze\textendash Kannan regularity lemma.
\begin{itemize}
\item Firstly, for our proof, we need to borrow Lemma~\ref{prop:pumpingC4}
which readily says that \emph{some parameter }($t(C_{4},\cdot)$,
in this case) has the pumping property, and the existence of any one
such parameter already allows to run the proof scheme~\ref{enu:TrivialPart}-\ref{enu:BoundedIter}.
This step was needed to infer Proposition~\ref{prop:pumpingCDIGP},
and it would be interesting to have a direct argument for this. 
\item Secondly, we used the compactness of the space $\left(\GRAPHONSPACE,\cutDIST\right)$,
which is actually known to be equivalent to the Frieze\textendash Kannan
regularity lemma,~\cite{Lovasz2007}.
\end{itemize}
\end{rem}

We believe that the same setting can be used in the setting of the
Szemerédi regularity lemma. We pose this as a problem.
\begin{conjecture}
\label{conj:SzeIndexPumping}Each cut distance identifying graphon
parameter that is continuous with respect to the cut distance can
be used as an <<index>> in the Szemerédi regularity lemma.
\end{conjecture}

The difficulty here is to provide a counterpart to Proposition~\ref{prop:pumpingCDIGP}
in the setting of the Szemerédi regularity lemma. That is, (without
explaining all the notation) we do not have a single set $X$ witnessing
large cut norm but rather many witnesses of irregularity on individual
pairs of clusters, none of them being substantial in the global sense
of the cut norm.

\subsection{Revising the parameter $\protect\INT_{f}\left(\cdot\right)$\label{subsec:RevisingINT}}

Recall that in~\cite{DH:WeakStar}, the parameter $\INT_{f}\left(\cdot\right)$
(for a strictly convex continuous function $f:\left[0,1\right]\rightarrow\mathbb{R}$)
was used to identify cut distance limits of sequences of graphons
(thus providing a new proof of Theorem~\ref{thm:compactness}). One
of the key steps in~\cite{DH:WeakStar} was to show that a certain
refinement of a graphon leads to an increase of $\INT_{f}\left(\cdot\right)$.
While not approached this way in~\cite{DH:WeakStar}, this hints
that $\INT_{f}\left(\cdot\right)$ is cut distance identifying. We
prove this statement in the current section, as a quick application
of the results from \cite[Section 4.4]{DGHRR:ACCLIM}. Also, here
we show that the requirement of continuity of $f$ was just an artifact
of the proof in~\cite{DH:WeakStar}.
\begin{thm}
\label{thm:INTdiscontinuous}
\begin{enumerate}[label=(\alph*)]
\item \label{enu:nonstrictlyvonv}Suppose that $f:\left[0,1\right]\rightarrow\mathbb{R}$
is a convex function. Then $\INT_{f}\left(\cdot\right)$ is cut distance
compatible.
\item \label{enu:strictlyvonv}Suppose that $f:\left[0,1\right]\rightarrow\mathbb{R}$
is a strictly convex function. Then $\INT_{f}\left(\cdot\right)$
is cut distance identifying.
\end{enumerate}
\end{thm}

\begin{proof}[Proof of Part~\ref{enu:nonstrictlyvonv}]
Recall that every convex function admits left and right derivatives
which are both increasing functions. The key is to observe that for
a graphon $\Gamma$, we have $\INT_{f}\left(\Gamma\right)=\int_{x\in\left[0,1\right]}f(x)\:\mathrm{d}\boldsymbol{\Phi}_{\Gamma}$,
where $\boldsymbol{\Phi}_{\Gamma}$ is defined by~(\ref{eq:pushforwardValues}).
Suppose that $U\preceq W$. By Proposition~\ref{prop:flatter}, we
have that $\boldsymbol{\Phi}_{U}$ is at least as flat as $\boldsymbol{\Phi}_{W}$.
Let $\Lambda$ be a measure on $\left[0,1\right]^{2}$ as in Definition~\ref{def:flatter}
that witnesses this fact. If $\Lambda$ is carried by the diagonal
of $[0,1]^{2}$ then $\boldsymbol{\Phi}_{U}=\boldsymbol{\Phi}_{W}$.
In that case $U\nprec W$ by Proposition~\ref{prop:flatter}. In
other words, $\left\langle U\right\rangle =\left\langle W\right\rangle $.
By Fact~\ref{fact:StrictlyBelowCutDistPos}, we have $\delta_{\square}(U,W)=0$.
Since $\theta$ is a graphon parameter, we conclude that $\theta(U)=\theta(W)$.
So it remains to consider the case when $\Lambda$ is not carried
by the diagonal. Then there are intervals $[a,b],[c,d]\subset[0,1]$
with $\Lambda\left([a,b]\times[c,d]\right)>0$ and $b<c$ (the other
case when $d<a$ is similar). 

Fix $\varepsilon>0$ and note that $f$ is continuous on the open
interval $(0,1)$ by convexity, thus the points 0 and 1 are the only
possible points of discontinuity of $f$. So for every $x\in(0,1)$
there is an interval $J_{x}\subset(0,1)$ containing $x$ such that
every two values of $f$ on $J_{x}$ differ by at most $\varepsilon$.
Take a covering of $(0,1)$ consisting of at most countably many such
intervals, add the singletons $\left\{ 0\right\} $ and $\left\{ 1\right\} $,
and then refine the resulting family to a countable disjoint covering
$\left\{ J_{1},J_{2},\ldots\right\} $ of $[0,1]$. Then for every
$i$ and for every $x\in J_{i}$ we have $|f(x)-f(x_{i})|\le\varepsilon$
where $x_{i}$ is the $\boldsymbol{\Phi}_{U}$-mean value of $x$
on $J_{i}$, i.e., (by~(\ref{eq:masspreserve}))
\begin{equation}
x_{i}=\frac{1}{\boldsymbol{\Phi}_{U}(J_{i})}\int_{J_{i}}x\;\mathrm{d}\boldsymbol{\Phi}_{U}=\frac{1}{\Lambda(J_{i}\times[0,1])}\int_{J_{i}\times[0,1]}x\;\mathrm{d}\Lambda=\frac{1}{\Lambda(J_{i}\times[0,1])}\int_{J_{i}\times[0,1]}y\;\mathrm{d}\Lambda\label{eq:strednihodnota}
\end{equation}
(if for some $i$ we have $\Phi_{U}(J_{i})=0$ then we can define
$x_{i}$ to be an arbitrary element of $J_{i}$). We may moreover
assume that for every $i$ either $J_{i}\subset[a,b]$ or $J_{i}\cap[a,b]=\emptyset$,
then $x_{i}\in[a,b]$ whenever $J_{i}\subset[a,b]$. Note that convexity
of $f$ implies that 
\begin{equation}
f(y)\ge f'_{+}(x_{i})\cdot y+(f(x_{i})-f'_{+}(x_{i})\cdot x_{i})\label{eq:deltaodhad2}
\end{equation}
for every $y\in[c,d]$ and every $i$ with $J_{i}\subset[a,b]$.

We have
\begin{align*}
\INT_{f}\left(U\right) & =\int_{x\in[0,1]}f(x)\:\mathrm{d}\boldsymbol{\Phi}_{U}=\sum_{i}\int_{x\in J_{i}}f(x)\:\mathrm{d}\boldsymbol{\Phi}_{U}\stackrel{\varepsilon}{\approx}\sum_{i}f(x_{i})\boldsymbol{\Phi}_{U}(J_{i})=\sum_{i}f(x_{i})\Lambda(J_{i}\times[0,1])\;.
\end{align*}

We continue by employing Jensen's inequality and~(\ref{eq:strednihodnota}),

\begin{align*}
\sum_{i}f(x_{i})\Lambda(J_{i}\times[0,1]) & \le\sum_{i}\int_{(x,y)\in J_{i}\times[0,1]}f(y)\,\mathrm{d}\Lambda=\int_{(x,y)\in[0,1]^{2}}f(y)\,\mathrm{d}\Lambda\\
 & =\int_{y\in[0,1]}f(y)\:\mathrm{d}\boldsymbol{\Phi}_{W}=\INT_{f}\left(W\right)\;.
\end{align*}

As this is true for every $\varepsilon>0$ we conclude that $\INT_{f}\left(U\right)\le\INT_{f}\left(W\right)$.

\emph{Proof of Part~\ref{enu:strictlyvonv}.} Suppose that $U\prec W$
(then $\boldsymbol{\Phi}_{U}$ is strictly flatter than $\boldsymbol{\Phi}_{W}$,
and so the witnessing measure $\Lambda$ cannot be carried by the
diagonal of $[0,1]^{2}$). In that case both one-sided derivatives
of $f$ are strictly increasing, and so it is easy to see that there
is $\delta>0$ such that Equation~(\ref{eq:deltaodhad2}) holds in
the stronger form

\begin{equation}
f(y)\ge f'_{+}(x_{i})\cdot y+(f(x_{i})-f'_{+}(x_{i})\cdot x_{i})+\delta\label{eq:deltaodhad2-2}
\end{equation}
for every $y\in[c,d]$ and every $i$ with $J_{i}\subset[a,b]$. We
show that then the application of Jensen's inequality above ensures
that $\INT_{f}\left(U\right)<\INT_{f}\left(W\right)$. To this end
it suffices to show that there is a constant $K>0$ not depending
on $\varepsilon$ such that

\[
\sum_{i\colon J_{i}\subset[a,b]}f(x_{i})\Lambda(J_{i}\times[0,1])\le\sum_{i\colon J_{i}\subset[a,b]}\int_{(x,y)\in J_{i}\times[0,1]}f(y)\,\mathrm{d}\Lambda-K\;.
\]
For every $i$ denote $g_{i}(y):=f'_{+}(x_{i})\cdot y+(f(x_{i})-f'_{+}(x_{i})\cdot x_{i})$.
Then we have
\begin{eqnarray*}
 &  & \sum_{i\colon J_{i}\subset[a,b]}\int_{(x,y)\in J_{i}\times[0,1]}f(y)\,\mathrm{d}\Lambda\\
 & = & \sum_{i\colon J_{i}\subset[a,b]}\int_{(x,y)\in J_{i}\times[c,d]}f(y)\,\mathrm{d}\Lambda+\sum_{i\colon J_{i}\subset[a,b]}\int_{(x,y)\in J_{i}\times([0,1]\setminus[c,d])}f(y)\,\mathrm{d}\Lambda\\
\JUSTIFY{(\ref{eq:deltaodhad2-2})\ and\ convexity} & \ge & \sum_{i\colon J_{i}\subset[a,b]}\int_{(x,y)\in J_{i}\times[c,d]}(g_{i}(y)+\delta)\,\mathrm{d}\Lambda+\sum_{i\colon J_{i}\subset[a,b]}\int_{(x,y)\in J_{i}\times([0,1]\setminus[c,d])}g_{i}(y)\,\mathrm{d}\Lambda\\
 & = & \sum_{i\colon J_{i}\subset[a,b]}\int_{(x,y)\in J_{i}\times[0,1]}g_{i}(y)\,\mathrm{d}\Lambda+\delta\cdot\Lambda([a,b]\times[c,d])\\
 & \substack{\stackrel{\eqref{eq:strednihodnota}}{=}}
 & \sum_{i\colon J_{i}\subset[a,b]}f(x_{i})\Lambda(J_{i}\times[0,1])+\delta\cdot\Lambda([a,b]\times[c,d])\;.
\end{eqnarray*}
So it suffices to set $K:=\delta\cdot\Lambda([a,b]\times[c,d])$.
\end{proof}
For a later reference, let us apply Theorem~\ref{thm:INTdiscontinuous}
to the strictly convex function $x\mapsto x^{2}$, for which $\INT_{x\mapsto x^{2}}(\cdot)=\left\Vert \cdot\right\Vert _{2}^{2}$.
\begin{cor}
\label{cor:L2CDIP}Suppose that $U$ and $W$ are two graphons with
$U\prec W$. Then $\left\Vert U\right\Vert _{2}<\left\Vert W\right\Vert _{2}$.
\end{cor}

\subsection{Convex graphon parameters\label{subsec:convexity_and_structurdness}}

In Definition~\ref{def:convex} we introduce convex graphon parameters.
In Theorem~\ref{thm:convex_functions_are_compatible} we prove that
such parameters are cut distance compatible if they are also $L^{1}$-continuous.
In Example~\ref{exa:compatiblenotconvex} we observe that the opposite
implication is not true.
\begin{defn}
\label{def:convex}A graphon parameter $g:\GRAPHONSPACE\rightarrow\mathbb{R}$
is \emph{convex }if for every $\alpha_{1},\alpha_{2},\alpha_{3},\ldots,\alpha_{k}\in[0,1]$
with $\sum_{i}\alpha_{i}=1$ and graphons $W,W_{1},W_{2},\ldots,W_{k}\in\GRAPHONSPACE$
with $W=\sum_{i}\alpha_{i}W_{i}$ we have $f(W)\le\sum_{i}\alpha_{i}g(W_{i})$.
\end{defn}

\begin{thm}
\label{thm:convex_functions_are_compatible}Let $g:\GRAPHONSPACE\rightarrow\mathbb{R}$
be a graphon parameter that is convex and continuous in $L^{1}$.
Then $g$ is cut distance compatible.
\end{thm}

Theorem~\ref{thm:convex_functions_are_compatible} can be used to
give a third proof of a weaker version of the first part of Theorem~\ref{thm:INTdiscontinuous},
in which \textemdash{} just like the version in~\cite{DH:WeakStar}
\textemdash{} it is needed to require that $f:\left[0,1\right]\rightarrow\mathbb{R}$
is continuous. Indeed, the continuity of $f$ easily implies that
the graphon parameter $\INT_{f}$ is continuous in $L^{1}$, and the
convexity of $\INT_{f}$ is also clear.

\medskip{}

Now we prove Theorem~\ref{thm:convex_functions_are_compatible}. 
\begin{proof}[Proof of Theorem~\ref{thm:convex_functions_are_compatible}]
Suppose that $U,V:\Omega^{2}\rightarrow\left[0,1\right]$ are arbitrary
graphons such that $V\prec U$. Suppose that $\varepsilon>0$ is arbitrary.
Let $N(\varepsilon)\in\mathbb{N}$ and $(\phi_{\varepsilon,i})_{i=1}^{N}$
satisfy~(\ref{eq:LampaSviti}) for $U,V$ and error $\varepsilon$
(we will not use the feature~(\ref{eq:farapartmycka}) in this application
of Lemma~\ref{lem:L1approxByVersions}). For every $i\in[N]$ we
denote the version $U^{\phi_{\varepsilon,i}}$ of $U$ by $U_{\varepsilon,i}$.
Then we have
\begin{align}
g\left(V\right) & =g\left(\sum_{i=1}^{N(\varepsilon)}\frac{1}{N(\varepsilon)}U_{\varepsilon,i}\right)+\left(g\left(V\right)-g\left(\sum_{i=1}^{N(\varepsilon)}\frac{1}{N(\varepsilon)}U_{\varepsilon,i}\right)\right)\nonumber \\
\JUSTIFY{convexity} & \le\sum_{i=1}^{N(\varepsilon)}\frac{1}{N(\varepsilon)}g\left(U_{\varepsilon,i}\right)+\left(g\left(V\right)-g\left(\sum_{i=1}^{N(\varepsilon)}\frac{1}{N(\varepsilon)}U_{\varepsilon,i}\right)\right)\nonumber \\
\JUSTIFY{\text{\ensuremath{g\left(U_{\varepsilon,i}\right)=g(U)}}} & =g\left(U\right)+\left(g\left(V\right)-g\left(\sum_{i=1}^{N(\varepsilon)}\frac{1}{N(\varepsilon)}U_{\varepsilon,i}\right)\right)\;.\label{eq:DianaKrall}
\end{align}
Now, as $\varepsilon$ goes to~0, the graphon $\sum_{i=1}^{N(\varepsilon)}\frac{1}{N(\varepsilon)}U_{\varepsilon,i}$
goes to $V$ in $L^{1}(\Omega^{2})$. Thus, the $L^{1}$-continuity
of $g$ tells us that the last term in~(\ref{eq:DianaKrall}) vanishes,
and thus $g\left(V\right)\le g(U)$. Thus $g$ is cut distance compatible.
\end{proof}
\begin{example}
\label{exa:compatiblenotconvex}In this example we first construct
two graphons $U$ and $V$ such that $V$ is a convex combination
of versions of $U$ but $V\not\preceq U$. We then use this to construct
a cut distance compatible graphon parameter $f^{*}$ that is not convex.
The graphons $U$ and $V$ are shown in Figure~\ref{fig:convex}.
The graphon $U$ is defined as $U\left(x,y\right)=1$ if and only
if $\left(x,y\right)\in[0,\frac{1}{2}]^{2}$ and $U\left(x,y\right)=0$
otherwise, while $V\left(x,y\right)=\frac{1}{2}$ if and only if $\left(x,y\right)\in[0,\frac{1}{2}]^{2}\cup[\frac{1}{2},1]^{2}$
and $V\left(x,y\right)=0$ otherwise. If we set $\varphi\left(x\right)=1-x$,
then clearly $V=\frac{U+U^{\varphi}}{2}$. Let us now argue that $V\not\preceq U$.
For any measure preserving bijection $\pi$ we have 
\[
\int_{[0,\frac{1}{2}]\times[\frac{1}{2},1]}U^{\pi}=\nu\left(\pi\left([0,\frac{1}{2}]\right)\cap[0,\frac{1}{2}]\right)\cdot\nu\left(\pi\left([0,\frac{1}{2}]\right)\cap[\frac{1}{2},1]\right)\;.
\]
Thus, for any sequence of measure preserving bijections $\pi_{1},\pi_{2},\dots$
such that $U^{\pi_{n}}\stackrel{w*}{\rightarrow}V$ we have (after
passing to a subsequence if necessary) either 
\[
\nu\left(\pi_{n}\left([0,\frac{1}{2}]\right)\cap[0,\frac{1}{2}]\right)\rightarrow0
\]
 or 
\[
\nu\left(\pi_{n}\left([0,\frac{1}{2}]\right)\cap[0,\frac{1}{2}]\right)\rightarrow\frac{1}{2}\;.
\]
This is clearly a contradiction.

Now, take any cut distance compatible parameter $f$ and suppose that
it is convex. In particular, we have that $\frac{1}{2}f\left(U\right)+\frac{1}{2}f\left(U^{\varphi}\right)\ge f\left(V\right)$
for the two graphons $U$ and $V$ defined above. We can now define
\[
f^{*}\left(W\right)=f\left(W\right)+\left(\frac{1}{2}f\left(U\right)+\frac{1}{2}f\left(U^{\varphi}\right)-f\left(V\right)+1\right)
\]
 for each graphon $W$ such that $W\succeq V$ and 
\[
f^{*}(W)=f(W)
\]
otherwise. The graphon parameter $f^{*}$ is clearly cut distance
compatible, but no longer convex, since 
\[
f^{*}\left(V\right)=\frac{1}{2}f\left(U\right)+\frac{1}{2}f\left(U^{\varphi}\right)+1>\frac{1}{2}f^{*}(U)+\frac{1}{2}f^{*}(U^{\varphi})\;.
\]

This example works even if we restrict ourselves to graphons lying
in the envelope of a certain fixed graphon $W$, since if we set $W\left(x,y\right)=1$
if and only if $\left(x,y\right)\in[0,\frac{1}{4}]^{2}\cup[\frac{1}{4},\frac{1}{2}]^{2}$
and $W(x,y)=0$ otherwise, and set $U'=\frac{U}{2},V'=\frac{V}{2}$,
then we have three graphons $U',V',W$ such that $U',V'\preceq W$,
$V'=\frac{U'+{U'}^{\varphi}}{2}$, but $V'\not\preceq U'$. 
\end{example}

\begin{figure}
\includegraphics[width=0.5\textwidth]{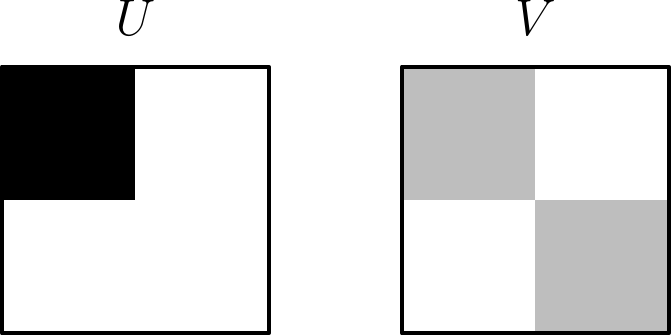}

\caption{Graphons $U$ and $V$ from Example~\ref{exa:compatiblenotconvex}}
\label{fig:convex}
\end{figure}

The function $f^{*}$ from Example~\ref{exa:compatiblenotconvex}
is, however, very unnatural since it is not continuous with respect
to $L^{1}$ (for a continuous parameter $f$ at least). We leave it
as an open problem, whether there is a continuous example. 
\begin{problem}
\label{prob:convex_versus_compatible}Is there a graphon parameter
$f:\GRAPHONSPACE\rightarrow\mathbb{R}$ that is not convex, but is
continuous in $L^{1}$ and cut distance compatible?
\end{problem}

\begin{rem}
\label{rem:CDCimpliesConvexForDensities}For homomorphism densities
$t(H,\cdot):\GRAPHONSPACE\rightarrow\mathbb{R}$, Theorem~\ref{thm:convex_functions_are_compatible}
can be reversed, under the additional assumption that $H$ is a connected
graph: $t(H,\cdot)$ is cut distance compatible if and only if it
is convex. Let us give details of the direction not covered by Theorem~\ref{thm:convex_functions_are_compatible}.
Suppose that $t(H,\cdot)$ is cut distance compatible, and $H$ is
connected. By Theorem~\ref{thm:compatibility->norming} below, $H$
is weakly norming. In particular, for the function $f:\POSITIVEKERNELSPACE\rightarrow\mathbb{R}$,
$f(U):=t(H,U)^{1/e(H)}$ we have that $f(U_{1}+U_{2})\le f(U_{1})+f(U_{2})$.
Now, for every $\alpha_{1},\alpha_{2},\alpha_{3},\ldots,\alpha_{k}\in[0,1]$
with $\sum_{i}\alpha_{i}=1$ and every graphons $W,W_{1},W_{2},\ldots,W_{k}\in\GRAPHONSPACE$
with $W=\sum_{i}\alpha_{i}W_{i}$, we have
\begin{align*}
t\left(H,W\right) & =f\left(\sum_{i}\alpha_{i}W_{i}\right)^{e(H)}\le\left(\sum_{i}f\left(\alpha_{i}W_{i}\right)\right)^{e(H)}\\
 & =\left(\sum_{i}\alpha_{i}f\left(W_{i}\right)\right)^{e(H)}\le\sum_{i}\alpha_{i}\left(f\left(W_{i}\right)\right)^{e(H)}=\sum_{i}\alpha_{i}t\left(H,W_{i}\right)\;,
\end{align*}
as required.
\end{rem}

\subsection{Spectrum\label{subsec:Spectrum}}

The main result in this section, Theorem~\ref{thm:spectrum}, asserts
that the spectral quasiorder defined in Section~\ref{subsec:SpectrumAndSpectralQuasiorder}
is a cut distance identifying graphon order. But first we need an
easy lemma.
\begin{lem}
\label{lem:limitQuadratic}Let $\left(W_{n}\right)_{n}$ be a sequence
of graphons on $\Omega^{2}$ such that $W_{n}\WEAKCONV U$ for some
graphon $U$. Let $u,v\in L^{2}\left(\Omega\right)$. Then we have
$\left\langle W_{n}u,v\right\rangle \rightarrow\left\langle Uu,v\right\rangle $.
\end{lem}

\begin{proof}
Since step functions are dense in $L^{2}\left(\Omega\right)$, and
since the forms $\left\langle W_{n}\cdot,\cdot\right\rangle $ and
$\left\langle U\cdot,\cdot\right\rangle $ are obviously bilinear,
it suffices to prove the statement for indicator functions of sets,
$u=\mathbf{1}_{A}$, $v=\mathbf{1}_{B}$ (where $A,B\subset\Omega$).
But in that case $\left\langle W_{n}u,v\right\rangle =\int_{A\times B}W_{n}$
and $\left\langle Uu,v\right\rangle =\int_{A\times B}U$. The statement
follows since $W_{n}\WEAKCONV U$.
\end{proof}
We are now ready to prove the main result of this section. Let us
note that the arguments that we use to prove this result also turned
out to be useful in the setting of finitely forcible graphs; in particular
Krá\v{l}, Lovász, Noel, and Sosnovec~\cite{KrLoNoSo} used our arguments
in the final step of their proof that for each graphon and each $\varepsilon>0$,
there exists a finitely forcible graphon that differs from the original
one only on a set of measure at most $\varepsilon$.
\begin{thm}
\label{thm:spectrum}The spectral quasiorder is a cut distance identifying
graphon order. That is, given two graphons $U,W\in\GRAPHONSPACE$, 
\begin{enumerate}[label=(\alph*)]
\item \label{enu:vana1}if $\delta_{\square}(U,W)=0$, then the spectra
of $U$ and $W$ are the same, and
\item \label{enu:vana2}if $U\prec W$, then $U\SSO W$.
\end{enumerate}
\end{thm}

\begin{proof}
Part~\ref{enu:vana1} follows from~\cite[Theorem 11.54]{Lovasz2012}.

So, the main work is to prove~\ref{enu:vana2}. Consider the sequence
$(W^{\pi_{n}})_{n}$ of versions of $W$ such that $W^{\pi_{n}}\WEAKCONV U$.
Let $\lambda_{1}^{+}\ge\lambda_{2}^{+}\ge\lambda_{3}^{+}\ge\ldots\geq0$
be the positive eigenvalues of $U$ with associated pairwise orthogonal
unit eigenvectors $u_{1},u_{2},u_{3},\ldots$, and let $\beta_{1}^{+}\ge\beta_{2}^{+}\ge\beta_{3}^{+}\ge\ldots\geq0$
be the positive eigenvalues of $W$. First, we will prove that for
any given $\varepsilon>0$ and $k$, we have $\beta_{k}^{+}\geq\lambda_{k}^{+}-\varepsilon$.
By the maxmin characterization of eigenvalues we have 
\begin{equation}
\beta_{k}^{+}=\max_{\substack{H\text{ subspace of }L^{2}(\Omega)\\
\text{dim}(H)=k
}
}\;\min_{\substack{g\in H\\
\left\Vert g\right\Vert _{2}=1
}
}\left\langle Wg,g\right\rangle \;.\label{eq:beh}
\end{equation}
Given $n$, consider the space $\widetilde{H}_{n}=\text{span}\left\{ u_{1}^{\pi_{n}^{-1}},u_{2}^{\pi_{n}^{-1}},\ldots,u_{k}^{\pi_{n}^{-1}}\right\} $,
where $u_{i}^{\pi_{n}^{-1}}(x)=u_{i}(\pi_{n}^{-1}(x))$. Then~(\ref{eq:beh})
gives
\begin{equation}
\beta_{k}^{+}\geq\min_{\substack{g\in\widetilde{H}_{n}\\
\left\Vert g\right\Vert _{2}=1
}
}\left\langle Wg,g\right\rangle .\label{eq:minWgg}
\end{equation}
Furthermore, by Lemma~\ref{lem:limitQuadratic} we can find $n$
large enough so that for all $i,j\le k$ we have
\begin{equation}
\left|\left\langle W^{\pi_{n}}u_{i},u_{j}\right\rangle -\left\langle Uu_{i},u_{j}\right\rangle \right|<\frac{\varepsilon}{k^{2}}.\label{eq:sekacka}
\end{equation}
Now, for $g\in\widetilde{H}_{n}$ that realizes the minimum in~(\ref{eq:minWgg}),
we can write its orthogonal decomposition as $g=\sum_{i=1}^{k}c_{i}u_{i}^{\pi_{n}^{-1}}$,
where $\sum_{i=1}^{k}c_{i}^{2}=1$. Thus, we obtain
\begin{align*}
\left\langle Wg,g\right\rangle  & =\left\langle W^{\pi_{n}}g^{\pi_{n}},g^{\pi_{n}}\right\rangle =\left\langle W^{\pi_{n}}\sum_{i=1}^{k}c_{i}u_{i},\sum_{i=1}^{k}c_{i}u_{i}\right\rangle =\sum_{i,j=1}^{k}c_{i}c_{j}\left\langle W^{\pi_{n}}u_{i},u_{j}\right\rangle \;.
\end{align*}
We can now use~(\ref{eq:sekacka}) to replace the terms $\left\langle W^{\pi_{n}}u_{i},u_{j}\right\rangle $
by the terms $\left\langle Uu_{i},u_{j}\right\rangle $, 
\begin{align*}
\left\langle Wg,g\right\rangle  & =\sum_{i=1}^{k}c_{i}^{2}\left(\left\langle Uu_{i},u_{i}\right\rangle \pm\frac{\varepsilon}{k^{2}}\right)+\sum_{\substack{i,j=1\\
i\neq j
}
}^{k}c_{i}c_{j}\left(\left\langle Uu_{i},u_{j}\right\rangle \pm\frac{\varepsilon}{k^{2}}\right)\\
 & \ge\sum_{i=1}^{k}c_{i}^{2}\left(\lambda_{i}^{+}-\frac{\varepsilon}{k^{2}}\right)-\sum_{\substack{i,j=1\\
i\neq j
}
}^{k}|c_{i}c_{j}|\frac{\varepsilon}{k^{2}}\geq\lambda_{k}^{+}-\varepsilon\;.
\end{align*}
Thus~(\ref{eq:minWgg}) implies $\beta_{k}^{+}\geq\lambda_{k}^{+}-\varepsilon$.

A similar argument can be used for the negative eigenvalues $\lambda_{1}^{-}\leq\lambda_{2}^{-}\leq\lambda_{3}^{-}\leq\ldots\leq0$
of $U$ and $\beta_{1}^{-}\leq\beta_{2}^{-}\leq\beta_{3}^{-}\leq\ldots\leq0$
of $W$ to show that $\beta_{k}^{-}\leq\lambda_{k}^{-}+\varepsilon$.
That implies $U\SO W$.

To show that for at least one eigenvalue the corresponding inequality
is strict, assume for contradiction that the eigenvalues of $U$ and
$W$ are all the same. Then a double application of~(\ref{eq:Parseval})
gives 
\[
\left\Vert W\right\Vert _{2}^{2}=\sum\left(\beta_{i}^{+}\right)^{2}+\sum\left(\beta_{i}^{-}\right)^{2}=\sum\left(\lambda_{i}^{+}\right)^{2}+\sum\left(\lambda_{i}^{-}\right)^{2}=\left\Vert U\right\Vert _{2}^{2}.
\]
But this is a contradiction with Corollary~\ref{cor:L2CDIP}. This
finishes the proof. 
\end{proof}

\subsection{Homomorphism densities\label{subsec:SubgraphDensities}}

In this section, we address the following problem.
\begin{problem}
\label{prob:CharacterizeGraphs}Characterize graphs $H$ for which
$t(H,\cdot):\GRAPHONSPACE\rightarrow\mathbb{R}$ is a cut distance
compatible (respectively a cut distance identifying) graphon parameter.
\end{problem}

Observe that thanks to Proposition~\ref{prop:CDIPstep}, for the
case of compatible graphon parameters, Problem~\ref{prob:CharacterizeGraphs}
reduces to characterizing graphs $H$ for which we have
\begin{align}
t\left(H,W^{\Join\mathcal{P}}\right) & \le t\left(H,W\right)\text{ for each \ensuremath{W\in\GRAPHONSPACE} and each finite partition \ensuremath{\mathcal{P}}}.\label{eq:B1}
\end{align}

Similarly, if true, our Conjecture~\ref{conj:identifying_characterization}
implies that for the case of identifying graphon parameters, Problem~\ref{prob:CharacterizeGraphs}
reduces to characterizing graphs $H$ for which we have

\begin{align}
t\left(H,W^{\Join\mathcal{P}}\right) & <t\left(H,W\right)\text{ for each \ensuremath{W\in\GRAPHONSPACE} and each finite partition \ensuremath{\mathcal{P}} for which \ensuremath{W\neq W^{\Join\mathcal{P}}}}.\label{eq:B2}
\end{align}
This is closely related to Sidorenko's conjecture (which was asked
independently by Simonovits, and by Sidorenko, \cite{MR776819,Sidorenko1993})
and the Forcing conjecture (first hinted in \cite[Section 5]{MR2080111}).
Indeed, these conjectures \textemdash{} when stated in the language
of graphons \textemdash{} ask to characterize graphs $H$ for which
we have
\begin{align}
t\left(H,W^{\Join\left\{ \Omega\right\} }\right) & \le t\left(H,W\right)\text{ for each \ensuremath{W\in\GRAPHONSPACE}}\label{eq:A1}\\
 & \text{(Sidorenko's conjecture), and}\nonumber \\
t\left(H,W^{\Join\left\{ \Omega\right\} }\right) & <t\left(H,W\right)\text{ for each nonconstant \ensuremath{W\in\GRAPHONSPACE}}\label{eq:A2}\\
 & \text{(Forcing conjecture)}.\nonumber 
\end{align}

Recall that Sidorenko's conjecture asserts that $H$ satisfies~(\ref{eq:A1})
if and only if $H$ is bipartite. Similarly, the Forcing conjecture
asserts that $H$ satisfies~(\ref{eq:A2}) if and only if $H$ is
bipartite and contains a cycle. In both cases, the $\Rightarrow$
direction is easy. Let us recall that the reason why at least one
cycle is required for the Forcing conjecture is that the homomorphism
density of any forest $H$ in any $p$-regular graphon (whether constant-$p$,
or not) is $p^{e(H)}$. The other direction in both conjectures is
open, despite being known in many special cases, see~\cite{MR3667583,Lov:Sidorenko,MR3456171,Conlon2010,Hat:Siderenko,Li,SzegedyEntropySidorenko,ConLee:Sidorenko,CoKiLe:AdvancesOnSidorenko}.

Because all the properties we investigate in this section strengthen~(\ref{eq:A1}),
we are concerned only with bipartite graphs throughout. The only exception
is Remark~\ref{rem:reversedcutdistsubgraph} which addresses a possible
<<converse>> definition of cut distance identifying properties.

Graphs satisfying~(\ref{eq:B1}) were investigated in~\cite{KMPW:StepSidorenko}
where these graphs are said to have the\emph{ step Sidorenko property}.
Similarly, graphs satisfying~(\ref{eq:B2}) are said to have the
\emph{step forcing property}. Clearly, these properties imply~(\ref{eq:A1})
and~(\ref{eq:A2}), respectively. These stronger <<step>> properties
do not follow automatically from~(\ref{eq:A1}) and~(\ref{eq:A2});
in~\cite[Section 2]{KMPW:StepSidorenko} it is shown that the 4-cycle
with a pendant edge $C_{4}^{+}$ has the Sidorenko property but not
the step Sidorenko property. Thus, every graph having the step Sidorenko
property must be bipartite and every graph having the step forcing
property must be bipartite with a cycle. The focus of~\cite{KMPW:StepSidorenko}
was in providing negative examples. For example, it was shown in~\cite{KMPW:StepSidorenko}
that a Cartesian product of cycles does not have the step Sidorenko
property, unless all the cycles have length~4.

The connection to our running Problem~\ref{prob:CharacterizeGraphs}
comes from Proposition~14.13 of~\cite{Lovasz2012} which implies
that each weakly norming graph has the step Sidorenko property.
\begin{cor}
\label{cor:weakly_norming->compatible}For each weakly norming graph
$H$ the function $t(H,\cdot)$ is cut distance compatible (or, equivalently,
$H$ has the step Sidorenko property).
\end{cor}

Corollary~\ref{cor:weakly_norming->compatible} also directly follows
from Theorem~\ref{thm:convex_functions_are_compatible}. We recall
the proof from~\cite{Lovasz2012} in Section~\ref{subsec:ProofOfCorWeaklyNormingStep}.

In Section~\ref{subsec:Proof-of-compatibility} we prove Theorem~\ref{thm:compatibility->norming}
which states that among connected graphs, the graphs with the step
Sidorenko property are exactly the weakly norming graphs (thus answering
a question of Krá\v{l}, Martins, Pach and Wrochna~\cite[Section 5]{KMPW:StepSidorenko}).
\begin{thm}
\label{thm:compatibility->norming}Suppose that $H$ is a connected
graph. If the function $t(H,\cdot)$ is cut distance compatible (or,
equivalently, if $H$ has the step Sidorenko property), then $H$
is weakly Hölder.
\end{thm}

\begin{rem}
\label{rem:disconnectedcharacterization}For disconnected graphs,
the statement of Corollary~\ref{cor:weakly_norming->compatible}\emph{
}actually does not require the graph to be weakly norming, and can
be strengthened as follows. \emph{If each component $H_{i}$ of a
graph $H=H_{1}\sqcup\ldots\sqcup H_{k}$ is weakly norming then $t(H,\cdot)$
is cut distance compatible}. Indeed, suppose that $U\preceq W$. Then
Corollary~\ref{cor:weakly_norming->compatible} tells us that for
each component, $t(H_{i},U)\le t(H_{i},W)$. Thus, $t(H,U)=\prod t(H_{i},U)\le\prod t(H_{i},W)=t(H,W)$,
as was needed. We remark, that this relation between weakly norming
disconnected graphs and cut distance compatibility might perhaps be
an equivalence.
\end{rem}

\begin{rem}
\label{rem:stepSidorenkoAreQuiteRestricted}Two nontrivial necessary
conditions for a graph $H$ to be weakly Hölder are established in~\cite[Theorem 2.10]{Hat:Siderenko}.
One of them basically says that $H$ does not contain a subgraph denser
than itself. The other condition says that if $V(H)=A_{1}\sqcup A_{2}$
is a bipartition of $H$ and $u,v\in A_{i}$ are two vertices from
the same part, then $\deg(u)=\deg(v)$. Thus, Theorem~\ref{thm:compatibility->norming}
restricts quite substantially the class of graphs having the step
Sidorenko property, compared to the class of all bipartite graphs
which are conjectured to have the Sidorenko property. In particular,
we see directly that $C_{4}^{+}$ does not have the step Sidorenko
property.
\end{rem}

The next theorem, which we prove in Section~\ref{subsec:ProofTheoremNormingStep},
is our another main result.
\begin{thm}
\label{thm:norming->step}Suppose that $H$ is a norming graph. Then
the parameter $t(H,\cdot)$ is cut distance identifying. In particular,
by the trivial direction of Conjecture~\ref{conj:identifying_characterization},
$H$ is step forcing.
\end{thm}

Note that Theorem~\ref{thm:norming->step} is an implication only.
It is reasonable to ask about the converse (for connected graphs,
for the same reasons as in Remark~\ref{rem:disconnectedcharacterization}).
\begin{problem}
Is it true that if a connected graph $H$ has the step forcing property
(or if $t(H,\cdot)$ is cut distance identifying, which may be a more
restrictive assumption), we also have that $H$ is norming?
\end{problem}

Before showing the proofs of Theorem~\ref{thm:compatibility->norming}
and Theorem~\ref{thm:norming->step}, we summarize in Figure~\ref{fig:weakly_norming_diagram}
the known and conjectured relations for weakly norming graphs, norming
graphs, graphs with the step Sidorenko or the step forcing property,
and graphs that give cut distance compatible or cut distance identifying
parameters.

\begin{figure}
\includegraphics[width=0.98\textwidth]{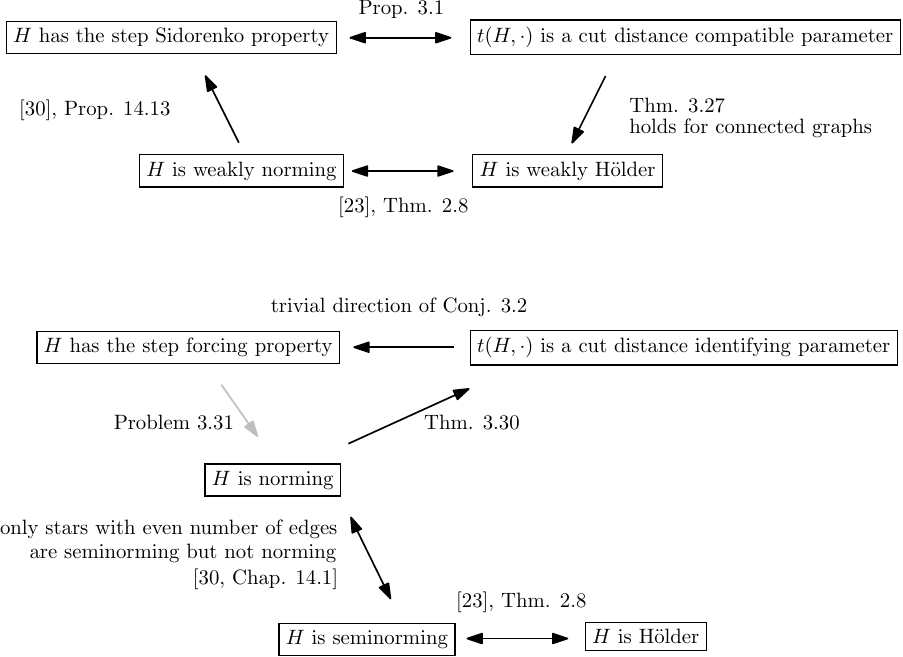}

\caption{Diagram of notions regarding homomorphism densities used in this paper
and their relations.}
\label{fig:weakly_norming_diagram}
\end{figure}

\subsubsection{\label{subsec:ProofOfCorWeaklyNormingStep}Proof of Corollary~\ref{cor:weakly_norming->compatible}}

Here, we prove that each weakly norming graph $H$ has the step Sidorenko
property. Our argument is a tailored version of the proof of Proposition 14.13
of~\cite{Lovasz2012} (where the statement is proven in bigger generality,
for so-called smooth invariant norms). The reason why we recall this
argument is that it will allow us to understand the strategy for proving
Theorem~\ref{thm:norming->step}, as we explain at the end of this
section.

So, suppose that $H$ is a weakly norming graph, $W:\Omega^{2}\rightarrow[0,1]$
is a graphon, $\mathcal{P}$ is a finite partition of $\Omega$. We
need to prove that $t\left(H,W\right)\ge t\left(H,W^{\Join\mathcal{P}}\right)$.
Without loss of generality, we can assume that $\Omega=[0,1)$ (see
Remark~\ref{rem:OtherGroundSpaces}), and that $\mathcal{P}$ is
a partition into intervals $\left\{ [a_{\ell},a_{\ell+1})\right\} _{\ell=1}^{|\mathcal{P}|}$.
Fix a number $\eta$ which is irrational with respect to the lengths
of all the intervals $[a_{\ell},a_{\ell+1})$. Consider the map $\gamma:[0,1)\rightarrow[0,1)$
that maps each number $x\in[0,1)$, say $x\in[a_{i},a_{i+1})$, to
$\left((x-a_{i}+\eta)\right)\mod(a_{i+1}-a_{i})+a_{i}$. Clearly,
the map $\gamma$ is a measure preserving bijection on $[0,1)$, where
each interval $[a_{\ell},a_{\ell+1})$ is $\gamma$-invariant, and
$\gamma$ restricted to each $[a_{\ell},a_{\ell+1})$ is ergodic.
It follows that the map $\left(x,y\right)\mapsto\left(\gamma(x),\gamma(y)\right)$
is ergodic when restricted on each set of the form $[a_{\ell},a_{\ell+1})\times[a_{k},a_{k+1})$.

For $n\in\mathbb{N}$, let $U_{n}$ be the version of $W$ obtained
using the $n$-th iteration of $\gamma$, $U_{n}:=W^{\gamma^{n}}$.
For $n\in\mathbb{N}$, let $S_{n}:=\frac{1}{n}\sum_{k=1}^{n}U_{k}$.
The Pointwise Ergodic Theorem tells us that the graphons $S_{n}$
converge pointwise to $W^{\Join\mathcal{P}}$. Hence,
\begin{align*}
\sqrt[e(H)]{t\left(H,W^{\Join\mathcal{P}}\right)} & =\lim_{n\rightarrow\infty}\sqrt[e(H)]{t\left(H,S_{n}\right)}=\lim_{n\rightarrow\infty}\sqrt[e(H)]{t\left(H,\frac{1}{n}\sum_{k=1}^{n}U_{k}\right)}\\
\JUSTIFY{\text{\ensuremath{\sqrt[e(H)]{t\left(H,|\cdot|\right)}}}\,is\,a\,seminorm} & \le\lim_{n\rightarrow\infty}\frac{1}{n}\cdot\sum_{k=1}^{n}\sqrt[e(H)]{t\left(H,U_{k}\right)}\\
\JUSTIFY{\text{\ensuremath{t\left(H,U_{k}\right)=t\left(H,W\right)}}} & =\sqrt[e(H)]{t\left(H,W\right)}\;,
\end{align*}
as was needed. This finishes the proof of Corollary~\ref{cor:weakly_norming->compatible}.

\medskip{}

Let us now look back at the argument to see what needs to be strengthened
to give Theorem~\ref{thm:norming->step}. Actually, in this informal
sketch, we only want to show that $H$ is step forcing, rather than
$t(H,\cdot)$ being cut distance identifying.

That is, we have a norming graph $H$, a graphon $W:\Omega^{2}\rightarrow[0,1]$,
a finite partition $\mathcal{P}$ of $\Omega$, such that $W\neq W^{\Join\mathcal{P}}$.
We need to prove that $t\left(H,W\right)>t\left(H,W^{\Join\mathcal{P}}\right)$.
The only space for getting the needed strict inequality in the calculation
above is in the triangle inequality on the second line. In view of
Remark~\ref{rem:whataremoduliofconvexity}, and using the fact that
$H$ is norming and hence $\left\Vert \cdot\right\Vert _{H}$ uniformly
convex by Theorem~\ref{thm:hatami_moduli_of_convexity}, it only
remains to argue that many of the graphons $U_{k}$ are far from colinear,
where <<far from colinear>> is measured in the $\left\Vert \cdot\right\Vert _{H}$-norm.\footnote{Note that we need to be careful about quantification: for example
having just one graphon $U_{k}$ to be <<somewhat far>> from the
others, could result in a strict triangle inequality that would disappear
when taking $\lim_{n\rightarrow\infty}$. So, we really need <<many>>
graphons that are <<uniformly far from colinear>>.} This is indeed plausible: since $W\neq W^{\Join\mathcal{P}}$, the
graphons $U_{k}$ must indeed be different.

As we shall see in Section~\ref{subsec:ProofTheoremNormingStep},
there are several difficulties in the actual proof of Theorem~\ref{thm:norming->step}.
In particular, we were not able to show that this approach works with
$U_{k}=W^{\gamma^{k}}$, and needed to choose the approximating graphons
$U_{k}$ with the help of (rather technical) Lemma~\ref{lem:L1approxByVersions}.

\subsubsection{\label{subsec:Proof-of-compatibility}Proof of Theorem~\ref{thm:compatibility->norming}}

Let $H$ be a connected graph such that $t(H,\cdot)$ is cut distance
compatible. Suppose that $H$ has $m$ edges and $n$ vertices. We
prove that $H$ is weakly Hölder. By Theorem~\ref{thm:norming_iff_holder}
we already know that weakly Hölder graphs are exactly weakly norming
graphs. We divide the proof of the theorem into two parts. At first
we prove that $t(H,\cdot)$ is subadditive up to a constant loss,
specifically, we show that
\begin{align}
t(H,U)^{1/m}+t(H,V)^{1/m} & \ge\frac{1}{4}\cdot t(H,U+V)^{1/m}\;.\label{eq:homdensity_subadditivity}
\end{align}
Then we use this inequality to prove that $H$ is weakly norming using
the tensoring technique in the same way as it is used in the proof
of Theorem~2.8 from~\cite{Hat:Siderenko}. 

Let $U$ and $V$ be two arbitrary graphons and let $W_{1}$ be a
graphon containing a copy of $U$ scaled by the factor of one half
in its top-left corner (i.e., $W_{1}\left(x,y\right)=U\left(2x,2y\right)$
for $(x,y)\in[0,\frac{1}{2}]^{2}$), a copy of $V$ in its bottom-right
corner (i.e., $W_{1}\left(x,y\right)=V\left(2(x-\frac{1}{2}),2(y-\frac{1}{2})\right)$
for $(x,y)\in[\frac{1}{2},1]^{2}$), and zero otherwise (see Figure~\ref{fig:averaging}).
Note that for the homomorphism density $t(H,W_{1})$ we have 
\begin{align*}
t(H,W_{1}) & =\frac{t(H,U)+t(H,V)}{2^{n}}\;.
\end{align*}
\begin{figure}
\includegraphics[scale=0.65]{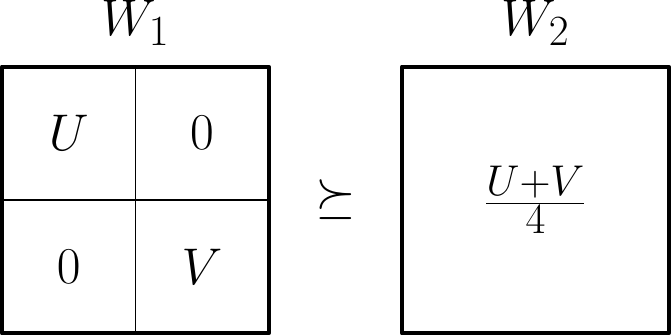}

\caption{Graphons $U$, $V$, $W_{1}$ and $W_{2}$ from the proof of Theorem~\ref{thm:compatibility->norming}}
\label{fig:averaging}
\end{figure}
This is because $H$ is connected and, thus, homomorphisms that map
a positive number of vertices of $H$ to $[0,\frac{1}{2}]$, and a
positive number of vertices to $[\frac{1}{2},1]$ do not contribute
to the value of the integral $t(H,W_{1})$. Now consider the graphon
$W_{2}=\frac{U+V}{4}$. By \cite[Lemma 4.2]{DGHRR:ACCLIM} we have
$W_{1}\succeq W_{2}$. It follows that $t(H,W_{1})\ge t(H,W_{2})$.
Observe that $t(H,W_{2})=t\left(H,\frac{U+V}{4}\right)=\frac{t(H,U+V)}{4^{m}}$,
hence we get 
\begin{align*}
\frac{t(H,U)+t(H,V)}{2^{n}} & \ge\frac{t(H,U+V)}{4^{m}}\;.
\end{align*}
We are actually interested in the quantity $t(H,U)^{1/m}$, so we
rewrite this as 
\begin{align*}
\left(t(H,U)+t(H,V)\right)^{1/m} & \ge\frac{2^{n/m}}{4}\cdot t(H,U+V)^{1/m}\ge\frac{1}{4}\cdot t(H,U+V)^{1/m}\;.
\end{align*}
Finally note that $t(H,U)^{1/m}+t(H,V)^{1/m}\ge\left(t(H,U)+t(H,V)\right)^{1/m}$,
as can be verified by raising the inequality to the $m$-th power.
This yields the desired inequality~(\ref{eq:homdensity_subadditivity}).

Now we need to improve the constant on the right-hand side of~(\ref{eq:homdensity_subadditivity}).
To this end we use the tensor power trick in the same way as was used
in~\cite[Theorem 2.8]{Hat:Siderenko} and in~\cite[Theorem 14.1]{Lovasz2012}.

Note that the inequality~(\ref{eq:homdensity_subadditivity}) can
be inductively generalised to yield that for a sequence of graphons
$U_{1},\dots,U_{\ell}$ we have 
\begin{align}
\sum_{i=1}^{\ell}t(H,U_{i})^{1/m} & \ge\left(\frac{1}{4}\right)^{\ell-1}\cdot t\left(H,\sum_{i=1}^{\ell}U_{i}\right)^{1/m}.\label{eq:betterboundforholder}
\end{align}
Let $(H,w)$ be a $\POSITIVEKERNELSPACE$-decoration of $H$. By Lemma~\ref{lem:HomogeneousWH}
we may assume that $t(H,W_{e})=1$ for every $W_{e}$. We want to
prove that $t(H,w)\le1$, but at first we prove the weaker inequality
$t(H,w)\le4^{m(m-1)}\cdot m^{m}$. We have
\begin{align}
t(H,w) & \le t\left(H,\sum_{e\in E(H)}W_{e}\right)\le\left(4^{m-1}\cdot\sum_{e\in E(H)}t(H,W_{e})^{1/m}\right)^{m}=4^{m(m-1)}\cdot m^{m}\;,\label{eq:almost_holder}
\end{align}
where in the first inequality we replaced each $W_{e}$ by $\sum_{e\in E(H)}W_{e}$,
while the second inequality is due to the bound~(\ref{eq:betterboundforholder}).
Now suppose that we decorate each edge $e$ of $H$ by $W_{e}^{\otimes k}$
for $k\ge1$. As we observed in~(\ref{eq:tensorgeneraliseddensities})
and~(\ref{eq:tensorproducteq}), we then have $t\left(H,w^{\otimes k}\right)=t(H,w)^{k}$
and $t\left(H,W_{e}^{\otimes k}\right)=t\left(H,W_{e}\right)^{k}=1$.
Thus inequality (\ref{eq:almost_holder}) gives that $t(H,w)^{k}=t(H,w^{\otimes k})\le4^{m(m-1)}\cdot m^{m}$,
thus $t(H,w)\le\left(4^{m(m-1)}\cdot m^{m}\right)^{1/k}$. Since this
holds for any $k\ge1$, we conclude that $t(H,w)\le1$.

\subsubsection{\label{subsec:ProofTheoremNormingStep}Proof of Theorem~\ref{thm:norming->step}}

Let $H$ be a norming graph. By Fact~\ref{fact:disconnGraphNorms},
it is enough to consider the case that $H$ is connected. Suppose
that $H$ has $m$ edges. Since $m>1$ (cf.~Fact~\ref{fact:basicseminorming}\ref{enu:NoTreeNorming}),
we know from Section~\ref{subsec:ModuliOfConvexityGraphs} that the
modulus of convexity $\mathfrak{d}_{H}$ of the norm~$\left\Vert \cdot\right\Vert _{H}$
defined by~(\ref{eq:defineseminorm}) is strictly positive. 

Suppose that $U,V\in\POSITIVEKERNELSPACE$ are such that $V\prec U$.
We want to prove that $t(H,U)>t(H,V)$. Clearly, we may assume that
$t(H,V)>0$. By rescaling, we may moreover assume that $t(H,U)=1$.
Let us define $\delta:=\frac{\delta_{\square}(U,V)}{32}$. We now
set
\[
\xi:=\min\big(\mathfrak{d}_{H}\left(\delta\right),\|V\|_{H}\big)\quad\text{and}\quad\varepsilon:=\frac{1}{4m\cdot\left\Vert U\right\Vert _{\infty}^{m-1}}\cdot\left(\frac{1}{10}\cdot\xi\right)^{m}\;.
\]
Let an even number $N$ and measure preserving bijections $\left(\phi_{i}\right)_{i=1}^{N}$
be given by Lemma~\ref{lem:L1approxByVersions} (see also Remark\ \ref{rem:RescalingL1approxByVersions})
for the inputs $U$ and $V$. Since $V\preceq U$, we have $\left\Vert V\right\Vert _{\infty}\le\left\Vert U\right\Vert _{\infty}.$
Since $\left\Vert V-\frac{\sum_{i=1}^{N}U^{\phi_{i}}}{N}\right\Vert _{1}<\varepsilon$,
we have
\begin{align}
\left|t\left(H,\frac{\sum_{i=1}^{N}U^{\phi_{i}}}{N}\right)-t(H,V)\right| & \overset{L\ref{lem:countinglemma}}{\le}4m\cdot\left\Vert U\right\Vert _{\infty}^{m-1}\cdot\cutDIST\left(\frac{\sum_{i=1}^{N}U^{\phi_{i}}}{N},V\right)\label{eq:densitiesclose}\\
 & \le4m\cdot\left\Vert U\right\Vert _{\infty}^{m-1}\cdot\left\Vert V-\frac{\sum_{i=1}^{N}U^{\phi_{i}}}{N}\right\Vert _{1}<\left(\frac{\xi}{10}\right)^{m}\;.\nonumber 
\end{align}

Observe that for each index $i\in\left\{ 1,2,\ldots,\frac{N}{2}\right\} $
that satisfies~(\ref{eq:farapartmycka}), an application of Proposition~\ref{prop:feelnorming}
to the kernel $X:=U^{\phi_{2i-1}}-U^{\phi_{2i}}$ gives that 
\[
\|U^{\phi_{2i-1}}-U^{\phi_{2i}}\|_{H}\ge\|U^{\phi_{2i-1}}-U^{\phi_{2i}}\|_{\square}\ge\delta\;.
\]

Since $\left\Vert U\right\Vert _{H}=1$, any two versions of $U$
that are far apart can play the role of $x$ and $y$ in~(\ref{eq:defmodulusconvexity}).
Thus, if we have an index $i\in\left\{ 1,2,\ldots,\frac{N}{2}\right\} $
that satisfies~(\ref{eq:farapartmycka}), then we get that
\[
1-\sqrt[m]{t\left(H,\frac{U^{\phi_{2i-1}}+U^{\phi_{2i}}}{2}\right)}\ge\mathfrak{d}_{H}\left(\delta\right)\ge\xi.
\]
Since $t(H,U)=1$, we may equivalently write
\begin{align}
\sqrt[m]{t(H,U^{\phi_{2i-1}})}+\sqrt[m]{t(H,U^{\phi_{2i}})} & \ge\sqrt[m]{t(H,U^{\phi_{2i-1}}+U^{\phi_{2i}})}+2\xi\;.\label{macenko_macenko}
\end{align}
We are now in a position to do the the final calculation. We have
\begin{equation}
\sqrt[m]{t(H,U)}=\frac{1}{N}\left(\sqrt[m]{t\left(H,U^{\phi_{1}}\right)}+\sqrt[m]{t\left(H,U^{\phi_{2}}\right)}+\dots+\sqrt[m]{t\left(H,U^{\phi_{N}}\right)}\right)\;.\label{eq:ann}
\end{equation}
Let us now group the summands on the right-hand side of~(\ref{eq:ann})
into $\frac{N}{2}$ pairs $\sqrt[m]{t\left(H,U^{\phi_{2i-1}}\right)}+\sqrt[m]{t\left(H,U^{\phi_{2i}}\right)}$.
Either a given pair satisfies~(\ref{macenko_macenko}), or, if not,
the subadditivity of $\left\Vert \cdot\right\Vert _{H}$ gives us
somewhat weaker $\sqrt[m]{t(H,U^{\phi_{2i-1}})}+\sqrt[m]{t(H,U^{\phi_{2i}})}\ge\sqrt[m]{t(H,U^{\phi_{2i-1}}+U^{\phi_{2i}})}$.
Recall that at least $\frac{N}{4}$ pairs satisfy~(\ref{macenko_macenko}).
Thus, 
\begin{align*}
\sqrt[m]{t(H,U)} & \ge\frac{1}{N}\left(\sqrt[m]{t\left(H,U^{\phi_{1}}+U^{\phi_{2}}\right)}+\dots+\sqrt[m]{t\left(H,U^{\phi_{N-1}}+U^{\phi_{N}}\right)}+\frac{N}{4}\cdot2\xi\right)\\
\JUSTIFY{subadditivity\,of\,\text{\ensuremath{\left\Vert \cdot\right\Vert _{H}}}} & \ge\frac{1}{N}\sqrt[m]{t\left(H,U^{\phi_{1}}+U^{\phi_{2}}+\dots+U^{\phi_{N}}\right)}+\frac{\xi}{2}\\
 & =\sqrt[m]{t\left(H,\frac{\sum_{i=1}^{N}U^{\phi_{i}}}{N}\right)}+\frac{\xi}{2}\\
\JUSTIFY{by\,\eqref{eq:densitiesclose};\;\text{\ensuremath{t(H,V)>(\frac{\xi}{10})^{m}}}} & \ge\sqrt[m]{t(H,V)-\big(\frac{\xi}{10}\big)^{m}}+\frac{\xi}{2}\\
 & \ge\sqrt[m]{t(H,V)}-\frac{\xi}{10}+\frac{\xi}{2}\\
 & >\sqrt[m]{t(H,V)}\;,
\end{align*}
as was needed.

\subsubsection{Discussion}
\begin{rem}
\label{rem:HowDoWeUseWeakStarInNormingThms}Let us comment on the
role of the abstract weak{*} approach we introduced in~\cite{DGHRR:ACCLIM}
in our proofs of Theorem~\ref{thm:compatibility->norming} and Theorem~\ref{thm:norming->step}.
\begin{itemize}
\item In Theorem~\ref{thm:compatibility->norming} we deduced the property
of being weakly Hölder by using inequalities of the form $t(H,W_{1})\ge t(H,W_{2})$,
where $W_{1}$ and $W_{2}$ are constructed from graphons $U$ and
$V$ as in Figure~\ref{fig:averaging}. These inequalities follow
from the fact that $W_{1}\succeq W_{2}$; but we do not have that
$W_{2}$ is a stepping of $W_{1}$. In other words, the immediate
property of $H$ we use is that $t(H,\cdot)$ is cut distance compatible,
rather than $H$ having the step Sidorenko property. (Of course, the
two properties are equivalent, by Proposition~\ref{prop:CDIPstep}.)
So, the weak{*} approach and the notion of the structuredness order
were instrumental here.
\item Theorem~\ref{thm:norming->step} cannot be even stated without the
notion of the structuredness order (until the validity of Conjecture~\ref{conj:identifying_characterization}
is confirmed). On the other hand, a version of Theorem~\ref{thm:norming->step}
which would not use the structuredness order, <<Suppose that $H$
is a norming graph. Then $H$ is step forcing.>> is either equivalent
(if Conjecture~\ref{conj:identifying_characterization} holds), or
only a tiny bit weaker. Our proof of Theorem~\ref{thm:norming->step}
could then be easily modified so that it avoids any notions introduced
in~\cite{DGHRR:ACCLIM}, that is using only traditional technology
available after Hatami~\cite{Hat:Siderenko}. Actually, in this setting
it would be enough to prove Lemma~\ref{lem:L1approxByVersions} for
$V=U^{\Join\mathcal{P}}$ rather than $V\prec U$, which would result
in a proof of that lemma shorter by one or two pages.
\end{itemize}
\end{rem}

\begin{rem}
\label{rem:reversedcutdistsubgraph}Note that the definition of cut
distance compatible (resp. identifying) parameters given at the beginning
of Section~\ref{subsec:BasicsOfCDIGP} was somewhat arbitrary. That
is, instead of requiring that $W_{1}\preceq W_{2}$ implies $\theta\left(W_{1}\right)\le\theta\left(W_{2}\right)$
(resp. that $W_{1}\prec W_{2}$ implies $\theta\left(W_{1}\right)<\theta\left(W_{2}\right)$),
we could have reversed the inequalities to $\theta\left(W_{1}\right)\ge\theta\left(W_{2}\right)$
(resp. $\theta\left(W_{1}\right)>\theta\left(W_{2}\right)$). However,
among graphon parameters induced by graph densities, there are only
trivial examples of cut distance compatible parameters in this sense.
These correspond to the graphs that are disjoint unions of cliques
on~1 and~2 vertices. For these graphs the homomorphism densities
are either always constant~1 (if the graph is a disjoint union of
vertices), or the power of the edge density of the graph (otherwise).
Since we know that $U\succeq V$ implies that the edge densities of
the two graphons are the same (Fact~\ref{fact:differentdensitiednotcomparable}),
these examples are cut distance compatible parameters in both senses
for a trivial reason, and, in particular, they are not cut distance
identifying parameters in this reverse sense. To see that there are
no other examples of cut distance compatible parameters in the reverse
sense, consider the two following graphons: a graphon $W_{\mathrm{clique}}$
consisting of a clique of measure $0.5$ ($W_{\mathrm{clique}}(x,y)=1$
if and only if $0\leq x,y\leq\frac{1}{2}$ and $W_{\mathrm{clique}}(x,y)=0$
otherwise), and the constant graphon $W_{\mathrm{const}}\equiv\frac{1}{4}$.
Now let $H$ be a graph that is not a disjoint union of cliques of
order one or two. Without loss of generality we may assume that $H$
does not contain any component consisting of a single vertex. Hence
$2e(H)>v(H)$. Now we have $W_{\mathrm{const}}\preceq W_{\mathrm{clique}}$,
but $t\left(H,W_{\mathrm{const}}\right)=\left(\frac{1}{4}\right)^{e(H)}<\left(\frac{1}{2}\right)^{v(H)}=t\left(H,W_{\mathrm{clique}}\right)$.
\end{rem}

\subsubsection{Local Sidorenko's conjecture}

An interesting weakening of Sidorenko's conjecture is to require~(\ref{eq:A1})
only for graphons $W$ that are close to a constant graphon. More
precisely, we say that a graph $H$ has the \emph{local Sidorenko
property with respect to the $L^{1}$-norm} (resp. \emph{with respect
to the cut norm} or \emph{with respect to the $L^{\infty}$-norm})
if for each $p\in\left[0,1\right]$ there exists an $\varepsilon>0$
such that for each graphon $W$ of edge density $p$ and with $\left\Vert W-p\right\Vert _{1}<\varepsilon$
(resp. with $\left\Vert W-p\right\Vert _{\square}<\varepsilon$ or
with $\left\Vert W-p\right\Vert _{\infty}<\varepsilon$) we have that
$t(H,W)\ge p^{e(H)}$. This weakening was first considered by Lovász~\cite{Lov:Sidorenko}
who proved that bipartite graphs are indeed locally Sidorenko (even
with respect to the cut norm, which is the strongest of the results).
Recently a full characterization of graphs with the local Sidorenko
was announced by Fox and Wei~\cite{EurocombFoxWei}: a graph is locally
Sidorenko if and only if it is a forest or has even girth.

We can combine the <<step>> and the <<local>> features in an obvious
way. We say that a graph $H$ has the \emph{local step Sidorenko property}
if for each partition $\mathcal{P}=\left(\Omega_{i}\right)_{i=1}^{k}$
of $\Omega$ and each template of densities $\left(p_{ij}\in[0,1]\right)_{i,j\in[k]}$
there exists $\varepsilon>0$ such that for each graphon $W$ for
which the average of $W$ on each $\Omega_{i}\times\Omega_{j}$ equals
$p_{ij}$, and for which $W^{\Join\mathcal{P}}$ is $\varepsilon$-close
to $W$ in some fixed norm as above, we have $t\left(H,W^{\Join\mathcal{P}}\right)\le t\left(H,W\right)$.
Locally step forcing graphs can be defined analogously.
\begin{problem}
Characterize locally step Sidorenko and locally step forcing graphs
(with respect to the norms $\left\Vert \cdot\right\Vert _{1}$, $\left\Vert \cdot\right\Vert _{\square}$,
or $\left\Vert \cdot\right\Vert _{\infty}$).
\end{problem}

\subsubsection{Two positive results directly}

We conclude the treatment of Problem~\ref{prob:CharacterizeGraphs}
by two positive results, namely that stars are step Sidorenko and
that even cycles are step forcing. Propositions \ref{prop:starsCUTdistIdent}
and \ref{prop:cyclesCUTdistIdent} in the case $\ell=2$ are not new
and follow from the results on weakly norming and Hölder graphs above.
Yet, the short proofs given here nicely employ other parts of the
theory established in this paper.
\begin{prop}
\label{prop:starsCUTdistIdent}For each $\ell\in\mathbb{N}$, the
graphon parameter $t\left(K_{1,\ell},\cdot\right):\mathcal{W}_{0}\rightarrow\mathbb{R}$
is cut distance compatible.
\end{prop}

\begin{proof}
The key is to observe that for a graphon $\Gamma$, we have $t\left(K_{1,\ell},\Gamma\right)=\int_{x\in\left[0,1\right]}x^{\ell}\mathrm{d}\boldsymbol{\Upsilon}_{\Gamma}$,
where $\boldsymbol{\Upsilon}_{\Gamma}$ is defined by~(\ref{eq:PushforwardDegrees}).
So, suppose that $U\preceq W$. By Proposition~\ref{prop:flatter},
we have that $\boldsymbol{\Upsilon}_{U}$ is at least as flat as $\boldsymbol{\Upsilon}_{W}$.
Let $\Lambda$ be the measure on $\left[0,1\right]^{2}$ as in Definition~\ref{def:flatter}
that witnesses this. In the following inequality, the measures $\Lambda_{x}^{1}$,
$x\in[0,1]$, are given by the disintegration of the measure $\Lambda$
on the first coordinate; see~\cite[p. 272]{DGHRR:ACCLIM} for a formal
definition. We have
\begin{align*}
t\left(K_{1,\ell},U\right) & =\int_{x\in\left[0,1\right]}x^{\ell}\mathrm{d}\boldsymbol{\Upsilon}_{U}\\
\JUSTIFY{by\,Lemma\;4.12\;from\;[12]} & =\int_{x\in[0,1]}\left(\int_{y\in[0,1]}y\:\mathrm{d}\Lambda_{x}^{1}\right)^{\ell}\:\mathrm{d}\boldsymbol{\Upsilon}_{U}\\
\JUSTIFY{Jensen's\,inequality} & \le\int_{x\in[0,1]}\int_{y\in[0,1]}y^{\ell}\:\mathrm{d}\Lambda_{x}^{1}\:\mathrm{d}\boldsymbol{\Upsilon}_{U}\\
\JUSTIFY{definition\,of\,disintegration} & =\int_{(x,y)\in[0,1]^{2}}y^{\ell}\:\mathrm{d}\Lambda\\
 & =\int_{y\in[0,1]}y^{\ell}\:\mathrm{d}\boldsymbol{\Upsilon}_{W}\\
 & =t\left(K_{1,\ell},W\right)\;.
\end{align*}
\end{proof}
\begin{prop}
\label{prop:cyclesCUTdistIdent}For each $\ell\in\left\{ 2,3,4,\ldots\right\} $,
the graphon parameter $t\left(C_{2\ell},\cdot\right):\mathcal{W}_{0}\rightarrow\mathbb{R}$
is cut distance identifying.
\end{prop}

Before giving a proof, let us note that Lemma~11 in~\cite{CoKrMa:FinitelyForcibleUniversal}
is equivalent to the case $\ell=2$ of the proposition. However, the
proof in \cite{CoKrMa:FinitelyForcibleUniversal} does not seem to
generalize to any higher $\ell$, in which case Proposition~\ref{prop:cyclesCUTdistIdent}
seems to be new.
\begin{proof}[Proof of Proposition~\ref{prop:cyclesCUTdistIdent}]
To prove the proposition, suppose that $\ell$ is fixed and $W_{1}\prec W_{2}$
are two graphons. Theorem~\ref{thm:spectrum} tells us that $W_{1}\SSO W_{2}$.
That is, the sum of the $\left(2\ell\right)$-th powers of eigenvalues
of $W_{1}$ is strictly smaller than that of $W_{2}$. The statement
now follows from Equation~(\ref{eq:EigenCycle}).
\end{proof}

\subsubsection{A recent result of Lee and Schülke\label{subsec:LeeSchulke}}

Combining Theorem~\ref{thm:convex_functions_are_compatible} and
Remark~\ref{rem:CDCimpliesConvexForDensities} (see also Figure~\ref{fig:weakly_norming_diagram}),
we get that for a connected graph $H$ the function $t(H,\cdot)$
is convex on $\GRAPHONSPACE$ if and only if $H$ is weakly norming.
A direct proof of this equivalence, together with a counterpart equivalence
of the convexity of $t(H,\cdot)$ on $\KERNELSPACE$ and $H$ being
norming was proven in~\cite{MR4282091}. These equivalences are then
used in~\cite{MR4282091} to argue that $K_{5,5}\setminus C_{10}$
(one of the smallest graphs where Sidorenko's conjecture is open)
is not weakly norming, and that $K_{t,t}$ minus a perfect matching
is not norming.

\section{Conclusion and possible further directions}

In this paper, we studied cut distance identifying and cut distance
compatible graphon parameters and graphon orders. This was based on
the structuredness order $\preceq$ introduced in~\cite{DGHRR:ACCLIM}.
The basic theory of the structurdness order and the key fact that
$\preceq$-maximal elements in the space of weak{*} limits are actually
cut-distance limits readily translate to some other combinatorial-analytic
objects such as kernels (that is, we allow even negative values),
or digraphons (limits of directed graphons, i.e., not necessarily
symmetric measurable functions $D:\Omega^{2}\rightarrow[0,1]$), and
so does the main feature of the pushforward measures expressed in
Proposition~\ref{prop:flatter}.\footnote{Even though these modifications have not been explicitly worked out
in~\cite{DGHRR:ACCLIM}.} We think it would be interesting to investigate cut distance identifying/compatible
parameters for these structures. For example, for digraph(on)s, there
is a reasonable theory of quasirandomness (see~\cite{MR3090716}
and references therein), and, as we saw, characterizing quasirandomness
is in a sense dual to characterizing $\preceq$-maximal elements in
the space of weak{*} limits.

The same program could be attempted for limits of $k$-uniform hypergraphs.
However, already the basic theory of the weak{*} approach to hypergraphons
seems to be substantially more involved (work in progress). Also,
note that the transition from graphon parameters to hypergraphon parameters
will not be automatic at all; for example, the Sidorenko conjecture
does not have a reasonable counterpart for hypergraphons (see~\cite{SzegedyEntropySidorenko}).

\section*{Acknowledgments}

We would like to thank David Conlon, Frederik Garbe, Dan Krá\v{l},
Dávid Kunszenti-Kovács, Joonkyung Lee. We also thank two anonymous
referees for their useful suggestions. This in particular concern
a substantial simplification related to Proposition~\ref{prop:feelnorming}.

\bibliographystyle{plain}
\bibliography{bibl}
\end{document}